\documentclass[a4paper, 11pt]{article}
\usepackage[english]{babel}
\usepackage[width=16cm, top=2.5cm, bottom=2cm]{geometry}

\usepackage{tikz}
\usepackage{gauPDE}
\usepackage{amssymb}
\usepackage{amsmath}
\usepackage{stmaryrd}
\usepackage{empheq}
\usepackage[mathcal]{euscript}
\usepackage{mathrsfs}
\usepackage{todonotes}

\usepackage[shortlabels]{enumitem}
\setlist[enumerate, 1]{label=(\roman*)}
\setlist[enumerate, 2]{label=(\alph*)}
\setlist[enumerate, 3]{label=(\arabic*)}

\usepackage{amsthm}
\theoremstyle{plain}
\newtheorem{theorem}{Theorem}[section]
\theoremstyle{definition}
\newtheorem{definition}[theorem]{Definition}
\theoremstyle{plain}
\newtheorem{lemma}[theorem]{Lemma}
\theoremstyle{plain}
\newtheorem{proposition}[theorem]{Proposition}
\theoremstyle{plain}
\newtheorem{corollary}[theorem]{Corollary}
\theoremstyle{definition}

\theoremstyle{remark}
\newtheorem{remark}[theorem]{Remark}
\theoremstyle{definition}
\newtheorem{assumption}[theorem]{Assumption}

\usepackage[backend=biber, style=alphabetic, giveninits, maxnames=6, maxalphanames=4]{biblatex}
\usepackage{csquotes}

\usepackage[colorlinks, linkcolor=blue, citecolor=blue!50!cyan]{hyperref}
\usepackage{cleveref}

\numberwithin{equation}{section}
\renewcommand{\theequation}{\arabic{section}.\arabic{equation}}
\SetSymbolFont{stmry}{bold}{U}{stmry}{m}{n} % Befehl für korrekten Umgang mit boldsymbols
\allowdisplaybreaks

\newcommand{\calH}{\mathcal{H}} % Hilbert space (state and "umbrella", primary, in (u,z))
\newcommand{\bbI}{\mathbb{I}} % Riesz isomorphism
\newcommand{\calE}{\mathcal{E}} % (total free) energy (primary; in (u,z))
\newcommand{\bbG}{\mathcal{G}} % Riemannian metric operator (primary; in (u,z))
\newcommand{\bbK}{\mathcal{K}} % Onsager operator (primary; in (u,z))
\newcommand{\calR}{\mathcal{R}} % Dissipation potential (primary; in (u,z))
\newcommand{\scrH}{\mathscr{H}} % Hilbert space (state and "umbrella", secundary, in (u,q))
\newcommand{\sfH}{{\mathsf{H}}} % Associated linear space for constant A, tau in (u,q)
\newcommand{\calF}{\mathscr{E}} % (total free) energy (secundary; in (u,q))
\newcommand{\scrR}{\mathscr{R}} % Dissipation potential (secundary; in (u,q))
\newcommand{\scrG}{\mathscr{G}} % Riemannian metric operator (secundary; in (u,q))
\newcommand{\subD}{\pa} % convex subdifferential
\newcommand{\subDF}{\pa^\mathrm{F}} % Fréchet subdifferential
\newcommand{\subDG}{\pa^\mathrm{G}} % Gateaux subdifferential
\newcommand{\subDL}{\pa^\ell} % Limit subdifferential
\newcommand{\Hdav}{{(\HH^1(\Omega))_\mathrm{av}^*}} % dual space of H^1 with average zero
\newcommand{\Havd}{{(\Hav^1(\Om))^*}} % dual space of H_av^1
\newcommand{\calP}{\mathcal{P}} % Power set
\newcommand{\AClocZ}{\ACloc^2([0,\infty);\calH_m)}

\newcommand{\bfu}{\boldsymbol{u}} % state variable (primary)
\newcommand{\bfv}{\boldsymbol{v}} % state variable (secundary)
\newcommand{\bfmu}{\boldsymbol{\mu}} % variable symbol for elements in H* (primary)
\newcommand{\olbfu}{\overline{\boldsymbol{u}}} % left piecewise constant interpolant
\newcommand{\olz}{\overline{z}}
\newcommand{\olu}{\overline{u}}
\newcommand{\whbfu}{\widehat{\boldsymbol{u}}} % piecewise affine interpolant / reparametrisation chain rule / time-rescaled variable
\newcommand{\whz}{\widehat{z}}
\newcommand{\whu}{\widehat{u}}

\newcommand{\ulbfu}{\underline{\boldsymbol{u}}} % right piecewise constant interpolant
\newcommand{\ulu}{\underline{u}}
\newcommand{\whbfmu}{\widehat{\bfmu}} % Reparametrisation chain rule
\newcommand{\whE}{\widehat{E}} % Reparametrisation chain rule

\newcommand{\Om}{\Omega} % Omega
\newcommand{\n}{\mathrm{n}} % outward normal unit
\newcommand{\fraka}{\mathfrak{a}} % symbol for average of something
\newcommand{\calECH}{\mathcal{E}_\mathrm{GL}} % Ginzburg-Landau free energy

\newcommand{\ve}{\varepsilon} % varepsilon
\newcommand{\vek}{{\ve_k}} % varepsilon with k index
\newcommand{\calRCH}{\calR_\mathrm{CH}}
\newcommand{\calRvCH}{\calR_\mathrm{vCH}}
\newcommand{\calRmAC}{\calR_\mathrm{mAC}}

\providecommand{\subjclass}[1]
{{\small\text{\textit{MSC 2020: }} #1}
}

\providecommand{\keywords}[1]
{{\small\text{\textit{Keywords and phrases: }} #1}
}

\title{Well-posedness and relaxation in a simplified model for viscoelastic phase separation via Hilbertian gradient flows
}
\author{Moritz Immanuel Gau\footnote{Weierstrass Institute for Applied Analysis and Stochastics, Mohrenstraße 39, 10117 Berlin, Germany, 
\\Email: \texttt{moritzimmanuel.gau@wias-berlin.de}; \texttt{katharina.hopf@wias-berlin.de}} \and Katharina Hopf$^*$}
\date{}

\begin{document}
\maketitle

\begin{abstract}
    This article is concerned with a gradient-flow approach to a 
    Cahn--Hilliard model for viscoelastic phase separation
    introduced by Zhou et al.\ (Phys.\ Rev.\ E, 2006) in its variant with constant mobility.
    By means of time-incremental minimisation and generalised contractivity estimates, we establish the global well-posedness 
    of the Cauchy problem for moderately regular initial data. For general finite-energy data we obtain the existence of gradient-flow solutions and a stability estimate of weak--strong type. 
    We further study the asymptotic behaviour for relaxation time and bulk modulus depending on a small parameter. 
    Depending on the scaling, we recover the Cahn--Hilliard, the mass-conserving Allen--Cahn or the viscous Cahn--Hilliard equation.
    A challenge in the well-posedness analysis is the failure of semiconvexity of the appropriate driving functional, which is caused by a phase-dependence of the bulk modulus.
\end{abstract}

\keywords{Cahn--Hilliard equation, viscoelastic phase separation, gradient systems, Cauchy problem, convexity, uniqueness, minimising movements, relaxation limit.
}

\subjclass{
35M33, %Initial-boundary value problems for mixed-type systems of PDEs
35G50, %Systems of nonlinear higher-order PDEs
35A15, %Variational methods applied to PDEs
47H05, %Monotone operators and generalizations
35B40, %Asymptotic behavior of solutions to PDEs
35Q92, %PDEs in connection with biology, chemistry and other natural sciences
76A10. %Viscoelastic fluids
}

\section{Introduction}

Phase separation in a homogeneous binary mixture can occur when the system is quenched to low temperature in a way that a state close to one of the pure phases becomes energetically favourable.
If the constituents of the mixture have different intrinsic time and length scales, the unmixing process can exhibit dynamical asymmetry, a phenomenon typical for {\em viscoelastic phase separation} (VPS) in polymer--solvent mixtures.
A first phenomenological continuum model for VPS in polymer solutions that adheres to thermodynamic principles was introduced by Zhou et al.~\cite{Zhou_2006}, building on earlier modelling work by Tanaka and co-workers. For further background on the modelling, we refer to~\cite{Tanaka_2022} and the references therein.

In this paper, we study a gradient-flow model for viscoelastic phase separation in polymer solutions, which corresponds to a version of the dissipative system in~\cite[Equation (9)]{Zhou_2006} with a constant mobility (normalised to 1).
Given a bounded domain $\Om\subset\R^d$, $1\le d\le 3$, the equations read as follows:
\begin{subequations}
	\makeatletter
	\def\@currentlabel{\textsf{VPS}}
	\makeatother
	\label{EQ:VPS_original}
	\renewcommand{\theequation}{\textsf{VPS}.\alph{equation}}
	\begin{align}\label{EQ:VPS_original_volume_equation}
		\left\{\begin{alignedat}{3}
			\dot u&=\divv\Bigparen{\nabla\Bigparen{\frac{\delta\calF}{\delta u}-A(u)q}}\quad&&\text{in }(0,\infty)\times\Om,\\
			\dot q&=-\frac{1}{\tau(u)}q-A(u)\divv\Bigparen{\nabla\Bigparen{\frac{\delta\calF}{\delta u}-A(u)q}}\quad&&\text{in }(0,\infty)\times\Om.
		\end{alignedat}\right.
	\end{align}
	Here, $\frac{\delta\calF}{\delta u}$ denotes the variational derivative with respect to $u$ of the total free energy
	\begin{equation}\label{EQ:energy_original}
		\calF(u,q)=\int_\Om\Bigparen{\frac{1}{2}\abs{\nabla u}^2+F(u)}\dd x+\frac{1}{2}\int_\Om q^2\dd x,
	\end{equation}
	where, for simplicity, fixed parameters are normalised.
	Denoting by $\n:\boundary{\Om}\to S^{d-1}$ the outward unit normal to $\boundary{\Om}$, system \eqref{EQ:VPS_original_volume_equation} is complemented by no-flux boundary conditions
	\begin{equation}
		\nabla\Bigparen{\frac{\delta\calF}{\delta u}-A(u)q}\cdot\n=\nabla u\cdot\n=0\quad\text{on }(0,\infty)\times\boundary{\Om}
	\end{equation}
	and initial values $u(0)=u^0$, $q(0)=q^0$.
\end{subequations}
The function $u:[0,\infty)\times\Om\to\R$ denotes the \textit{phase field variable}, modelling the concentration of the polymer molecules (i.e.\ $u\approx0$ indicates solvent-rich regions and $u\approx1$ polymer-rich regions). The other unknown is the bulk stress $qI_d$, where we seek for the scalar variable $q:[0,\infty)\times\Om\to\R$, hereafter referred to as the \textit{scalar bulk stress}. Moreover, $A$ models the bulk modulus, $\tau$ denotes the relaxation time associated with the bulk stress quantity $q$, and we let $F$ be a regular potential with $f:=F'$. A typical choice is the double well potential $F(u)=u^2(u-1)^2$.
Our particular interest lies in the case of phase-dependent $A$ and $\tau$, which can cause a dynamic asymmetry between solvent-rich and polymer-rich regions.
This is reflected in \eqref{EQ:VPS_original} through the concentration flux -- which in addition to the gradient of the chemical potential $\nabla(\frac{\delta\calF}{\delta u})$ has the extra term $\nabla(A(u)q)$ -- and through the relaxation rate $\frac{1}{\tau(u)}$ of the scalar bulk stress.
These simplified systems are able to reproduce key features of VPS related to dynamic asymmetry like \textit{volume shrinking} and \textit{phase inversion}, as confirmed by various numerical simulations~\cite{Zhou_2006, Ding_Liu_Liu_2020, Spiller_et_al_2021, Brunk_Egger_Habrich_Lukacova_2024}. 

This paper is devoted to a rigorous gradient-flow formulation 
of~\eqref{EQ:VPS_original} and the development of a well-posedness theory for the  associated Cauchy problem, as well as to the derivation of relaxation limits, assuming that $A=A_\ve$ and $\tau=\tau_\ve$ depend on a small parameter $0<\ve\ll1$.
In this context, let us mention the related \textit{viscous Cahn--Hilliard equation} (vCH), introduced by Novick-Cohen~\cite{Cohen_1988} to model polymeric systems with viscous effects and analysed by Elliott and Stuart~\cite{Elliott_Stuart_1996}.
Mathematically, vCH serves as a homotopy interpolating between the Cahn--Hilliard equation (CH) and the mass-conserving Allen--Cahn equation (mAC) proposed by Rubinstein--Sternberg~\cite{Rubinstein_Sternberg_1992}. 
The three models CH, vCH, and mAC will reappear in the relaxation limit.

\subsection{Related literature}\label{SUBSEC:literature}

We briefly review the literature on the analysis of related VPS models and the gradient-flow theory relevant to our study, and discuss our contribution in this context.

The existence of global-in-time weak solutions in 2D for a stress-diffusive version of the full fluid VPS system in~\cite{Zhou_2006} -- involving the incompressible Navier--Stokes equations and an evolution equation for the stress tensor -- was established by Brunk and Luk{\'a}{\v c}ov{\'a}-Medvid'ov{\'a} in~\cite{Brunk_Lukacova_PartII_2022} both in the case of regular  mobilities and smooth potentials as well as in the case of degenerate mobilities and logarithmic potentials by means of the Faedo--Galerkin technique. (See also~\cite{Brunk_Lukacova_PartI_2022} for a previous result.)
In the subsequent articles~\cite{Brunk_Lukacova_RelativeEnergyWeakStrongUniqueness_2023,Brunk_2023}, a relative entropy method was used to establish weak--strong stability estimates (for $d\in\{2,3\}$), which are of conditional type in the 3D case.
An error analysis via relative entropies of a structure-preserving discrete finite-element scheme for the simplified VPS model without hydrodynamic transport (still in the presence of bulk stress diffusion) was performed in~\cite{Brunk_Egger_Habrich_Lukacova_2024} under the assumption of a smooth solution to the continuous problem. 
We would also like to point out the work by Colli et al.~\cite{Colli_et_al_1999}, who considered a non-isothermal phase-field system with memory, where temperature accounts for an extra contribution in the concentration flux. Despite the different physical context, this model has a parabolic--ODE structure at the PDE level that is related to that of~\eqref{EQ:VPS_original} (cf.~\eqref{EQ:VPS_new_variables} below), even though confined to a phase-independent constant coupling.

Gradient-flow methods have proved a powerful tool in the analysis of Cahn--Hilliard equations and related systems. 
The most classical result concerns the global well-posedness of the Cauchy problem for semiconvex, lower semicontinuous driving functionals (or ``energies") in Hilbert spaces~(\cite{Brezis_OperateursMaximauxMonotones_1973},\cite{Komura_1967}), see also~\cite[Introduction]{Muratori_Savare_2020} for a short review, or~\cite{Ambrosio_Brue_Semola_LecturesOptimalTransport_2024}. 
In~\cite{Rossi_Savare_2006}, Rossi and Savar\'e relax the hypothesis of semiconvexity using, as a key ingredient, Hilbert space-valued Young measures: under a compactness condition on the energy, they develop an existence theory, via minimising movements, for a generalisation of the gradient-flow equation involving a closure of the subdifferential,  provided that this limiting subdifferential is either convex-valued or satisfies a suitable chain rule.
In general, this leads to a very weak solution concept, where uniqueness may not be expected. The dissipation mechanism in their approach is still the classical inner product of the underlying Hilbert space.
An extension to state-dependent dissipation in the context of quasi-linear non-degenerate Cahn-Hilliard--Allen-Cahn systems was carried out by Heida~\cite{Heida_2015}. Again, this work  focusses on the existence of generalised solutions.
For an extension of the variational methods in~\cite{Rossi_Savare_2006} to a Banach space setting and some refinements, we refer to~\cite{Mielke_Rossi_Savare_2013} and to the review in~\cite{Mielke_lectureNotes_2023} for further literature.
For some recent developments in the Hilbertian setting, see e.g.~\cite{HS_2024_Mullins-Sekerka} in the context of sharp interface evolutions.
Closer in spirit to PDE analysis, there is also an increasing body of literature using gradient-flow techniques for the construction of global weak solutions to Cahn--Hilliard type models both in the non-degenerate case, see e.g.~\cite{Rocca_Scala_2017, Garcke_Knopf_2020} for constant mobility, and in the case of a degenerate mobility~\cite{LMS_2012}.
Note that, in the degenerate case, the gradient structure is of metric type based on generalised Wasserstein distances.

While, at a purely formal level, the gradient-flow structure of versions of the simplified model~\cite[Equation (9)]{Zhou_2006} was exploited in~\cite{Hopf_2024}, to the best of the authors' knowledge, the present work is the first to develop a rigorous gradient-flow interpretation for such a system. At the PDE level, our results imply an (unconditional) well-posedness theory without stress diffusive modification in a setting of moderately low regularity.
With regard to methodology, a specific feature of the present problem lies in the simultaneous failure of compactness and semiconvexity of the energy, which renders the analysis non-standard.
However, observing that semiconvexity only fails to a limited extent, determined by the state-dependence of the bulk modulus, we are able to leverage ideas from~\cite[Section~4]{Ambrosio_Gigli_Savare_GradientFlows_2005} and~\cite{Brezis_OperateursMaximauxMonotones_1973}, extending them to a ``slightly non-semiconvex" situation. This is enabled by combining gradient-flow with PDE methods, which provide the required extra regularity. 
Outside the semiconvex case, a systematic application of the gradient-flow theory to the study of (unconditional) uniqueness properties and stability estimates does not appear to have been explored extensively in the literature. In a very broad sense, our approach may be seen as reminiscent to the Wasserstein gradient-flow problem~\cite{BMZ_2023}, where a singular cross-diffusive perturbation is introduced that destroys uniform displacement convexity of the energy, but where exponential convergence rates to equilibrium can still be recovered under a smallness condition. 
Let us mention that an analysis of the gradient structure for VPS in the case of a degenerate mobility will be part of a future work.
Our results on the relaxation limit furnish a rigorous link between~\eqref{EQ:VPS_original} and classical models for phase separation, the Cahn--Hilliard, the viscous Cahn--Hilliard, and the mass-preserving Allen--Cahn equation and extend the latter two to dynamically asymmetric versions. 

\subsection{Outline}
The manuscript is structured as follows.
Section~\ref{SEC:the_model} begins with a general introduction to gradient systems with a Hilbert space structure (Subsection~\ref{SUBSEC:basic_concepts}) and an outline of our strategy (Subsection~\ref{SUBSEC:strategy}). The remaining part (Subsection~\ref{SUBSEC:main_results}) is devoted to introducing our set-up, the general hypotheses as well as the main results. 
In Section~\ref{SEC:preliminaries}, we establish fundamental functional analytic results including estimates reminiscent to those in a semiconvex setting and prove a chain rule.
These results will be crucial in the subsequent analysis. Section~\ref{SEC:well_posedness} is devoted to the proofs of the main results regarding the well-posedness properties of \eqref{EQ:VPS_original}.
Finally, in Section~\ref{SEC:relaxation_limit} we prove the main results concerning the relaxation limit.

\section{Main results}\label{SEC:the_model}

\subsection{Basic concepts for Hilbertian gradient systems}\label{SUBSEC:basic_concepts}
We briefly review basic facts and notions, confining ourselves to the aspects most relevant to the present application. For further developments, we refer the interested reader to~\cite{Mielke_Rossi_Savare_2013, Peletier_lectureNotes_2014, Mielke_Montefusco_Peletier_2021, Mielke_lectureNotes_2023, Schmeller_Peschka_2023}.

\begin{definition}[Gradient system]\label{DEF:gradient_system}
    We call a triple $(\calH_m,\calE,\calR)$ a \textit{(Hilbertian) gradient system} if
    \begin{enumerate}
        \item $\calH_m$ is an affine space over a Hilbert space denoted by $H$,
        \item $\calE:\calH_m\to[0,+\infty]$ a functional with proper domain $\dom(\calE):=\set{\bfu\in\calH_m}{\calE(\bfu)<\infty}$,
        \item $\calR:\dom(\calE)\times H\to[0,\infty)$ can be represented as $\calR(\bfu;\bfv)=\frac{1}{2}\product{\bbG(\bfu)\bfv}{\bfv}_{H^*,H}$ for a family of linear, bounded, symmetric and positive definite operators $\bbG(\bfu):H\to H^*$, $\bfu\in\dom(\calE)$.
    \end{enumerate}
\end{definition}

We immediately see that $\bfv\mapsto\calR(\bfu;\bfv)$ is a positive bounded quadratic form with Fr\'echet differential $\D_{\bfv}\calR(\bfu;\bfv)=\bbG(\bfu)\bfv$ for $(\bfu,\bfv)\in\dom(\calE)\times H$.

Gradient systems are used to describe dissipative evolution equations. The affine space $\calH_m$ then represents the state space, $\calE$ is the driving functional, which in our present model will be the physically relevant free energy, and $\calR$ is called the dissipation potential, describing the dissipation mechanism of the system. The evolution equation associated to a gradient system $(\calH_m,\calE,\calR)$ is the \textit{gradient-flow equation}
\begin{align}\label{EQ:gfe_general_definition}
	-\bbG(\bfu(t))\dot\bfu(t)\in\subDG\calE(\bfu(t))\quad\text{for a.e. }t\in(0,\infty)\tag{\textsf{GFE}}
\end{align}
for absolutely continuous functions $\bfu:(0,\infty)\to\calH_m$, where $\subDG\calE$ denotes the Gateaux subdifferential
\begin{equation}\label{EQ:gateaux_subdifferential_definition}
	\subDG\calE(\bfu)=\biggset{\bfmu\in H^*}{\forall\bfv\in H:\liminf_{t\searrow0}\frac{\calE(\bfu+t\bfv)-\calE(\bfu)}{t}\geq\product{\bfmu}{\bfv}_{H^*,H}}.
\end{equation}
Note that, for $\bfu\in H$ fixed, $\subDG\calE(\bfu)\subset H^*$ is a closed and convex (possibly empty) subset. 
In classical applications, \eqref{EQ:gfe_general_definition} is often formulated using the Fr\'echet subdifferential
\begin{equation}\label{EQ:frechet_subdifferential_definition}
    \subDF\calE(\bfu)=\biggset{\bfmu\in H^*}{\liminf_{\bfv\to\bfu}\frac{\calE(\bfv)-\calE(\bfu)-\product{\bfmu}{\bfv-\bfu}_{H^*,H}}{\norm{\bfu-\bfv}_H}\geq0}.
\end{equation}
However, the analysis of the present manuscript relies on the more general version~\eqref{EQ:gateaux_subdifferential_definition}, which reduces to~\eqref{EQ:frechet_subdifferential_definition} if the state $\bfu$ enjoys some extra regularity of (cf.~Proposition~\ref{PRO:subdifferential} (ii)).

\begin{definition}[Gradient-flow solution]\label{DEF:gf_solution}
    Let $(\calH_m,\calE,\calR)$ be a gradient system. We call an absolutely continuous curve $\bfu:(0,\infty)\to\calH_m$ a \textit{gradient-flow solution}, in short \textit{GF solution}, if it satisfies \eqref{EQ:gfe_general_definition}.
\end{definition}

A desirable feature of gradient-flow solutions, which holds true in the smooth case, is the \textit{energy--dissipation balance} (EDB). A pair $(\bfu,\bfmu)\in\AClocZ\times\Lloc^2([0,\infty);H^*)$ is said to satisfy the EDB if
\begin{align}\label{EQ:edb_general_definition}
	\calE(\bfu(t))+\int_s^t\bigparen{\calR(\bfu(r);\dot\bfu(r))+\calR^*(\bfu(r);\bfmu(r))}\dd r=\calE(\bfu(s)),\quad0\leq s\leq t<\infty,\tag{\textsf{EDB}}
\end{align}
where $\calR^*(\bfu;\bfmu):=\sup_{\bfv\in H}\bigparen{\product{\bfmu}{\bfv}_{H^*,H}-\calR(\bfu;\bfv)}$ denotes the Legendre--Fenchel conjugate of $\calR$, and $\ACloc^2(I;\calH_m)$ for a general interval $I\subseteq[0,\infty)$ is defined as
\begin{align*}
    \ACloc^2(I;\calH_m):=\bigset{\bfu:I\to\calH_m\text{ locally absolutely continuous}}{\dot\bfu\in \Lloc^2(I;H)}.
\end{align*}
For instance, \eqref{EQ:edb_general_definition} may play the role of the second law of thermodynamics in isothermal systems, so as in our model \eqref{EQ:VPS_original}. From a technical point of view, \eqref{EQ:edb_general_definition} (or \eqref{EQ:edi_general_definition} below) is a practical identity to derive a priori estimates in certain scenarios, as in the proof of Theorems~\ref{TH:existence_extended_initials}~and~\ref{TH:relaxation_limit}.

To ensure \eqref{EQ:edb_general_definition}, one typically shows a chain rule of the following type (cf., e.g.~\cite[$(2.\mathrm{E}_4)$]{Mielke_Rossi_Savare_2013} or~\cite[Definition 3.3.1]{Mielke_2016}): If $\bfu\in\AClocZ$ and $\bfmu\in\Lloc^2([0,\infty);H^*)$ are such that $\bfmu(t)\in\subDG\calE(\bfu(t))$ for almost all $t\in(0,\infty)$, then $\calE\circ\bfu\in\AC([0,\infty);\R)$ and
\[\ddd{t}(\calE\circ\bfu)(t)=\product{\bfmu(t)}{\dot\bfu(t)}_{H^*,H}\quad\text{for a.e. }t\in(0,\infty).\]
Thus, if $\bfu$ additionally satisfies~\eqref{EQ:gfe_general_definition}, then such a chain rule can be applied to $\bfmu:=-\bbG(\bfu)\dot\bfu$ and then Legendre--Fenchel equivalences~\cite[Theorem 3.32]{Rindler_2018}
\begin{equation}\label{EQ:fenchel_equivalences}
    \product{\bbG(\bfu)\dot\bfu}{\dot\bfu}_{H^*,H}=\calR(\bfu;\dot\bfu)+\calR^*(\bfu;\bfmu)=\product{\bfmu}{\bbK(\bfu)\bfmu}_{H^*,H}
\end{equation}
imply \eqref{EQ:edb_general_definition}.

In the absence of smoothness and convexity, an appropriate chain rule may not always be established. However, in such cases, it is often still possible to prove the weaker \textit{energy-dissipation inequality} (EDI)
\begin{align}\label{EQ:edi_general_definition}
	\calE(\bfu(t))+\int_0^t\bigparen{\calR(\bfu(r);\dot\bfu(r))+\calR^*(\bfu(r);\bfmu(r))}\dd r\leq\calE(\bfu(0)),\quad t\in[0,\infty).\tag{\textsf{EDI}}
\end{align}
It is worth noting that, if a suitable chain rule holds after all, the formulations \eqref{EQ:gfe_general_definition}, \eqref{EQ:edb_general_definition} and \eqref{EQ:edi_general_definition} are in fact equivalent; see, e.g., \cite[Theorem 3.3.1]{Mielke_2016}. This is known as the \textit{energy-dissipation principle}, which is based on ideas from De Giorgi~\cite{DeGiorgi_1980, DeGiorgi_1993}.

\subsection{Our strategy}\label{SUBSEC:strategy}

\paragraph{Motivation}
If $A$ and $\tau$ are constant, we can apply the well-known theory of gradient flows of $\lambda$-convex ($\lambda\in\mathbb R$), lower semicontinuous functionals in Hilbert spaces mentioned in Subsection~\ref{SUBSEC:literature}.
This guarantees global-in-time existence and uniqueness of gradient-flow solutions, also in the sense of evolution variational inequalities (EVI), and the associated $\lambda$-contractive continuous semiflow extends to data in the closure of $\dom(\mathcal E)$.
To apply the theory to the present situation, consider
\begin{equation}\label{EQ:hilbert_space_definition_constant_case}
	\sfH:=\Bigset{(u,q)\in\Havd\times(\HH^1(\Om))^*}{Au+q\in\LL^2(\Om)}
\end{equation}
(see \eqref{EQ:averaged_spaces_definition} -- \eqref{EQ:dual_average_isoiso} for details to $\Hav^1(\Om)$ and $\Havd$) and endow $\sfH$ with the scalar product
\[
    \bigparen{(u_1,q_1),(u_2,q_2)}_\sfH:=\product{u_1}{(-\Delta)^{-1}u_2}_{\Havd,\Hav^1(\Om)}+\tau(Au_1+q_1,Au_2+q_2)_{\LL^2(\Om)}.  
 \]
Moreover, put $\scrH_m:=\Hdav+(m,0)$ for some $m\in\R$. Then, under suitable conditions on $F$ (for instance as in Assumption~\ref{ASS:general}), the energy $\calF$ defined in \eqref{EQ:energy_original} takes proper values on the dense subset $(\Hav^1(\Om)+m)\times\LL^2(\Om)\subset\scrH_m$ and is semiconvex with respect to $\norm{\cdot}_\sfH$.
With the (state-independent) dissipation potential $\scrR(\bfu;\bfv):=\frac{1}{2}\norm{\bfv}_\sfH^2=\frac{1}{2}\product{\bbI_\sfH\bfv}{\bfv}_{\sfH^*,\sfH}$, the associated gradient-flow equation of $(\scrH_m,\calF,\scrR)$ becomes $-\bbI_\sfH\dot\bfu\in\subDG\calF(\bfu)$, which, at the PDE level, coincides with~\eqref{EQ:VPS_original}. Existence and semi-contractivity of the gradient flow defined for initial values in the entire space $\scrH_m$ is then a consequence of the classical theory.

\paragraph{Change of variables in phase-dependent case}

For phase-dependent $A$ and $\tau$,  the problem still exhibits a formal gradient structure with driving functional $\calF$ and dissipation mechanism 
 \[\scrG(u)=\begin{pmatrix}
    (-\Delta)^{-1}+A^2(u)\tau(u)&A(u)\tau(u)
    \\A(u)\tau(u)&\tau(u)
\end{pmatrix}.\]
However, the question of determining an appropriate functional setting in the variables $(u,q)$ is more involved, since the regularity of the cross terms (corresponding to the off-diagonal terms in $\scrG(u)$) now depends on that of $u$.
To cope with such coupling, we consider the new variable $z:=q+K(u)$, where $K'=A,K(0)=0$, i.e.\ we introduce the change of variables $\Psi(u,z):=(u,z{-}K(u))$. This leads to the new gradient system $(\calH_m,\calE,\calR)$ with the pulled-back driving functional $\calE(\bfu)=\calF(\Psi(\bfu))$, $\bfu:=(u,z)$, and dissipation mechanism $\bbG(\bfu)=\D\Psi(\bfu)^T\scrG(u)\D\Psi(\bfu)$, which amounts to
\begin{align}\label{EQ:dissipation_mechanism_new_variables}
	\bbG(\bfu)=\begin{pmatrix}
		(-\Delta)^{-1}&0\\
        0&\tau(u)
    \end{pmatrix}.
\end{align}
Note that the dissipation mechanism in the new variables is diagonal and induces a simple functional setting $\Havd\times\LL^2(\Omega)$ whenever $\tau(u)\sim 1$.

In PDE form, the associated gradient-flow equation $\bbG(\bfu)\dot\bfu\in -\subDG\calE(\bfu)$ reads as
\begin{equation*}\tag{$\mathsf{VPS}'$}\label{EQ:VPS_new_variables}
	\left\{\begin{alignedat}{3}
		\PDElineIn{\dot u}{\divv\bigparen{\nabla\bigparen{-\Delta u+f(u)-A(u)(z-K(u))}}}{(0,\infty)\times\Om},\\
		\PDElineIn{\dot z}{-\frac{1}{\tau(u)}(z-K(u))}{(0,\infty)\times\Om},\\
		\PDElineOn{0}{\nabla u\cdot\n=\nabla\bigparen{-\Delta u+f(u)-A(u)(z-K(u))}\cdot\n}{(0,\infty)\times\boundary{\Om}}.
	\end{alignedat}
	\right.
\end{equation*}
Thus, instead of developing directly a solution theory for \eqref{EQ:VPS_original} in the variables $(u,q)$, we seek a pair $(u,z)$ satisfying system \eqref{EQ:VPS_new_variables}.
The challenge is that, in contrast to the constant-coefficient case mentioned above, we loose the semiconvexity of the energy.
However, as we will show in Section~\ref{SEC:preliminaries}, under extra regularity assumptions on $(u,z)$, we can derive inequalities that mimic the semiconvex case: crucially, we derive a subgradient estimate of the form
\begin{align*}
	\calE(\bfv)\ge \calE(\bfu)+\product{\delta\calE(\bfu)}{\bfv-\bfu}_{H^*,H}&+\lambda(\bfu)\norm{\bfv-\bfu}_H^2,
    \\&\mathbin{\phantom{+}}\lambda(u,z)=-C(1+\norm{A'}_\infty\norm{z-K(u)}_{\LL^\infty(\Om)})^2
\end{align*}
for $\bfu\in\dom(\subDG\calE)$, $\subDG\calE(\bfu)=\{\delta\calE(\bfu)\}$, and an analogous semi-monotonicity estimate for $\subDG\calE(\bfu)$.
To exploit such ``almost $\lambda$-convexity", we show that the flow propagates a certain mild regularity using the ODE in the second line in~\eqref{EQ:VPS_new_variables}.
The extra regularity will ensure that, for solution trajectories $\bfu=\bfu(t)$, the function $t\mapsto\lambda(\bfu(t))$ is locally integrable in $[0,\infty)$, which turns out to be sufficient for establishing key properties such as local Cauchy estimates at the approximate level as well as the chain rule.

\subsection{Function spaces, hypotheses, and main results}\label{SUBSEC:main_results}

\subsubsection{General notations and hypotheses}\label{SUBSUBSEC:hyp_not}
We now render precise the functional framework for the gradient system $(\calH_m,\calE,\calR)$ induced by the transformation $\Psi$. First, we define the averaged spaces
\begin{align}\label{EQ:averaged_spaces_definition}
	\begin{aligned}
		\Hav^1(\Om)&=\bigset{\phi\in\HH^1(\Om)}{\textstyle\int_\Om\phi\dd x=0},\\
		\Hdav&=\bigset{\phi\in(\HH^1(\Om))^*}{\product{\phi}{1}_{(\HH^1(\Om))^*,\HH^1(\Om)}=0},
	\end{aligned}
\end{align}
and equip them with
\begin{align}\label{EQ:averaged_spaces_inner_products}
    \begin{aligned}
        (\phi,\psi)_{\Hav^1(\Om)}&:=\int_\Om\nabla\phi\cdot\nabla\psi\dd x,\\
        (\phi,\psi)_\Hdav&:=\bigparen{(-\Delta)^{-1}(\phi\vert_{\Hav^1(\Om)}),(-\Delta)^{-1}(\psi\vert_{\Hav^1(\Om)})}_{\Hav^1(\Om)}.
    \end{aligned}
\end{align}
Here, $-\Delta:\Hav^1(\Om)\to\Havd$ denotes the Neumann Laplacian, defined as
\[\product{-\Delta\phi}{\psi}_{\Havd,\Hav^1(\Om)}:=\int_\Om\nabla\phi\cdot\nabla\psi\dd x=(\phi,\psi)_{\Hav^1(\Om)},\]
which in fact agrees with the Riesz isomorphism of $\Hav^1(\Om)$. Due to Poincar\'e-Wirtinger's inequality, \eqref{EQ:averaged_spaces_inner_products} define inner products on the respective spaces, equivalent to the ones induced by the standard inner products of $\HH^1(\Om)$ and $(\HH^1(\Om))^*$.

We then define the affine state space as $\calH_m:=(\Hdav+m)\times\LL^2(\Om)$ for $m\in\R$ fixed and identify the associated linear space $H=\Havd\times\LL^2(\Om)$. We justify this by the map
\begin{equation}\label{EQ:dual_average_isoiso}
    \Hdav\to\Havd,\qquad\phi\mapsto\phi\vert_{\Hav^1(\Om)},
\end{equation}
which is an isometric isomorphism by the above choice~\eqref{EQ:averaged_spaces_inner_products} of inner products. Then, $H^*=\Hav^1(\Om)\times\LL^2(\Om)$. We will use the `bold font' notation $\bfu=(u,z)$ and $\bfv=(v,y)$ for elements $(u,z)$, $(v,y)$ in $\calH_m$ or $H$, and $\bfmu=(\mu,\xi)$ for $(\mu,\xi)\in H^*$. Since $\calH_m$ is contained in the Hilbert space $\Havd\times\LL^2(\Om)$, we have the following closedness property for $\ACloc^2(I;\calH_m)$ for closed intervals $I\subseteq[0,\infty)$ due to \cite[Sections 9.1, 11.2]{Ambrosio_Brue_Semola_LecturesOptimalTransport_2024}: If $(\bfu_n)_{n\in\N}\in\ACloc^2(I;\calH_m)$ satisfies $\bfu_n(t)\to\bfu(t)$ in $\calH_m$ for any $t\in I$ and $\sup_{n\in\N}\norm{\dot\bfu_n}_{\LL^2(I;H)}<\infty$, then $\bfu\in\ACloc^2(I;\calH_m)$ and $\dot\bfu_n\weakto\dot\bfu$ in $\LL^2(I;H)$. We occasionally use the notation $\LL^p(0,T;\calH_m)$ for the space of functions $\bfu$ in $\LL^p(0,T;(\HH^1(\Om))^*\times\LL^2(\Om))$ with $\bfu(t)\in\calH_m$ for almost all $t$.\\

We define the driving functional $\calE:\calH_m\to[0,+\infty]$ via
\begin{equation}\label{EQ:banach_energy_functional}
	\calE(u,z)=
    \begin{cases}
        \int_\Om\bigparen{(\frac12\abs{\nabla u}^2+F(u))+\frac12(z-K(u))^2}\dd x
        &\text{if }(u,z)\in(\Hav^1(\Om){+}m)\times\LL^2(\Om)
        \\+\infty&\text{otherwise}.
    \end{cases}
\end{equation}
In Proposition~\ref{PRO:subdifferential}, we show that $\subDG\calE(\bfu)$  contains at most one element, namely
\begin{equation}\label{EQ:subdifferential_of_energy}
	\subDG\calE(u,z)=\Braces{\begin{pmatrix}
			-\Delta u+f(u)-A(u)(z-K(u))-\fraka(u,z)\\
			z-K(u)
		\end{pmatrix}}
\end{equation}
for regular enough $\bfu\in\calH_m$, where $\fraka(u,z)=\avint{_\Om}(f(u)-A(u)(z-K(u)))\dd x$.

According to \eqref{EQ:dissipation_mechanism_new_variables}, $\calR$ is induced by
\[\bbG(\bfu):H\to H^*,\qquad\bbG(\bfu)\bfv:=\bigparen{(-\Delta)^{-1}v,\tau(u)y}\]
for $\bfu=(u,z)\in\dom(\calE)$, which indeed is linear, bounded, symmetric and positive definite under assumption~\eqref{EQ:tau_lower_upper_bound} below. The inverse $\bbK(\bfu):=(\bbG(\bfu))^{-1}$ is given by
\[\bbK(\bfu):H^*\to H,\qquad\bbK(\bfu)\bfmu=\Bigparen{-\Delta\mu,\frac{1}{\tau(u)}\xi},\]
so that the Legendre--Fenchel conjugate of $\calR$ (cf.~\cite[Section 3.2]{Peletier_lectureNotes_2014}) can be expressed as $\calR^*(\bfu;\bfmu)=\frac12\product{\bfmu}{\bbK(\bfu)\bfmu}_{H^*,H}$ for $(\bfu,\bfmu)\in\dom(\calE)\times H^*$.

Throughout the rest of this manuscript, we impose the following hypotheses on the data:
\begin{assumption}\label{ASS:general}
	We let $\Om\subset\R^d$ ($d\in\{1,2,3\}$) be a bounded domain that is either convex or has a $\CC^{1,1}$-regular boundary.
    For the potential $F\in\CC^2(\R)$ and its derivative $f:=F'$, we assume $\inf f'\geq-\beta$ for some $\beta\in(0,\infty)$, which equivalently means that $F$ can be decomposed as
	\begin{equation}\label{EQ:decomposition_F}
		F(u)=h(u)-\frac{\beta}{2}u^2
	\end{equation}
	for a convex function $h\in\CC^2(\R)$.
    In the dimensions $d=2,3$, we additionally suppose the polynomial growth
	\begin{equation}\label{EQ:growth_condition_f}
		\abs{f(u)}\leq c_1(\abs{u}^p+1)
	\end{equation}
	for all $u\in\R$, some constant $c_1\in(0,\infty)$, and with exponent
	\[\begin{cases}
		p\in[1,\infty)&\text{for }d=2,\\
		p\in[1,5]&\text{for }d=3.
	\end{cases}\]
	Observe that this also implies
	\begin{equation}\label{EQ:growth_condition_F}
		\abs{F(u)}\leq c_2(\abs{u}^{p+1}+1)
	\end{equation}
	for any $u\in\R$ and some constant $c_2\in(0,\infty)$.
    Finally, we assume $A,\tau\in\WW^{1,\infty}(\R)$ (bounded and Lipschitz continuous) and
	\begin{subequations}
		\begin{align}
			A_*&\leq A(u)\leq A^*,\\
			\tau_*&\leq\tau(u)\leq\tau^*\label{EQ:tau_lower_upper_bound}
		\end{align}
	\end{subequations}
	for all $u\in\R$, where $A_*,A^*,\tau_*,\tau^*\in(0,\infty)$ are fixed.
\end{assumption}

We point out that the proofs in Sections~\ref{SEC:well_posedness}~and~\ref{SEC:relaxation_limit} are presented for the 2D/3D setting to avoid repeated dimension-specific case distinctions; the results, however, remain valid in 1D.
The point is that the $p$-growth is relevant only in 2D and 3D, ensuring $F(u)\in\LL^1(\Om)$ for $u\in\HH^1(\Om)$ via the Sobolev embedding $\HH^1(\Om)\hookrightarrow\LL^p(\Om)$, while in 1D we have $\HH^1(\Om)\hookrightarrow\CC(\varcl{\Om})$, which immediately yields $f(u),F(u)\in\CC(\varcl{\Om})$.

Note that condition \eqref{EQ:growth_condition_F} particularly implies $\dom(\calE)=(\Hav^1(\Om)+m)\times\LL^2(\Om)$. Moreover, \eqref{EQ:tau_lower_upper_bound} ensures
\begin{subequations}\label{EQ:dissipation_norm_equivalences}
	\begin{align}
		\frac{\min\{1,\tau_*\}}{2}\norm{\bfv}_H^2&\leq\calR(\bfu;\bfv)\leq\frac{\max\{1,\tau^*\}}{2}\norm{\bfv}_H^2,\label{EQ:dissipation_norm_equivalence}\\
		\frac{1}{2\max\{1,\tau^*\}}\norm{\bfmu}_{H^*}^2&\leq\calR^*(\bfu;\bfmu)\leq\frac{1}{2\min\{1,\tau_*\}}\norm{\bfmu}_{H^*}^2\label{EQ:dissipation_dual_norm_equivalence}
	\end{align}
\end{subequations}
for every $(\bfu,\bfv,\bfmu)\in\dom(\calE)\times H\times H^*$.

\subsubsection{Well-posedness}
To formulate our first main result, we recall Definition~\ref{DEF:gf_solution} of a gradient-flow solution associated to a given gradient system. Unless specified otherwise, the underlying gradient system is understood to be that introduced in Subsection~\ref{SUBSUBSEC:hyp_not}.

\begin{theorem}[Well-posedness]\label{TH:well_posedness}
	For any initial data $\bfu^0=(u^0,z^0)\in(\Hav^1(\Om)+m)\times\LL^\infty(\Om)$, there exists a gradient-flow solution $\bfu=(u,z)\in\AClocZ$ with $\bfu(0)=\bfu^0$. Any such solution $\bfu$ enjoys the regularity $\bfu\in\Lloc^2([0,\infty);\HH^2(\Om))\times\Hloc^1([0,\infty);\LL^\infty(\Om))$ and satisfies the energy-dissipation balance
	\begin{equation}\label{EQ:edb}
		\calE(\bfu(t))+\int_s^t\int_\Om\Bigparen{\bigabs{\nabla(-\Delta u+f(u)-A(u)(z-K(u)))}^2+\frac{1}{\tau(u)}(z-K(u))^2}\dd x\dd r=\calE(\bfu(s))
	\end{equation}
	for all $0\leq s\leq t<\infty$.

	Moreover, any two gradient-flow solutions $\bfu_i=(u_i,z_i)\in \AClocZ$ with initial values $\bfu_i(0)\in (\Hav^1(\Om)+m)\times\LL^\infty(\Om)$, $i=1,2$, satisfy the stability estimate
	\begin{equation}\label{EQ:stability_estimate}
		\norm{\bfu_1(t)-\bfu_2(t)}_H\leq\e^{C\int_s^t(1+(\norm{A'}_\infty+\norm{\tau'}_\infty)\norm{z_1-K(u_1)}_{\LL^\infty(\Om)})^2\dd r}\norm{\bfu_1(s)-\bfu_2(s)}_H
	\end{equation}
	for all $0\leq s\leq t<\infty$, where $C\in(0,\infty)$ is a constant only depending on $\beta$, $A^*$ and $\tau_*$. In particular, solutions emanating from a given initial value are unique.
\end{theorem}

The proof of Theorem~\ref{TH:well_posedness} is carried out in Subsection~\ref{SUBSEC:well_posedness}.\medskip

\begin{remark}
    \leavevmode
    \begin{enumerate}
        \item Note that the term $\nabla(-\Delta u+f(u)-A(u)(z-K(u)))$ appearing in~\eqref{EQ:edb} is a well-defined function in $\Lloc^2([0,\infty);\LL^2(\Om))$. This is implied by the regularity $\dot\bfu\in\Lloc^2([0,\infty);H)$, the gradient-flow equation and the subdifferential~\eqref{EQ:subdifferential_of_energy}.
        From \eqref{EQ:edb} and the Fenchel equivalences~\eqref{EQ:fenchel_equivalences} we even obtain the global control $z-K(u)\in\LL^2([0,\infty);\LL^2(\Om))$, as well as $-\Delta u+f(u)-A(u)(z-K(u))-\fraka(u,z)\in\LL^2([0,\infty);\HH^1(\Om))$ and $\dot\bfu\in\LL^2([0,\infty);H)$.
        \item In the above theorem, the unique solution $\bfu$ also satisfies the GFE with the Gateaux subdifferential replaced by the Fr\'echet subdifferential. This follows from the regularity $\bfu\in\Lloc^2([0,\infty);\LL^\infty(\Om;\R^2))$ and the subgradient estimate in Proposition~\ref{PRO:subdifferential}.
        \item If $A$ and $\tau$ are constant, the integrating factor in the stability estimate~\eqref{EQ:stability_estimate} does no longer depend on the solution. In particular, via an approximation argument, estimate~\eqref{EQ:stability_estimate} yields existence and uniqueness of GF solutions $\bfu$ in $\ACloc^2((0,\infty);\calH_m)$ (even EVI-solutions, as mentioned in the motivation) with respect to initial values in $\calH_m$.
    \end{enumerate}
\end{remark}

The second main result extends Theorem~\ref{TH:well_posedness} to initial values in the energy domain, allowing for $z^0$ merely in $\LL^2(\Om)$. Theorem~\ref{TH:existence_extended_initials} asserts that solutions still exist and satisfy the energy--dissipation inequality. The stability estimate~\eqref{EQ:stability_estimate} for two solutions remains valid as long as the initial value of one of the solutions lies in $\HH^1(\Om)\times\LL^\infty(\Om)$, which enables us to prove a weak--strong-type uniqueness.
\begin{theorem}[Finite-energy data]\label{TH:existence_extended_initials}
	Let $\bfu^0=(u^0,z^0)\in(\Hav^1(\Om)+m)\times\LL^2(\Om)$. Then, there exists a gradient-flow solution $\bfu=(u,z)\in\AClocZ$ with $\bfu(0)=\bfu^0$ that satisfies the energy-dissipation inequality
	\begin{equation}\label{EQ:edi}
		\calE(\bfu(t))+\int_0^t\int_\Om\Bigparen{\bigabs{\nabla\bigparen{-\Delta u+f(u)-A(u)(z-K(u))}}^2+\frac{1}{\tau(u)}(z-K(u))^2}\dd x\dd s\leq\calE(\bfu^0)
	\end{equation}
	for all $t\in[0,\infty)$.

	Moreover, if $\bfu=(u,z),\bfv=(v,y)\in\AClocZ$ are two gradient-flow solutions with initial values $\bfu(0)=:\bfu^0\in(\Hav^1(\Om)+m)\times\LL^\infty(\Om)$ and $\bfv(0)=:\bfv^0\in(\Hav^1(\Om)+m)\times\LL^2(\Om)$, then $\bfu\in\Lloc^2([0,\infty);\HH^2(\Om))\times\Hloc^1([0,\infty);\LL^\infty(\Om))$ by Theorem \ref{TH:well_posedness} and it holds
	\begin{equation}\label{EQ:ext_stability_estimate}
		\norm{\bfu(t)-\bfv(t)}_H\leq\e^{C\int_s^t(1+(\norm{A'}_\infty+\norm{\tau'}_\infty)\norm{z-K(u)}_{\LL^\infty(\Om)})^2\dd r}\norm{\bfu(s)-\bfv(s)}_H
	\end{equation}
	for all $0\leq s\leq t<\infty$, where $C\in(0,\infty)$ only depends on $\beta$, $A^*$ and $\tau_*$. In addition, if $\norm{z^0}_{\LL^\infty(\Om)}\leq M$ and $\calE(\bfu^0)\leq E_0$ for certain $M,E_0\in(0,\infty)$, and $T\in(0,\infty)$ is fixed, then $\norm{z-K(u)}_{\LL^2(0,T;\LL^\infty(\Om))}$ can be estimated by a constant depending only on $\Om$, $T$, $d$, $E_0$, $M$, $m$, $\beta$, $A^*$ and $\tau_*$.
\end{theorem}

The proof is provided in Subsection~\ref{SUBSEC:existence_extended_initials}.\medskip

By virtue of Lemma~\ref{LEM:sobolev_bochner_chain_rule}, the variable transformation $\Psi(u,z)=(u,q)$ from Subsection~\ref{SUBSEC:strategy} is rigorously justified. Thus, Theorems~\ref{TH:well_posedness}~and~\ref{TH:existence_extended_initials} imply:

\begin{corollary}[Original variables]
	For any $(u^0,q^0)\in(\Hav^1(\Om)+m)\times\LL^2(\Om)$, there exists a weak solution $(u,q)$ to \eqref{EQ:VPS_original} satisfying the energy-dissipation inequality, i.e.\ the following six conditions are fulfilled:
	\begin{enumerate}
		\item $u\in\LL^\infty([0,\infty);\HH^1(\Om))\cap\Lloc^2([0,\infty);\HH^2(\Om))$ with $\dot u\in\LL^2([0,\infty);\Havd)$;
		\item $q\in\CC([0,\infty);\LL^2(\Om))$ with $\dot q\in\Lloc^{4/3}([0,\infty);(\HH^1(\Om))^*)$;
		\item $-\Delta u+f(u)-A(u)q\in\Lloc^2([0,\infty);\HH^1(\Om))$;
		\item for almost all $t\in(0,\infty)$ it holds
		\begin{align*}
			\product{\dot u(t)}{\varphi}_{\Havd,\Hav^1(\Om)}&=-\int_\Om\nabla\bigparen{-\Delta u(t)+f(u(t))-A(u(t))q(t)}\cdot\nabla\varphi\dd x,\\
			\product{\dot q(t)}{\psi}_{(\HH^1(\Om))^*,\HH^1(\Om)}&=-\product{\dot u(t)}{A(u(t))\psi}_{(\HH^1(\Om))^*,\HH^1(\Om)}-\int_\Om\frac{1}{\tau(u(t))}q(t)\psi\dd x
		\end{align*}
		for all $\varphi\in\Hav^1(\Om)$, $\psi\in\HH^1(\Om)$;
		\item (initial conditions) $u(0)=u^0$ in $\LL^2(\Om)$ and $q(0)=q^0$ in $\LL^2(\Om)$ are fulfilled;
		\item (energy-dissipation inequality) for all $t\in(0,\infty)$ it holds (with $\calF$ as in \eqref{EQ:energy_original})
		\[\calF(u(t),q(t))+\int_0^t\int_\Om\Bigparen{\bigabs{\nabla\bigparen{-\Delta u+f(u)-A(u)q}}^2+\frac{1}{\tau(u)}q^2}\dd x\dd r\leq\calF(u^0,q^0).\]
	\end{enumerate}
	If in addition $q_0+K(u_0)\in\LL^\infty(\Om)$, then $q\in\Lloc^2([0,\infty);\LL^\infty(\Om))$, the solution $(u,q)$ is unique and satisfies the energy-dissipation balance.
\end{corollary}

\subsubsection{Relaxation}\label{SUBSEC:relaxation}
Beyond the well-posedness theory, we are interested in connecting the viscoelastic system to classical models of phase separation, where elastic effects are neglected.
Observe that, so far, the VPS system was formulated in dimensionless variables and inessential multiplicative constants were set equal to unity. To study relaxation limits, however, we need to take the characteristic size $q_c$ of the bulk stress variable to be of order $o(1)$.
We will therefore consider a rescaled free energy, which at the level of the original physical variables reads as
\begin{align*}
    \mathscr{\widetilde E}(u,q)=\int_\Om\Bigparen{\frac{1}{2}\abs{\nabla u}^2+F(u)}\dd x+\frac12\int_\Om\left(\frac{q}{q_c}\right)^2\dd x,
\end{align*}
and an appropriately rescaled dissipation mechanism
\begin{align*}
	\mathscr{\widetilde G}(u,q)=t_c^{-1}\begin{pmatrix}
		(-\Delta)^{-1}+A^2(u)\tau(u)&q_c^{-1}A(u)\tau(u)
		\\q_c^{-1}A(u)\tau(u)&q_c^{-2}\tau(u)
	\end{pmatrix}. 
\end{align*} 
Note that here $t_c>0$ denotes a characteristic time scale.
Below, we perform the corresponding change of variables directly in the ``diagonal" coordinates $(u,z)$.
We will then make the ansatz $q_c=\ve$ for $\ve\in(0,1)$ and study the asymptotic behaviour, as $\ve\searrow0$, for suitable choices of $A=\widehat A_\ve$, $\tau=\widehat\tau_\ve$, and $t_c=t_{c,\ve}$.

For the rigorous asymptotics, we first invoke Theorem~\ref{TH:existence_extended_initials} (with $A=\widehat A_\ve, \tau=\widehat\tau_\ve$), which provides gradient-flow solutions $\widehat\bfu_\ve$ to $(\calH_m,\widehat\calE_\ve,\widehat\calR_\ve)$, where $\widehat\calE_\ve$ and $\widehat\calR_\ve$ are given via
\begin{align}
	\widehat\calE_\ve(\bfu)&=\int_\Om\Bigparen{\frac{1}{2}\abs{\nabla u}^2+F(u)}\dd x+\frac{1}{2}\int_\Om\bigparen{z-\widehat K_\ve(u)}^2\dd x,\label{EQ:unrescaled_epsilon_energy}\\
	\widehat\calR_\ve(\bfv;\bfu)&=\frac{1}{2}\bigparen{\norm{v}_\Havd^2+(\widehat\tau_\ve(u)y,y)_{\LL^2(\Om)}}\notag
\end{align}
with $\widehat K_\ve'=\widehat A_\ve, \widehat K_\ve(0)=0$.
As explained above, before studying the limit $\ve\searrow0$, we need to consider new variables that account for a vanishing bulk stress size $q_c$.
The appropriate transformation of the dependent variables is given by $\widehat\Psi(u,z)=(u,q_c^{-1}z)=(\widehat u,\widehat z),$ where $q_c=\ve$.
Combined with the time rescaling under the ansatz $t_c=\ve^{-\gamma}, \gamma\in[0,\infty)$, this leads to a gradient system with the following rescaled energy and dissipation potential
\begin{align}
	\calE_\ve(u,z)&:=\int_\Om\Bigparen{\frac{1}{2}\abs{\nabla u}^2+F(u)}\dd x+\frac{1}{2\ve^2}\int_\Om(z-K_\ve(u))^2\dd x,\label{EQ:energy_epsilon}\\
	\calR_\ve^{(\gamma,\kappa)}((u,z);(v,y))&:=\frac{1}{2}\bigparen{\ve^\gamma\norm{v}_\Havd^2+\ve^\kappa(\tau_\ve(u)y,y)_{\LL^2(\Om)}},\label{EQ:dissipation_potential_epsilon}
\end{align}
where $K_\ve(u):=\ve\widehat K_\ve(u)$, and $\ve^\kappa\tau_\ve(u):=\ve^{\gamma-2}\widehat\tau_\ve(u)$.
The goal is then to show that, in some sense, the sequence of solutions in the new variables $(\bfu_\ve)_{\ve\in(0,1)}$, determined by $\widehat\Psi(\bfu_\ve(t))=\widehat\bfu_\ve(\ve^{-\gamma}t)$, converges (up to a subsequence) to a gradient-flow solution $\bfu_0$ of some ``effective'' gradient system $(\calH_\mathrm{eff},\calE_\mathrm{eff},\calR_\mathrm{eff})$ (see also Remark \ref{REM:relation_pE_EDB_EDP_convergence} below).
Formula \eqref{EQ:energy_epsilon} and the definition of $K_\ve(u)$ suggest to take $\widehat A_\ve=\ve^{-1}A_\ve$ for $A_\ve$ bounded away from zero and infinity.

\begin{assumption}\label{ASS:relaxation_limit}
	$\widehat A_\ve,\widehat\tau_\ve\in\WW^{1,\infty}(\R)$ for every $\ve\in(0,1)$ and, given $\gamma\in[0,\infty)$, there exists $\kappa\in[0,\infty)$ with $\kappa\cdot\gamma=0$, such that $A_\ve:=\ve\widehat A_\ve$ and $\tau_\ve:=\ve^{\gamma-2-\kappa}\widehat\tau_\ve$ converge uniformly, i.e.
	\begin{equation}\label{EQ:bulk_relaxation_convergence}
		\norm{A_\ve-A}_{\CC(\R)}\to0\text{ and }\norm{\tau_\ve-\tau}_{\CC(\R)}\to0\text{ as }\ve\searrow0
	\end{equation}
	for certain $A,\tau\in\WW^{1,\infty}(\R)$. Moreover, there exist $\tau_*,\tau^*,A_*,A^*\in(0,\infty)$ such that
	\begin{subequations}\label{EQ:bulk_relaxation_uniform_epsilon_bound}
		\begin{align}
			A_*&\leq A_\ve(u)\leq A^*,\\
			\tau_*&\leq\tau_\ve(u)\leq\tau^*.
		\end{align}
	\end{subequations}
\end{assumption}
Note that the bounds \eqref{EQ:bulk_relaxation_uniform_epsilon_bound} are preserved in the limit, i.e.\ $A_*\leq A\leq A^*$ and $\tau_*\leq\tau\leq\tau^*$. Denote by $K$ the primitive of $A$ with $K(0)=0$. Since $K$ is strictly increasing, its inverse $K^{-1}$ exists.

In Proposition~\ref{PRO:energies_gamma_convergence}, we will show that, with this choice of $\widehat A_\ve$, $(\calE_\ve)_{\ve\in(0,1)}$ has a $\Gamma$-limit $\calE_0:\calH_m\to[0,\infty]$ given by
\[\calE_0(\bfu):=\begin{cases}
	\calECH(u)&\text{if }z=K(u)\text{ and }u\in\Hav^1(\Om)+m,\\
	+\infty&\text{otherwise},
\end{cases}\]
where
\begin{equation}\label{EQ:CH_energy}
	\calECH(u)=\int_\Om\Bigparen{\frac{1}{2}\abs{\nabla u}^2+F(u)}\dd x
\end{equation}
denotes the \textit{Ginzburg--Landau energy} or \textit{Modica-Mortola functional}. Note that the constraint $z=K(u)$ in the definition of $\calE_0$ is nonlinear as soon as $A$ is not constant, in which case the explicit form of any reasonable subdifferential of $\calE_0$ may be quite involved.

The effective dissipation potential $\calR_\mathrm{eff}$ will depend on the choice of parameters.
Formula~\eqref{EQ:dissipation_potential_epsilon} suggests that there are three non-trivial cases, depending on the positivity of $\kappa,\gamma\in[0,\infty)$:
\[\text{(CH) if $\gamma=0,\kappa>0$ $``(0,+)"$, \; (vCH) if $\gamma=0,\kappa=0$ $``(0,0)"$, \; (mAC) if $\gamma>0,\kappa=0$ $``(+,0)"$}.\]
The actual gradient systems are given by $(\Hdav+m,\calECH,\calRCH)$, $(\Lav^2(\Om)+m,\calECH,\calRmAC)$ and $(\Lav^2(\Om)+m,\calECH,\calRvCH)$, where
\begin{align*}
	\calRCH:\Havd&\to[0,\infty),&\calRCH(v)&:=\frac{1}{2}\norm{v}_\Havd^2,\\
	\calRmAC:(\Hav^1(\Om)+m)\times\Lav^2(\Om)&\to[0,\infty),&\calRmAC(u;v)&:=\frac{1}{2}\bigparen{A^2(u)\tau(u)v,v}_{\LL^2(\Om)},\\
	\calRvCH:(\Hav^1(\Om)+m)\times\Lav^2(\Om)&\to[0,\infty),&\calRvCH(u;v)&:=\calRCH(v)+\calRmAC(u;v).
\end{align*}
\smallskip

We are now in a position to state our final main result, which will be proved in Section~\ref{SEC:relaxation_limit}.

\begin{theorem}[Relaxation limit]\label{TH:relaxation_limit}
	Assume $(\bfu_\ve)_{\ve\in(0,1)}\subset\AClocZ$ is a family of gradient-flow solutions to $(\calH_m,\calE_\ve,\calR_\ve^{(\gamma,\kappa)})$ that satisfy the energy-dissipation inequality \eqref{EQ:edi_general_definition}.
    Let the initial conditions $\bfu_\ve(0)=:\bfu_\ve^0\in(\Hav^1(\Om)+m)\times\LL^2(\Om)$ be well-prepared, i.e.\ assume that there exists $u_0^0\in\Hav^1(\Om)+m$ such that
	\begin{equation}\label{EQ:init_well_prepared}
		\bfu_\ve^0\to(u_0^0,K(u_0^0))\text{ in }\calH_m,\quad\calE_\ve(\bfu_\ve^0)\to\calECH(u_0^0),\quad\ve\searrow0.
	\end{equation}
	Then, for every sequence $\vek\searrow0$, $k\to\infty$, there exists a subsequence (not relabelled) and
    $u_0\in\CC([0,\infty);\LL^2(\Om))\cap\LL^\infty([0,\infty);\HH^1(\Om))\cap\Lloc^2([0,\infty);\HH^2(\Om))$ such that
	\begin{subequations}\label{EQ:relaxation_convergences}
		\begin{empheq}[left=\empheqlbrace]{align}
			u_\vek\to u_0&\quad\text{in }\CC_\mathrm{loc}([0,\infty);\LL^2(\Om)),\label{EQ:relaxation_u_convergence}\\
			z_\vek\to K(u_0)&\quad\text{in }\CC_\mathrm{loc}([0,\infty);\LL^2(\Om)),\label{EQ:relaxation_z_convergence}\\
			\calE_\vek(\bfu_\vek(t))\to\calECH(u_0(t))&\quad\text{for all }t\in[0,\infty),\label{EQ:relaxation_energy_convergence}
		\end{empheq}
	\end{subequations}
	$u_0$ has a time derivative in the space
	\[\dot u_0\in\begin{cases}
		\LL^2([0,\infty);\Havd)&\text{if }(\gamma,\kappa)=(0,+),\\
		\LL^2([0,\infty);\Lav^2(\Om))&\text{if }\kappa=0
	\end{cases}\]
	with convergence
	\begin{equation}\label{EQ:relaxation_time_derivative_convergence}
		\begin{cases}
			\dot u_\vek\to\dot u_0\text{ in }\Lloc^2([0,\infty);\Havd)&\text{if }\gamma=0,\\
			\dot z_\vek\to A(u_0)\dot u_0\text{ in }\Lloc^2([0,\infty);\Lav^2(\Om))&\text{if }\kappa=0,
		\end{cases}
	\end{equation}
	and $u_0$ is a gradient-flow solution, also satisfying \eqref{EQ:edb_general_definition}, of one of the three gradient systems
	\[\begin{cases}
		(\Havd+m,\calECH,\calRCH)&\text{if }(\gamma,\kappa)=(0,+),\\
		(\Lav^2(\Om)+m,\calECH,\calRvCH)&\text{if }(\gamma,\kappa)=(0,0),\\
        (\Lav^2(\Om)+m,\calECH,\calRmAC)&\text{if }(\gamma,\kappa)=(+,0).
	\end{cases}\]
\end{theorem}

\begin{remark}[Well-prepared data]
    Thanks to the $\Gamma$-convergence of the energies (cf.\ Proposition~\ref{PRO:energies_gamma_convergence}), the well-preparedness~\eqref{EQ:init_well_prepared} can be achieved for any $(u,K(u))$ with $u\in\Hav^1(\Om)+m$. 
\end{remark}

\begin{remark}[Relation to EDB- and EDP-convergence]\label{REM:relation_pE_EDB_EDP_convergence}
	In the study of singular limits of gradient systems, the notions of \textit{EDB-convergence} and \textit{EDP-convergence} are commonly encountered in the literature, both falling under the umbrella of \textit{Evolutionary $\Gamma$-convergence}.
    The core idea of EDB-convergence is to establish the $\Gamma$-convergences $\calE_\ve\Gato\calE_\mathrm{eff}$ and $\calR_\ve^{(\gamma,\kappa)}\Gato\calR_\mathrm{eff}$ to certain effective $\calE_\mathrm{eff}$ and $\calR_\mathrm{eff}$.
    Then, assuming that the approximate solutions $\bfu_\ve$ comply with \eqref{EQ:edb_general_definition} (or \eqref{EQ:edi_general_definition}) and the initial conditions $(\bfu_\ve(0))_\ve$ are well-prepared, one can, under additional technical assumptions, extract suitable convergent subsequences.
    Combining liminf-estimates in \eqref{EQ:edb_general_definition} (or \eqref{EQ:edi_general_definition}) with a subdifferential closedness, and assuming an appropriate chain rule holds for $\calE_\mathrm{eff}$, the energy-dissipation principle mentioned in Subsection~\ref{SUBSEC:basic_concepts} guarantees that the limit is indeed a solution to $(\calH_m,\calE_\mathrm{eff},\calR_\mathrm{eff})$.
    The strategy behind EDP-convergence is similar, but instead of showing $\calR_\ve^{(\gamma,\kappa)}\Gato\calR_\mathrm{eff}$, one proves a liminf-estimate of the total dissipation functional.

    We highlight that our approach in Section~\ref{SEC:relaxation_limit} is mainly based on the GFE rather than on EDB- or EDP-convergence. One can show a liminf-estimate, which would yield the EDI for $(\calH_m,\calR_\mathrm{eff},\calE_\mathrm{eff})$.
    However, in order to conclude $-\D_{\bfv}\calR_\mathrm{eff}(\bfu;\dot\bfu)\in\subDL\calE_0(\bfu)$ with the limit subdifferential $\subDL\calE_0$ determined in Lemma~\ref{LEM:limit_subdifferential}, one would need to apply a suitable chain rule whose validity, without extra assumptions, is unclear at this stage.
    In the special case $A,\tau\equiv\text{const.}$, one has $\subDL\calE_0=\subDF\calE_0$, and a chain rule can be shown through a simple subgradient estimate. Hence, in the constant case, EDP-type arguments can be used.
    In the (vCH) case, the dissipation potential $\calR_\ve^{(\gamma,\kappa)}$ is even non-degenerate, and the limit can be taken using the (EVI) formulation.
\end{remark}

\section{Properties of the energy and its subdifferential}\label{SEC:preliminaries}
In this section, we establish key properties of $\calE$ and $\subDG\calE$ needed in the proof of Theorem~\ref{TH:well_posedness}. We emphasise that $\calE$ is not semiconvex. To see this, we note that the Hessian
\[\Hess\Bigparen{(u,z)\mapsto\frac{1}{2}(z-K(u))^2}=\begin{pmatrix}
	A^2(u)-A'(u)(z-K(u))&-A(u)\\
	-A(u)&1
\end{pmatrix},\]
viewed as a bilinear form, can not be bounded from below uniformly in $(u,z)$ unless $A'\equiv0$. Consequently, several valuable properties -- like the strong--weak closedness of the graph of $\subDG\calE$ or the chain rule -- do not come for free. We will therefore invest particular effort into a thorough investigation of the Gateaux subdifferential.

In Proposition~\ref{PRO:subdifferential} we characterise $\subDG\calE$, showing that the Gateaux subdifferential $\subDG\calE(\bfu)$ is single-valued and takes the precise form anticipated in \eqref{EQ:subdifferential_of_energy}. An additional observation is the subdifferential estimate \eqref{EQ:subdifferential_characterising_estimate}, which will play an important role in the upcoming proof of the chain rule. Another important estimate concerning $\subDG\calE$, namely \eqref{EQ:subdifferential_monotonicity}, is proved in Proposition~\ref{PRO:monotonicity_subdifferential}. This will be used to show that the absolutely continuous interpolants $(\whbfu_\kappa)_\kappa$ as defined in \eqref{EQ:piecewise_affine_interpolant} form a Cauchy sequence in the strong topology $\CC([0,T];\calH_m)$ for any $T<\infty$, and to derive the stability estimates \eqref{EQ:stability_estimate} and \eqref{EQ:ext_stability_estimate}. Proposition~\ref{PRO:monotonicity_subdifferential} is followed by an $\HH^2(\Om)$-estimate for the phase field variable $u$, the closedness properties of $\subDG\calE$ and $\D_{\bfv}\calR$, and, eventually, the chain rule Proposition~\ref{PRO:chain_rule}, whose proof is based on \eqref{EQ:subdifferential_characterising_estimate}.

Let us first recall two auxiliary results.

\begin{lemma}\label{LEM:interpolations}
	In the sense of Gelfand triples, we have
	\begin{align}
		\norm{u}_{\LL^2(\Om)}&\leq\norm{u}_{\Hav^1(\Om)}^{1/2}\norm{u}_\Havd^{1/2}&\text{for all }u&\in\Hav^1(\Om),\label{EQ:L2_interpolation}\\
		\norm{u}_{\Hav^1(\Om)}&\leq C\norm{u}_{\HH^2(\Om)}^{2/3}\norm{u}_\Havd^{1/3}&\text{for all }u&\in\Hav^2(\Om),\label{EQ:H1_interpolation}
	\end{align}
	where $C\in(0,\infty)$ only depends on $\Om$ and $d$.
\end{lemma}

\begin{lemma}\label{LEM:H2_Delta_estimate}
    There exists a constant $C\in(0,\infty)$ only depending on $\Om$ such that
	\begin{equation}\label{EQ:H2_Delta_estimate}
		\biggnorm{u-\avint{_\Om}u\dd x}_{\HH^2(\Om)}\leq C\norm{\Delta u}_{\LL^2(\Om)}
	\end{equation}
	for all $u\in\HH^2(\Om)$ with $\nabla u\cdot\n=0$ on $\boundary{\Om}$.
\end{lemma}

The interpolation inequality \eqref{EQ:L2_interpolation} follows from $\product{u}{u}_{\Havd,\Hav^1(\Om)}=(u,u)_{\LL^2(\Om)}$, and \eqref{EQ:H1_interpolation} can be readily deduced from $\norm{u}_{\Hav^1(\Om)}=\norm{\nabla u}_{\LL^2(\Om;\R^d)}\leq C\norm{u}_{\HH^2(\Om)}^{1/2}\norm{u}_{\LL^2(\Om)}^{1/2}$ (Nirenberg~\cite{Nirenberg_1959}) and \eqref{EQ:L2_interpolation}.
Lemma \ref{LEM:H2_Delta_estimate} is a combination of the a priori estimate from elliptic regularity (see \cite[Chapters II, III]{Grisvard_EllipticNonsmooth_2011} for instance) $\norm{u}_{\HH^2(\Om)}\leq C(\norm{\Delta u}_{\LL^2(\Om)}+\norm{u}_{\LL^2(\Om)})$ and Poincar\'e-Wirtinger's inequality $\norm{u-\avint{_\Om}u\dd x}_{\LL^2(\Om)}\leq\CPW\norm{\nabla u}_{\LL^2(\Om;\R^d)}\leq\CPW^2\norm{\Delta u}_{\LL^2(\Om)}$.

\begin{proposition}\label{PRO:energy_weakly_lower_semicontinuous}
	The energy $\calE:\calH_m\to[0,+\infty]$, defined in~\eqref{EQ:banach_energy_functional}, is weakly lower semicontinuous.
\end{proposition}

\begin{proof}
	Let $(\bfu_n)_{n\in\N}=(u_n,z_n)_{n\in\N}\subset\calH_m$, $\bfu=(u,z)\in\calH_m$ such that $\bfu_n\weakto\bfu$ for $n\to\infty$ in $\calH_m$. If $\liminf_{n\to\infty}\calE(\bfu_n)=+\infty$ then there is nothing to do. Otherwise, there exists a subsequence $(\bfu_{n_k})_{k\in\N}\subseteq(\bfu_n)_{n\in\N}$ such that $\liminf_{n\to\infty}\calE(\bfu_n)=\lim_{k\to\infty}\calE(\bfu_{n_k})\in[0,\infty)$ and $\bfu_{n_k}\in\dom(\calE)$ for any $k$. In particular, $(\calE(\bfu_{n_k}))_{k\in\N}$ is bounded.

	From there we get that $(\nabla u_{n_k})_{k\in\N}$ is bounded in $\LL^2(\Om;\R^d)$ and $(z_{n_k}-K(u_{n_k}))_{k\in\N}$ is bounded in $\LL^2(\Om)$. Since $\avint{_\Om}u_{n_k}\dd x=m$ for all $k\in\N$, reflexivity of $\HH^1(\Om)$ and $\LL^2(\Om)$ and uniqueness of weak limits imply $(u_{n_k},z_{n_k})\weakto(u,z)$ in $\HH^1(\Om)\times\LL^2(\Om)$. Moreover, Rellich--Kondrachov provides $u_{n_k}\to u$ in $\LL^2(\Om)$ and almost everywhere in $\Om$ for non-relabelled subsequence. Since $K$ is continuous and grows linearly, we also infer $K(u_{n_k})\to K(u)$ in $\LL^2(\Om)$, and continuity of $F$ yields $F(u_{n_k})\to F(u)$ almost everywhere in $\Om$. Thus, weak lower semicontinuity
	\[\norm{\nabla u}_{\LL^2(\Om;\R^d)}^2\leq\liminf_{k\to\infty}\norm{\nabla u_{n_k}}_{\LL^2(\Om;\R^d)}^2,\quad\norm{z-K(u)}_{\LL^2(\Om)}^2\leq\liminf_{k\to\infty}\norm{z_{n_k}-K(u_{n_k})}_{\LL^2(\Om)}^2\]
	and Fatou's lemma $\int_\Om F(u)\dd x\leq\liminf_{k\to\infty}\int_\Om F(u_{n_k})\dd x$ yield $\calE(\bfu)\leq\liminf_{n\to\infty}\calE(\bfu_n)$.
\end{proof}

\begin{proposition}[Gateaux subdifferential $\subDG\calE$]\label{PRO:subdifferential}
	Define
	\[\mathcal{A}:=\biggset{(u,z)\in(\Hav^2(\Om)+m)\times\LL^2(\Om)}{\begin{array}{c}
			\nabla u\cdot\n=0\text{ on }\boundary{\Om}\text{ and}\\
			-\Delta u+f(u)-A(u)(z-K(u))\in\HH^1(\Om)
		\end{array}}\]
	and, for any $\bfu=(u,z)\in\mathcal{A}$,
	\begin{equation}\label{EQ:subdifferential_delta_definition}
		\delta\calE(\bfu):=\begin{pmatrix}
			-\Delta u+f(u)-A(u)(z-K(u))-\fraka(u,z)\\
			z-K(u)
		\end{pmatrix},
	\end{equation}
	where $\fraka(u,z)=\avint{_\Om}(f(u)-A(u)(z-K(u)))\dd x$. Observe that $\delta\calE(\bfu)\in H^*$ for every $\bfu\in\mathcal{A}$. Then:
	\begin{enumerate}
		\item $\dom(\subDG\calE)=\mathcal{A}$ and $\subDG\calE(\bfu)=\{\delta\calE(\bfu)\}$,
		\item if $\bfu=(u,z)\in\mathcal{A}$ and in addition $z\in\LL^\infty(\Om)$, then the subgradient estimate
		\begin{equation}\label{EQ:subdifferential_characterising_estimate}
			\calE(\bfv)\geq\calE(\bfu)+\product{\delta\calE(\bfu)}{\bfv-\bfu}_{H^*,H}+\lambda(\bfu)\norm{\bfv-\bfu}_H^2
		\end{equation}
        holds for all $\bfv\in\calH_m$, where $\lambda(u,z):=-\frac{1}{8}\bigparen{\beta+\norm{A'}_\infty\norm{z-K(u)}_{\LL^\infty(\Om)}}^2$.
	\end{enumerate}
\end{proposition}

\begin{proof}
	Recalling $F(u)=h(u)-\frac{\beta}{2}u^2$, we compute for $\bfu=(u,z)\in\mathcal{A}$, $\bfv=(v,y)\in\dom(\calE)$
	\begin{align}\label{EQ:subdifferential_preliminary_estimate}
		\calE(\bfv)&-\calE(\bfu)-\product{\delta\calE(\bfu)}{\bfv-\bfu}_{H^*,H}\notag\\
		&=\int_\Om\Bigparen{\frac{1}{2}\abs{\nabla v}^2+F(v)+\frac{1}{2}(y-K(v))^2}\dd x\notag\\
		&\quad-\int_\Om\Bigparen{\frac{1}{2}\abs{\nabla u}^2+F(u)+\frac{1}{2}(z-K(u))^2}\dd x\notag\\
		&\quad-\int_\Om\Bigparen{\bigparen{-\Delta u+f(u)-A(u)(z-K(u))}(v-u)+\bigparen{z-K(u)}(y-z)}\dd x\notag\displaybreak[1]\\
		&=\frac{1}{2}\int_\Om\abs{\nabla(v-u)}^2\dd x+\int_\Om\bigparen{\underbrace{h(v)-h(u)-h'(u)(v-u)}_{\geq0\text{ since }h\text{ is convex}}}\dd x\notag\\
		&\quad-\frac{\beta}{2}\int_\Om(v-u)^2\dd x+\frac{1}{2}\int_\Om\bigparen{(y-K(v))-(z-K(u))}^2\dd x\notag\\
		&\quad+\int_\Om(z-K(u))\bigparen{K(u)-K(v)+A(u)(v-u)}\dd x\notag\displaybreak[1]\\
		&\geq\frac{1}{2}\norm{v-u}_{\Hav^1(\Om)}^2-\frac{\beta}{2}\norm{v-u}_{\LL^2(\Om)}^2-\frac{1}{2}\norm{A'}_\infty\int_\Om\abs{z-K(u)}\abs{v-u}^2\dd x.
	\end{align}
	Here, in the first equation we inserted the definitions of $\calE$ and $\delta\calE(\bfu)$, in the second we rearranged the terms which is well-defined since $K(u),K(v)\in\LL^2(\Om)$ and $h'(u)\in\LL^{6/5}(\Om)$ by \eqref{EQ:growth_condition_f} and the Sobolev embedding $\HH^1(\Om)\hookrightarrow\LL^6(\Om)$, and in the third equation we used
	\[K(v)-K(u)-A(u)(v-u)=\int_u^v(v-\zeta)A'(\zeta)\dd\zeta\leq\frac{1}{2}\norm{A'}_\infty(u-v)^2.\]
	Now we turn to the proof of the statements (i) and (ii).

	\begin{enumerate}
		\item We first show that $\delta\calE(\bfu)\in\subDG\calE(\bfu)$ for $\bfu\in\mathcal{A}$ via definition \eqref{EQ:gateaux_subdifferential_definition}. So, let $\bfv\in H$. If $\bfv\notin\Hav^1(\Om)\times\LL^2(\Om)$, then $\bfu+t\bfv\notin\dom(\calE)$ for every $t>0$ and there is nothing to do. Otherwise, i.e.\ if $\bfv=(v,y)\in\Hav^1(\Om)\times\LL^2(\Om)$, then $\bfu+t\bfv\in\dom(\calE)$. From \eqref{EQ:subdifferential_preliminary_estimate} we deduce
		\begin{multline}\label{EQ:gateaux_subdifferential_computation}
			\calE(\bfu+t\bfv)-\calE(\bfu)-t\product{\delta\calE(\bfu)}{\bfv}_{H^*,H}\\
			\quad\geq\frac{t^2}{2}\biggparen{\norm{v}_{\Hav^1(\Om)}^2-\beta\norm{v}_{\LL^2(\Om)}^2-\norm{A'}_\infty\norm{z-K(u)}_{\LL^2(\Om)}\norm{v}_{\LL^4(\Om)}^2},
		\end{multline}
		and the right-hand side is real-valued due to Sobolev embedding $\HH^1(\Om)\hookrightarrow\LL^4(\Om)$. Hence, after dividing \eqref{EQ:gateaux_subdifferential_computation} by $t$, we obtain in the limit $t\searrow0$ the desired estimate from \eqref{EQ:gateaux_subdifferential_definition} and we conclude $\delta\calE(\bfu)\in\subDG\calE(\bfu)$ for $\bfu\in\mathcal{A}$.

		It remains to show $\dom(\subDG\calE)\subseteq\mathcal{A}$ and $\subDG\calE(\bfu)\subseteq\{\delta\calE(\bfu)\}$. To this end, we let $\bfu\in\dom(\subDG\calE)$ and $\bfmu=(\mu,\xi)\in\subDG\calE(\bfu)$, aiming to show $\bfu\in\mathcal{A}$ and $\bfmu=\delta\calE(\bfu)$. By definition of $\subDG\calE(\bfu)$ we have for all $\bfv\in H$
		\begin{equation}\label{EQ:subdifferential_definition_inserted}
			\liminf_{t\searrow0}\frac{\calE(\bfu+t\bfv)-\calE(\bfu)}{t}\geq\product{\bfmu}{\bfv}_{H^*,H}\geq\limsup_{t\nearrow0}\frac{\calE(\bfu+t\bfv)-\calE(\bfu)}{t}.
		\end{equation}
		Since $\bfu+t\bfv\in\dom(\calE)$ for every $t\in\R$ and $\bfv=(v,y)\in\Hav^1(\Om)\times\LL^2(\Om)$, we compute
		\begin{align*}
			&\quad\frac{\calE(\bfu+t\bfv)-\calE(\bfu)}{t}\\
			&=\frac{1}{t}\int_\Om\Bigparen{\frac{1}{2}\abs{\nabla(u+tv)}^2+F(u+tv)+\frac{1}{2}(z+ty-K(u+tv))^2}\dd x\\
			&\quad-\frac{1}{t}\int_\Om\Bigparen{\frac{1}{2}\abs{\nabla u}^2+F(u)+\frac{1}{2}(z-K(u))^2}\dd x\\
			&=\int_\Om\nabla u\cdot\nabla v\dd x+\frac{t}{2}\int_\Om\abs{\nabla v}^2\dd x+\int_\Om\frac{F(u+tv)-F(u)}{t}\dd x\\
			&\quad+\int_\Om\frac{(K(u+tv))^2-(K(u))^2}{2t}\dd x-\int_\Om\frac{K(u+tv)-K(u)}{t}z\dd x\\
			&\quad+\int_\Om(z-K(u+tv))y\dd x+\frac{t}{2}\int_\Om y^2\dd x.
		\end{align*}
		Owing to the $p$-growth of $f$ from \eqref{EQ:growth_condition_f} and the linear growth of $K$, we can take the limits $t\searrow0$ and $t\nearrow0$. In combination with \eqref{EQ:subdifferential_definition_inserted} we conclude that, for every $\bfv\in\Hav^1(\Om)\times\LL^2(\Om)$,
		\begin{multline}\label{EQ:determine_frechet_subdifferential}
			\product{v}{\mu}_{\Havd,\Hav^1(\Om)}+(\xi,y)_{\LL^2(\Om)}=\product{\bfmu}{\bfv}_{H^*,H}\\
			=\int_\Om\Bigparen{\nabla u\cdot\nabla v+\bigparen{f(u)+A(u)K(u)-A(u)z}v+(z-K(u))y}\dd x.
		\end{multline}
		In particular, considering $\bfv=(0,y)$ for $y\in\LL^2(\Om)$, we deduce $\xi=z-K(u)$ from \eqref{EQ:determine_frechet_subdifferential}. Furthermore, if we restrict \eqref{EQ:determine_frechet_subdifferential} to functions of type $\bfv=(v,0)$ for $v\in\Hav^1(\Om)$, we obtain
		\[\int_\Om\nabla u\cdot\nabla v\dd x=\int_\Om\bigparen{\mu-f(u)+A(u)(z-K(u))}v\dd x,\]
		i.e.\ $u$ is a weak solution of the nonlinear elliptic problem
		\begin{PDEsystem*}
			\PDElineIn{-\Delta u+h'(u)}{\beta u+\mu+A(u)(z-K(u))+\fraka(u,z)}{\Om},\\
			\PDElineOn{\nabla u\cdot\n}{0}{\boundary{\Om}}.
		\end{PDEsystem*}
		Since $\ell:=\beta u+\mu+A(u)(z-K(u))+\fraka(u,z)\in\LL^2(\Om)$ and $h\in\CC^2(\R)$ is convex, i.e.\ $h'$ monotonically increasing, we can use the Stampacchia method in order to show $u\in\LL^\infty(\Om)$. For this argument, $\ell$ needs to lie in $\LL^r(\Om)$ for some $r>d/2$, hence $r=2$ is valid for $d\leq3$, see, e.g.\ \cite[Theorem 4.5 and Section 7.2.2]{Troeltzsch_OptimaleSteuerungPDE_2009}.

		In particular, since $h'$ is continuous, $h'(u)\in\LL^\infty(\Om)$, and whence $-\Delta u\in\LL^2(\Om)$. The regularity assumptions on $\Om$ therefore imply $u\in\HH^2(\Om)$ (cf. \cite[Chapters II and III]{Grisvard_EllipticNonsmooth_2011}). Finally,
		\[-\Delta u+h'(u)-\beta u-A(u)(z-K(u))=\mu+\fraka(u,z)\in\HH^1(\Om),\]
		so that $\bfu\in\mathcal{A}$ is shown and we have $\bfmu=\delta\calE(\bfu)$.
		\item Let $\bfv=(v,y)\in\calH_m$, and we may assume $\bfv\in\dom(\calE)$, otherwise \eqref{EQ:subdifferential_characterising_estimate} is trivial. Then, from \eqref{EQ:subdifferential_preliminary_estimate} we receive
		\begin{align*}
			\calE(\bfv)&-\calE(\bfu)-\product{\delta\calE(\bfu)}{\bfv-\bfu}_{H^*,H}\\
			&\geq\frac{1}{2}\norm{v-u}_{\Hav^1(\Om)}^2-\frac{\beta}{2}\norm{v-u}_{\LL^2(\Om)}^2-\frac{1}{2}\norm{A'}_\infty\norm{z-K(u)}_{\LL^\infty(\Om)}\norm{v-u}_{\LL^2(\Om)}^2.
		\end{align*}
		Since $u\in\HH^2(\Om)$ and $d\leq3$, the Sobolev-embedding $\HH^2(\Om)\hookrightarrow\LL^\infty(\Om)$ ensures $u\in\LL^\infty(\Om)$. If we then apply \eqref{EQ:L2_interpolation} to $v-u$ (note that $v$ and $u$ have the same mean value $m$) and $\ve$-Young's inequality, we arrive directly at \eqref{EQ:subdifferential_characterising_estimate}, noting that $\norm{v-u}_\Havd\leq\norm{\bfv-\bfu}_H$.\qedhere
	\end{enumerate}
\end{proof}

\begin{proposition}[Monotonicity-type estimate for $\subDG\calE$]\label{PRO:monotonicity_subdifferential}
	Let $\bfu_i=(u_i,z_i)\in\dom(\subDG\calE)$ and $\bfv_i=(v_i,y_i)\in\dom(\calE)$, $i=1,2$, where $z_i\in\LL^\infty(\Om)$ for at least one $i$. Then it holds
	\begin{multline}\label{EQ:subdifferential_monotonicity}
		\bigparen{\bbK(\bfv_1)\delta\calE(\bfu_1)-\bbK(\bfv_2)\delta\calE(\bfu_2),\bfu_1-\bfu_2}_H\\
		\geq-\Lambda(\bfu_i)\norm{\bfu_1-\bfu_2}_H^2-\omega(\bfu_i)\bigparen{\norm{v_1-u_1}_{\LL^2(\Om)}^2+\norm{v_2-u_2}_{\LL^2(\Om)}^2},
	\end{multline}
	where $\Lambda(\bfu):=C_1\bigparen{1+(\norm{A'}_\infty+\norm{\tau'}_\infty)\norm{z-K(u)}_{\LL^\infty(\Om)}}^2$ for some constant $C_1\in(0,\infty)$ depending only on $\beta,A^*,\tau_*$, and $\omega(\bfu):=\norm{\tau'}_\infty^2\norm{z-K(u)}_{\LL^\infty(\Om)}^2$.
\end{proposition}

\begin{proof}
	After possibly renumbering indices, we may assume $i=1$. Inserting the definitions of $\bbK(\bfv_i)$, $\delta\calE(\bfu_i)$, and rearranging terms, we obtain
	\begin{align}\label{EQ:subdifferential_monotonicity_computation}
		\bigparen{\bbK&(\bfv_1)\delta\calE(\bfu_1)-\bbK(\bfv_2)\delta\calE(\bfu_2),\bfu_1-\bfu_2}_H\notag\\
		&=\int_\Om\nabla\Bigparen{-\Delta(u_1-u_2)+f(u_1)-f(u_2)-A(u_1)(z_1-K(u_1))\notag\\
			&\qquad\qquad+A(u_2)(z_2-K(u_2))}\cdot\nabla\bigparen{(-\Delta)^{-1}(u_1-u_2)}\dd x\notag\\
		&\quad+\int_\Om\Bigparen{\frac{z_1-K(u_1)}{\tau(v_1)}-\frac{z_2-K(u_2)}{\tau(v_2)}}(z_1-z_2)\dd x\notag\\
		&=\int_\Om\abs{\nabla(u_1-u_2)}^2\dd x+\int_\Om\underbrace{(h'(u_1)-h'(u_2))(u_1-u_2)}_{\geq0\text{ by convexity of }h}\dd x\notag\\
		&\quad-\beta\int_\Om(u_1-u_2)^2\dd x-\int_\Om A(u_2)(z_1-z_2)(u_1-u_2)\dd x\notag\\
		&\quad-\int_\Om(z_1-K(u_1))(A(u_1)-A(u_2))(u_1-u_2)\dd x\notag\\
		&\quad+\int_\Om\underbrace{A(u_2)}_{\geq A_*>0}\underbrace{(K(u_1)-K(u_2))(u_1-u_2)}_{\geq A_*(u_1-u_2)^2\geq0}\dd x\notag\\
		&\quad+\int_\Om(z_1-K(u_1))\Bigparen{\frac{1}{\tau(v_1)}-\frac{1}{\tau(v_2)}}(z_1-z_2)\dd x\notag\\
		&\quad-\int_\Om\frac{1}{\tau(v_2)}(K(u_1)-K(u_2))(z_1-z_2)\dd x+\int_\Om\frac{1}{\tau(v_2)}(z_1-z_2)^2\dd x\notag\\
		&\geq\norm{u_1-u_2}_{\Hav^1(\Om)}^2-A^*\norm{z_1-z_2}_{\LL^2(\Om)}\norm{u_1-u_2}_{\LL^2(\Om)}\notag\\
		&\quad-\bigparen{\beta+\norm{A'}_\infty\norm{z_1-K(u_1)}_{\LL^\infty(\Om)}}\norm{u_1-u_2}_{\LL^2(\Om)}^2\notag\\
		&\quad-\frac{\norm{\tau'}_\infty}{\tau_*^2}\norm{z_1-K(u_1)}_{\LL^\infty(\Om)}\norm{v_1-v_2}_{\LL^2(\Om)}\norm{z_1-z_2}_{\LL^2(\Om)}\notag\\
		&\quad-\frac{A^*}{\tau_*}\norm{u_1-u_2}_{\LL^2(\Om)}\norm{z_1-z_2}_{\LL^2(\Om)},
	\end{align}
	where we used that $A$, $K$ and $\frac{1}{\tau}$ are Lipschitz continuous with respective Lipschitz constants $\norm{A'}_\infty$, $A^*$ and $\frac{\norm{\tau'}_\infty}{\tau_*^2}$. Applying Young's inequality several times and triangular inequality $\norm{v_1-v_2}_{\LL^2(\Om)}\leq\norm{v_1-u_1}_{\LL^2(\Om)}+\norm{u_1-u_2}_{\LL^2(\Om)}+\norm{v_2-u_2}_{\LL^2(\Om)}$ to \eqref{EQ:subdifferential_monotonicity_computation}, one obtains a constant $C=C(\beta,A^*,\tau_*)\in[1,\infty)$ satisfying
	\begin{multline*}
		\bigparen{\bbK(\bfv_1)\delta\calE(\bfu_1)-\bbK(\bfv_2)\delta\calE(\bfu_2),\bfu_1-\bfu_2}_H\\
		\geq\norm{u_1-u_2}_{\Hav^1(\Om)}^2-C\bigparen{1+\norm{A'}_\infty\norm{z_1-K(u_1)}_{\LL^\infty(\Om)}}\norm{u_1-u_2}_{\LL^2(\Om)}^2\\
        -C(1+\norm{\tau'}_\infty\norm{z_1-K(u_1)}_{\LL^\infty(\Om)})^2\norm{z_1-z_2}_{\LL^2(\Om)}^2\\
		-\norm{\tau'}_\infty^2\norm{z_1-K(u_1)}_{\LL^\infty(\Om)}^2\bigparen{\norm{v_2-u_2}_{\LL^2(\Om)}^2+\norm{v_1-u_1}_{\LL^2(\Om)}^2}.
	\end{multline*}
	Eventually, applying \eqref{EQ:L2_interpolation} and Young's inequality once more, we arrive at the desired estimate \eqref{EQ:subdifferential_monotonicity}.
\end{proof}

\begin{lemma}\label{LEM:W_estimate}
	There exists a constant $C_2=C_2(\Om,d,A^*,\beta)\in(0,\infty)$ such that
	\begin{equation}\label{EQ:W_estimate}
		\norm{u-m}_{\HH^2(\Om)}\leq C_2\bigparen{\norm{(u-m,z)}_H+\norm{\delta\calE(\bfu)}_{H^*}}
	\end{equation}
	for all $\bfu=(u,z)\in\dom(\subDG\calE)$.
\end{lemma}

\begin{proof}
	From the definition of $\delta\calE(\bfu)$, i.e.\ \eqref{EQ:subdifferential_delta_definition}, it holds
	\begin{equation}\label{EQ:W_estimate_subgradient_norm}
		\norm{\delta\calE(\bfu)}_{H^*}^2=\int_\Om\Bigparen{\abs{\nabla(-\Delta u+f(u)-A(u)(z-K(u)))}^2+(z-K(u))^2}\dd x.
	\end{equation}
	Thanks to \eqref{EQ:H2_Delta_estimate}, it suffices to show \eqref{EQ:W_estimate} with left-hand side $\norm{\Delta u}_{\LL^2(\Om)}$. Multiplying $-\Delta u+f(u)-A(u)(z-K(u))$ with $-\Delta u$ and integrating by parts twice, we get, by making use of $\nabla u\cdot\n=0$ on $\boundary{\Om}$ and $\nabla(f(u))=f'(u)\nabla u$ (as $f\in\CC^1(\R)$ and $u\in\HH^2(\Om)\hookrightarrow\LL^\infty(\Om)$),
	\begin{align*}
		\int_\Om\abs{\Delta u}^2\dd x&=\int_\Om\nabla(-\Delta u+f(u)-A(u)(z-K(u)))\cdot\nabla u\dd x\\
		&\quad+\int_\Om f(u)\cdot\Delta u\dd x-\int_\Om A(u)(z-K(u))\cdot\Delta u\dd x\\
		&\leq\norm{\nabla(-\Delta u+f(u)-A(u)(z-K(u)))}_{\LL^2(\Om;\R^d)}\norm{u-m}_{\Hav^1(\Om)}\\
		&\quad-\int_\Om f'(u)\abs{\nabla u}^2\dd x+A^*\norm{\Delta u}_{\LL^2(\Om)}\norm{(z-K(u))}_{\LL^2(\Om)}.
	\end{align*}
	Hence, as $f'\geq-\beta$, Young's inequality provides
	\begin{align*}
		\norm{\Delta u}_{\LL^2(\Om)}^2&\leq\norm{\nabla(-\Delta u+f(u)-A(z-K(u)))}_{\LL^2(\Om;\R^d)}^2\\
		&\quad+(1+2\beta)\norm{u-m}_{\Hav^1(\Om)}^2+(A^*)^2\norm{z-K(u)}_{\LL^2(\Om)}^2\\
		&\leq\max\{1,(A^*)^2\}\norm{\bfmu}_{H^*}^2+(1+2\beta)C\norm{\Delta u}_{\LL^2(\Om)}^{4/3}\norm{u-m}_\Havd^{2/3},
	\end{align*}
	where the last estimate follows from \eqref{EQ:H1_interpolation}, \eqref{EQ:H2_Delta_estimate} and \eqref{EQ:W_estimate_subgradient_norm}. Since $\norm{u-m}_\Havd\leq\norm{(u-m,z)}_H$, the desired estimate \eqref{EQ:W_estimate} finally follows after one more application of Young's inequality.
\end{proof}

\begin{proposition}[Closedness of $\subDG\calE$]\label{PRO:closedness_subdifferential_energy}
	The subdifferential $\subDG\calE$ is strong--weak closed, i.e.\ if $(\bfu_n)_{n\in\N}\subset\dom(\subDG\calE)$, $\bfu\in\calH_m$ and $\bfmu\in H^*$ are such that $\bfu_n\to\bfu$ in $\calH_m$ and $\delta\calE(\bfu_n)\weakto\bfmu$ in $H^*$ as $n\to\infty$, then $\bfmu=\delta\calE(\bfu)$.
\end{proposition}

\begin{proof}
	In the following, we write $(u_n,z_n)=\bfu_n$, $(u,z)=\bfu$. From $\bfu_n\to\bfu$ in $\calH_m$ we immediately obtain $z_n\to z$ in $\LL^2(\Om)$. Moreover, since $(\delta\calE(\bfu_n))_{n\in\N}$ and $(\bfu_n)_{n\in\N}$ are bounded in $H^*$ resp.\ $\calH_m$, estimate \eqref{EQ:W_estimate} provides a uniform $\HH^2(\Om)$-bound for $(u_n)_{n\in\N}$, so that in fact $u\in\HH^2(\Om)$ and, for a non-relabelled subsequence, $u_n\weakto u$ in $\HH^2(\Om)$, $u_n\to u$ in $\LL^{2p}(\Om)$ and almost everywhere in $\Om$. In particular, using $A_*\leq A\leq A^*$ and continuity of $A$, we infer $A(u_n)(z_n-K(u_n))\to A(u)(z-K(u))$ and $K(u_n)\to K(u)$ in $\LL^2(\Om)$. Furthermore, the $p$-growth \eqref{EQ:growth_condition_f} of $f$ ensures that $f(u_n)\to f(u)$ in $\LL^2(\Om)$. Collecting the derived convergences, we thus conclude that $\delta\calE(\bfu_n)\weakto\delta\calE(\bfu)$ in $\LL^2(\Om;\R^2)$. On the other hand, $\delta\calE(\bfu_n)\weakto\bfmu$ follows in $\LL^2(\Om;\R^2)$ by hypothesis. As the weak limit is unique, we conclude $\bfmu=\delta\calE(\bfu)$.
\end{proof}

If we apply Lemma~\ref{LEM:abstract_subdifferential_closedness} to $S_n:=\subDG\calE$ and combine it with Proposition~\ref{PRO:closedness_subdifferential_energy}, we readily deduce the following evolutionary closedness:

\begin{corollary}[Evolutionary closedness of $\subDG\calE$]\label{COR:evolutionary_closedness_subdifferential_energy}
	Let $T<\infty$. The subdifferential $\subDG\calE$ is evolutionary strong--weak closed, i.e.\ if $\bfu,\bfu_n\in\LL^2(0,T;\calH_m),\bfmu,\bfmu_n\in\LL^2(0,T;H^*)$ are such that
	\begin{enumerate}[(1)]
		\item $\bfu_n\to\bfu$ in $\LL^2(0,T;\calH_m)$ as $n\to\infty$,
		\item $\bfmu_n\weakto\bfmu$ in $\LL^2(0,T;H^*)$ as $n\to\infty$ and
		\item $\bfmu_n(t)\in\subDG\calE(\bfu_n(t))$ for almost all $t\in(0,T)$, all $n\in\N$,
	\end{enumerate}
	then it follows $\bfmu(t)\in\subDG\calE(\bfu(t))$ for almost every $t\in(0,T)$.
\end{corollary}

\begin{proposition}[Closedness of $\D_{\bfv}\calR$]\label{PRO:closedness_subdifferential_diss_potential}
	Let $T<\infty$. The differential of $\calR$ is closed in the following sense: If $\bfu,\bfu_n\in\LL^2(0,T;\calH_m)$, $\bfv,\bfv_n\in\LL^2(0,T;H)$ are such that
	\begin{enumerate}[(1)]
		\item $\bfu_n\to\bfu$ in $\LL^2(0,T;\calH_m)$ as $n\to\infty$,
		\item $\bfv_n\weakto\bfv$ in $\LL^2(0,T;H)$ as $n\to\infty$ and
		\item $\sup_{n\in\N}\esssup_{t\in[0,T]}\calE(\bfu_n(t))<\infty$,
	\end{enumerate}
	then $\bbG(\bfu_n)\bfv_n\weakto\bbG(\bfu)\bfv$ in $\LL^2(0,T;H^*)$ for $n\to\infty$.
\end{proposition}

\begin{proof}
	Write $(v_n,y_n)=\bfv_n$, $(u_n,z_n)=\bfu_n$, $(v,y)=\bfv$ and $(u,z)=\bfu$. By definition of $\bbG$, we need to prove $\bfu(t)\in\dom(\calE)$ for almost every $t\in(0,T)$ and
	\[((-\Delta)^{-1}v_n,\tau(u_n)y_n)\weakto((-\Delta)^{-1}v,\tau(u)y)\]
	in $\LL^2(0,T;\Hav^1(\Om)\times\LL^2(\Om))$. From (2) we particularly get $v_n\weakto v$ in $\LL^2(0,T;\Havd)$, thus $(-\Delta)^{-1}v_n\weakto(-\Delta)^{-1}v$ in $\LL^2(0,T;\Hav^1(\Om))$. Moreover, \eqref{EQ:L2_interpolation} yields
	\[\norm{u_n-u}_{\LL^2(0,T;\LL^2(\Om))}^2\leq\underbrace{\norm{u_n-u}_{\LL^2(0,T;\Havd)}}_{\to0\text{ by (1)}}\underbrace{\norm{u_n-u}_{\LL^2(0,T;\Hav^1(\Om))}}_{\text{uniformly bounded by (3)}},\]
	so we infer $u_n\to u$ in $\LL^2(0,T;\LL^2(\Om))$ and -- up to a subsequence (not relabelled) -- almost everywhere in $(0,T)\times\Om$. Therefore, together with $y_n\weakto y$ from (2) and $\abs{\tau(u)}\leq\tau^*$, we infer $\tau(u_n)y_n\weakto\tau(u)y$ in $\LL^2(0,T;\LL^2(\Om))$. This holds true for any subsequence we might have chosen, hence we have convergence along the whole sequence.
\end{proof}

\begin{proposition}[Chain rule]\label{PRO:chain_rule}
	Let $T<\infty$. Assume $\bfu\in\ACloc^2([0,T];\calH_m)\cap\LL^2(0,T;\HH^2(\Om)\times\LL^\infty(\Om))$ and $\bfmu\in\LL^2(0,T;H^*)$ are such that $\bfmu(t)\in\subDG\calE(\bfu(t))$ for almost every $t\in(0,T)$. Then, the function $\calE\circ\bfu:[0,T]\to[0,\infty)$ is absolutely continuous, and for almost all $t\in(0,T)$
	\[\ddd{t}\calE(\bfu(t))=\product{\bfmu(t)}{\dot\bfu(t)}_{H^*,H}.\]
\end{proposition}

\begin{proof}
	As always, $\bfu$ is identified with its continuous representative. Moreover, we fix a representative for $\bfmu$ and put the set of ``good points'' $\Sigma:=\set{t\in(0,T)}{\bfmu(t)\in\subDG\calE(\bfu(t))}$. By assumption, $\Leb^1([0,T]\setminus\Sigma)=0$.\\

	We aim to give a self-contained proof by relying on the subdifferential estimate \eqref{EQ:subdifferential_characterising_estimate} and following largely the strategy from the lecture notes \cite[Theorem 3.12]{Mielke_lectureNotes_2023}, with adaptations where necessary. We start with Step~1 by introducing the arc-length parametrisation $\eta:[0,T]\to[0,L]$ of $\bfu$ with left inverse $\theta$, so that $\whbfu:=\bfu\circ\eta$ is 1-Lipschitz continuous. Accordingly, we define $\whE:=\calE\circ\bfu\circ\eta$ and a quantity $\whbfmu$, which will be a slightly modified version of $\bfmu\circ\eta$. In Step~2 we will show that the ``good properties'' of $\Sigma$ are inherited to $\theta(\Sigma)$ for the reparametrised functions $\whbfu$ and $\whbfmu$. Step~3 is devoted to showing an absolute continuity estimate for $\whE$ on $\theta(\Sigma)$. This will be achieved by employing the subdifferential estimate \eqref{EQ:subdifferential_characterising_estimate} and approximating the integral of $g$ by Riemann sums, for which the assumed $\LL^\infty$ space regularity plays a crucial role. In Step~4 we will then prove that this absolute continuity estimate transfers back to the original variables $\bfu,\bfmu,\calE\circ\bfu$ on $\Sigma$. As a result, we obtain a function $E\in\WW^{1,1}([0,T];\R)$ that agrees with $\calE\circ\bfu$ on $\Sigma$. The task is then to show that $E(t)=\calE(\bfu(t))$ for \textit{any} $t\in[0,T]$, thereby establishing the desired absolute continuity for $\calE\circ\bfu$. Finally, in Step~5 we will identify the weak derivative of $\calE\circ\bfu$.\\

	\emph{Step 1: Reparametrisation to 1-Lipschitz curve.}\\
		The following goes back to \cite[Lemma 1.1.4]{Ambrosio_Gigli_Savare_GradientFlows_2005}. For $L:=\norm{\dot\bfu}_{\LL^1(0,T;H)}$, we consider the absolutely continuous and non-decreasing surjective function
		\[\theta:[0,T]\to[0,L],\qquad\theta(t):=\int_0^t\norm{\dot\bfu(r)}_H\dd r.\]
        Define the left continuous, non-decreasing map
        \[\eta:[0,L]\to[0,T],\qquad\eta(s):=\min\set{t\in[0,T]}{\theta(t)=s},\]
        which, in fact, is right inverse to $\theta$. We make the following observation: If $\eta(\theta(t))<t$, then $\theta$ is constant on $[\eta(\theta(t)),t]$, which means that $\bfu$ is constant on $[\eta(\theta(t)),t]$ as well. Therefore, $\bfu(\eta(\theta(t)))=\bfu(t)$. Hence, the following properties of $\theta$ and $\eta$ hold true
		\begin{equation}\label{EQ:chain_rule_reparametrisation_properties}
			\begin{cases}
				\theta(\eta(s))=s&\text{for all }s\in[0,L],\\
                \eta(\theta(t))\leq t&\text{for all }t\in[0,T],\\
				\bfu(\eta(\theta(t)))=\bfu(t)&\text{for all }t\in[0,T].
			\end{cases}
		\end{equation}
		With that, we define $\whbfu:=\bfu\circ\eta:[0,L]\to\calH_m$, which is 1-Lipschitz continuous as for $0\leq s_0\leq s_1\leq L$
		\[\norm{\whbfu(s_1)-\whbfu(s_0)}_H\leq\int_{\eta(s_0)}^{\eta(s_1)}\norm{\dot\bfu(s)}_H\dd s=\theta(\eta(s_1))-\theta(\eta(s_0))=s_1-s_0,\]
		where the penultimate equality is simply the definition of $\theta$. Using \eqref{EQ:chain_rule_reparametrisation_properties}, we also observe $\bfu=\whbfu\circ\theta$ on $[0,T]$. We also put $\whE:=\calE\circ\bfu\circ\eta$ and define
        $\whbfmu:[0,L]\to H^*$ via
        \[\whbfmu(s):=\begin{cases}
            \bfmu(\eta(s))&\text{if }s\notin D,\\
            \delta\calE(\whbfu(s))&\text{if }s\in D,
        \end{cases}\]
        where $D:=\set{s\in[0,L]}{\eta\text{ discontinuous at }s}$ is at most countable due to monotonicity of $\eta$. Note that $\whbfmu(s)$ is well-defined for $s\in D$ (i.e.\ $\whbfu(s)\in\dom(\subDG\calE)$), because $\bfu$ is constant on $[\eta(s),\eta(s^+)]$ and, since $\Leb^1([\eta(s),\eta(s^+)]\cap\Sigma)>0$, it must be that $\bfu(t)\in\dom(\subDG\calE)$ for all $t\in[\eta(s),\eta(s^+)]$. We also stress that $\whbfmu$ is Bochner measurable as $\theta$ is non-decreasing.

        By the change of variables formula \cite[Theorem 5.8.30]{Bogachev_MeasureTheory_2007} and observing that $\theta$ satisfies the Lusin property~\cite[Definition 3.6.8]{Bogachev_MeasureTheory_2007} since it is absolutely continuous, we find for all $[t_0,t_1]\subseteq[0,T]$
		\begin{align}\label{EQ:chain_rule_change_of_variables_mu}
			\int_{\theta(t_0)}^{\theta(t_1)}\norm{\whbfmu(s)}_{H^*}\dd s&=\int_{\theta(t_0)}^{\theta(t_1)}\norm{(\bfmu\circ\eta)(s)}_{H^*}\dd s=\int_{t_0}^{t_1}\norm{\bfmu(\eta(\theta(t)))}_{H^*}\theta'(t)\dd t\notag\\
			&=\int_{t_0}^{t_1}\norm{\bfmu(\eta(\theta(t)))}_{H^*}\norm{\dot\bfu(t)}_H\dd t=\int_{t_0}^{t_1}\norm{\bfmu(t)}_{H^*}\norm{\dot\bfu(t)}_H\dd t,
		\end{align}
		where the first equality uses that $\Leb^1(D)=0$ (i.e.\ $\whbfmu=\bfmu\circ\eta$ almost everyhwere), and the fourth that $\dot\bfu=0$ almost everywhere in $\set{t\in[0,T]}{\eta(\theta(t))\ne t}$. Similarly, we obtain
		\begin{equation}\label{EQ:chain_rule_change_of_variables_Linfty_term}
			\int_{\theta(t_0)}^{\theta(t_1)}\norm{\whz(s)-K(\whu(s))}_{\LL^\infty(\Om)}\dd s=\int_{t_0}^{t_1}\norm{z(t)-K(u(t))}_{\LL^\infty(\Om)}\norm{\dot\bfu(t)}_H\dd t.
		\end{equation}
		In particular, $\whbfmu\in\LL^1(0,L;H^*)$ and $\whz-K(\whu)\in\LL^1(0,L;\LL^\infty(\Om))$.\\

	\emph{Step 2: $\Leb^1([0,L]\setminus\theta(\Sigma))=0$ and $\whbfmu(s)\in\subDG\calE(\whbfu(s))$ for all $s\in\theta(\Sigma)$.}\\
		Since $[0,T]\setminus\Sigma$ is a null set, the Lusin property tells us $\theta([0,T]\setminus\Sigma)$ is a null set as well. Hence the first claim already follows from the inclusion $[0,L]\setminus\theta(\Sigma)\subseteq\theta([0,T]\setminus\Sigma)$, which holds true by surjectivity, i.e.\ $\theta([0,T])=[0,L]$.\\

		For the second claim, we assume $s\in\theta(\Sigma)$ and take $t\in\Sigma$ such that $\theta(t)=s$. If in addition $s\in D$, then $\whbfmu(s)=\delta\calE(\whbfu(s))\in\subDG\calE(\whbfu(s))$ already holds by definition of $\whbfmu$. Let us consider the case $s\notin D$. If $\eta(s)\ne t$, then $\eta(s)<t$ by \eqref{EQ:chain_rule_reparametrisation_properties} and $\eta(s)<t\leq\eta(s^+)$ since $\theta(t)=s$. In particular $s\in D$, which is a contradiction. Thus, $\eta(s)=t\in\Sigma$, whence
        \[\whbfmu(s)=\bfmu(\eta(s))\in\subDG\calE(\bfu(\eta(s)))=\subDG\calE(\whbfu(s)).\]

    \emph{Step 3: Absolute continuity estimate for $\whE$.}\\
        Our goal is to show that $\abs{\whE(t)-\whE(s)}\leq(1+M)\int_{s}^{t}g(r)\dd r$ for all $s,t\in\theta(\Sigma)$ with $s<t$, where
		\begin{align}\label{EQ:chain_rule_definition_g}
			g(r)&:=\norm{\whbfmu(r)}_{H^*}+\frac{1}{2}\norm{A'}_\infty\norm{\whz(r)-K(\whu(r))}_{\LL^\infty(\Om)},
		\end{align}
		and $M:=\norm{g}_{\LL^1([0,L])}$. Note that $M<\infty$ by Step~1.

        To begin with, we note that for arbitrary $s_0,s_1\in\theta(\Sigma)$ with $s_0<s_1$, Step~2 tells us that estimate \eqref{EQ:subdifferential_characterising_estimate} is applicable. Hence, exploiting the 1-Lipschitz continuity of $\whbfu$, we obtain
        \begin{align}\label{EQ:chain_rule_first_subdifferential_estimate}
			\lambda(\whbfu(s_0))\abs{s_1-s_0}^2-\norm{\whbfmu(s_0)}_{H^*}\abs{s_1-s_0}&\leq\whE(s_1)-\whE(s_0)\notag\\
            &\leq\norm{\whbfmu(s_1)}_{H^*}\abs{s_1-s_0}-\lambda(\whbfu(s_1))\abs{s_1-s_0}^2.
		\end{align}
	    Let now $s,t\in\theta(\Sigma)$ with $s<t$. Then, for a partition $s=s_0<s_1<\dotsc<s_N=t$ with $s_i\in\theta(\Sigma)$, we infer from~\eqref{EQ:chain_rule_first_subdifferential_estimate}
        \begin{align}\label{EQ:chain_rule_partition_telescope_estimate}
			\whE(t)-\whE(s)&=\sum_{j=1}^n\bigparen{\whE(s_j)-\whE(s_{j-1})}\notag\\
			&\leq\sum_{j=1}^n\Bigparen{\norm{\whbfmu(s_j)}_{H^*}\abs{s_j-s_{j-1}}-\lambda(\whbfu(s_j))\abs{s_j-s_{j-1}}^2}\notag\\
			&\leq\sum_{j=1}^n\norm{\whbfmu(s_j)}_{H^*}\abs{s_j-s_{j-1}}+\frac{\beta^2}{4}\sum_{j=1}^n(s_j-s_{j-1})^2\notag\\
			&\quad+\biggparen{\sum_{j=1}^n\frac{1}{2}\norm{A'}_\infty\norm{\whz(s_j)-K(\whu(s_j))}_{\LL^\infty(\Om)}(s_j-s_{j-1})}^2\notag\\
			&\leq\biggparen{1+\sum_{j=1}^n(s_j-s_{j-1})g(s_j)}\sum_{j=1}^n(s_j-s_{j-1})g(s_j)+\frac{\beta^2}{4}\sum_{j=1}^n(s_j-s_{j-1})^2,
		\end{align}
		where we have used $\norm{\cdot}_{\ell^2}^2\leq\norm{\cdot}_{\ell^1}^2$ in the penultimate inequality.

		Using Lemma~\ref{LEM:lebesgue_riemann_approximation} for approximating the Lebesgue integral by Riemann sums with partition points from a set of full measure, we get a sequence of partitions $s=s_0^n<s_1^n<\dotsc<s_{N(n)}^n=t$, indexed by $n\in\N$, such that $\max_{1\leq j\leq N(n)}\abs{s_j^n-s_{j-1}^n}\to0$ as $n\to\infty$, $s_j^n\in\theta(\Sigma)$ for every $j$, $n$ and
		\[\lim_{n\to\infty}\sum_{j=1}^{N(n)}(s_j^n-s_{j-1}^n)g(s_j^n)=\int_s^tg(r)\dd r.\]
		Hence, applying \eqref{EQ:chain_rule_partition_telescope_estimate} to these partitions $\{s_0^n,\dotsc,s_{N(n)}^n\}$, we obtain in the limit $n\to\infty$
		\[\whE(t)-\whE(s)\leq(1+\norm{g}_{\LL^1(s,t)})\int_s^tg(r)\dd r.\]
		Using the lower bound in \eqref{EQ:chain_rule_first_subdifferential_estimate}, one shows $\whE(t)-\whE(s)\geq-(1+\norm{g}_{\LL^1(s,t)})\int_s^tg(r)\dd r$ very analogously. Thus, in combination, we deduce $\abs{\whE(t)-\whE(s)}\leq(1+M)\int_s^tg(r)\dd r$.\\

	\emph{Step 4: Absolute continuity of $\calE\circ\bfu$.}\\
		Let $t_0,t_1\in\Sigma$, $t_0<t_1$. Then $\theta(t_0),\theta(t_1)\in\theta(\Sigma)$, hence, by \eqref{EQ:chain_rule_reparametrisation_properties}, Step~3, \eqref{EQ:chain_rule_change_of_variables_mu} and \eqref{EQ:chain_rule_change_of_variables_Linfty_term},
		\begin{align*}
			\abs{\calE(\bfu(t_1))-\calE(\bfu(t_0))}&=\bigabs{\calE\bigparen{\bfu\bigparen{\eta(\theta(t_1))}}-\calE\bigparen{\bfu\bigparen{\eta(\theta(t_0))}}}=\bigabs{\whE(\theta(t_1))-\whE(\theta(t_0))}\\
			&\leq(1+M)\int_{\theta(t_0)}^{\theta(t_1)}\Bigparen{\norm{\whbfmu(t)}_{H^*}+\frac{1}{2}\norm{A'}_\infty\norm{\whz(t)-K(\whu(t))}_{\LL^\infty(\Om)}}\dd t\\
			&=(1+M)\int_{t_0}^{t_1}\norm{\bfmu(t)}_{H^*}\norm{\dot\bfu(t)}_H\dd t\\
			&\quad+\frac{1}{2}(1+M)\norm{A'}_\infty\int_{t_0}^{t_1}\norm{z(t)-K(u(t))}_{\LL^\infty(\Om)}\norm{\dot\bfu(t)}_H\dd t.
		\end{align*}
		Since $\Leb^1([0,T]\setminus\Sigma)=0$, we therefore find $E\in\WW^{1,1}([0,T];\R)$ such that $E(t)=\calE(\bfu(t))$ for all $t\in\Sigma$. Now, we want to argue similarly as in \cite[Theorem 1.2.5]{Ambrosio_Gigli_Savare_GradientFlows_2005} and \cite[Theorem 3.12, Step 2]{Mielke_lectureNotes_2023} to show $E(t)=\calE(\bfu(t))$ for all $t\in[0,T]$, albeit with slight adaptations. The actual difference is that $\bfu$ is not 1-Lipschitz continuous but certainly $\frac{1}{2}$-H\"older continuous
		\begin{equation}\label{EQ:chain_rule_bfu_hoelder_continuity}
			\norm{\bfu(t)-\bfu(s)}_H\leq\int_s^t\norm{\dot\bfu(r)}_H\dd r\leq\norm{\dot\bfu}_{\LL^2(s,t;H)}\abs{t-s}^{1/2},\quad 0\leq s\leq t\leq T,
		\end{equation}
		and that $\bfmu$ is $\LL^2$-integrable in time, which, in the end, turns out to be enough to carry out the argument.\\

		Lower semicontinuity of $\calE$ already gives $\calE(\bfu(t))\leq E(t)$ for all $t\in[0,T]\setminus\Sigma$. For the reverse inequality, we restrict to $t\in[T/2,T]\setminus\Sigma$ and consider the averages for $r\in[0,T/2]$
		\[E_r(t):=\frac{1}{r}\int_{t-r}^t\calE(\bfu(s))\dd s=\frac{1}{r}\int_{t-r}^tE(s)\dd s\xrightarrow{r\searrow0}E(t),\]
		where the equality is justified since $\set{s}{\calE(\bfu(s))\ne E(s)}$ is a null set, and the limit $r\searrow0$ because $E$ is continuous. Hence, in order to get $\calE(\bfu(t))\geq E(t)$, it suffices to show $\calE(\bfu(t))\geq\lim_{r\searrow0}E_r(t)$. To this end, we compute by applying \eqref{EQ:subdifferential_characterising_estimate}, \eqref{EQ:chain_rule_bfu_hoelder_continuity} and using the fact that $\abs{t-s}\leq r$ for $s\in[t-r,t]$,
		\begin{align*}
			\calE(\bfu(t))-E_r(t)&=\frac{1}{r}\int_{t-r}^t\bigparen{\calE(\bfu(t))-\calE(\bfu(s))}\dd s\\
			&\geq\frac{1}{r}\int_{t-r}^t\Bigparen{\bigproduct{\bfmu(s)}{\bfu(t)-\bfu(s)}_{H^*,H}+\lambda(\bfu(s))\norm{\bfu(t)-\bfu(s)}_H^2}\dd s\\
			&\geq\frac{1}{r}\int_{t-r}^t\Bigparen{-\norm{\bfmu(s)}_{H^*}\norm{\dot\bfu}_{\LL^2(s,t;H)}\abs{t-s}^{1/2}+\lambda(\bfu(s))\norm{\dot\bfu}_{\LL^2(s,t;H)}^2\abs{t-s}}\dd s\\
			&\geq-\frac{1}{\sqrt{r}}\norm{\dot\bfu}_{\LL^2(t-r,t;H)}\int_{t-r}^t\norm{\bfmu(s)}_{H^*}\dd s+\norm{\dot\bfu}_{\LL^2(t-r,t;H)}^2\int_{t-r}^t\lambda(\bfu(s))\dd s\\
			&\geq-\norm{\dot\bfu}_{\LL^2(t-r,t;H)}\norm{\bfmu}_{\LL^2(t-r,t;H^*)}+\norm{\dot\bfu}_{\LL^2(t-r,t;H)}^2\norm{\lambda(\bfu)}_{\LL^1(t-r,t)},
		\end{align*}
		where we used $\norm{\bfmu}_{\LL^1(t-r,t;H^*)}\leq\sqrt{r}\norm{\bfmu}_{\LL^2(t-r,t;H^*)}$ in the last inequality. As $\bfmu\in\LL^2(0,T;H^*)$ and $\lambda(\bfu)=-\frac{1}{8}(\beta+\norm{A'}\norm{z-K(u)}_{\LL^\infty(\Om)})^2\in\LL^1(0,T)$, the right-hand side vanishes in the limit $r\searrow0$ by absolute continuity of the Lebesgue integral. This shows the desired estimate $\calE(\bfu(t))\geq\lim_{r\searrow0}E_r(t)$ in the case $t\in[T/2,T]$. For $t\in[0,T/2]$, we instead consider $E_r(t)=\frac{1}{r}\int_t^{t+r}\calE(\bfu(s))\dd s$ and the analogous argument yields the same result in the end.\\

	\emph{Step 5: Identification of derivative $\ddd{t}\calE(\bfu)=\product{\bfmu}{\dot\bfu}_{H^*,H}$.}\\
		Let $\Sigma_0$ denote the set of points $t\in(0,T)$ where $\bfu$ and $\calE\circ\bfu$ are differentiable, $\bfmu(t)=\delta\calE(\bfu(t))$, and $u(t),z(t)\in\LL^\infty(\Om)$. By hypothesis and Step~4, we have $\Leb^1([0,T]\setminus\Sigma_0)=0$. For any $t\in\Sigma_0$, it then holds for all $h\in[-t,T-t]$ that
		\[\calE(\bfu(t+h))-\calE(\bfu(t))\geq\product{\bfmu(t)}{\bfu(t+h)-\bfu(t)}_{H^*,H}+\lambda(\bfu(t))\norm{\bfu(t+h)-\bfu(t)}_H^2.\]
		Dividing this by $h>0$ and taking the limit $h\searrow0$, the differentiability of $\calE\circ\bfu$ and $\bfu$ yield
		\[\ddd{t}\calE(\bfu(t))\geq\product{\bfmu(t)}{\dot\bfu(t)}_{H^*,H}+0.\]
		Finally, dividing by $h<0$ and taking the limit $h\nearrow0$, we obtain the reverse inequality.
\end{proof}

\begin{remark}\label{REM:chain_rule}
	We emphasise that the abstract chain rule results \cite[Proposition 2.4]{Mielke_Rossi_Savare_2013} and \cite[Appendix A, Proposition A.1]{Mielke_Rossi_2023} are not directly applicable. As in our case, their proofs rely on a subgradient estimate of the form
	\[\calE(\bfv)\geq\calE(\bfu)+\product{\bfmu}{\bfv-\bfu}_{H^*,H}+\rho(\bfu,\bfv)\norm{\bfv-\bfu}_H.\]
	In Mielke et al., the modulus of subdifferentiability $\rho:\calH_m\times\calH_m\to[0,\infty)$ is assumed to be upper semicontinuous, which guarantees $\sup_{s,t\in[0,T]}\rho(\bfu(s),\bfu(t))<\infty$ for continuous curves $\bfu:[0,T]\to\calH_m$. In our setting, however, the modulus takes the form
    \[\rho(\bfu,\bfv)=\lambda(\bfu)\norm{\bfu-\bfv}_H\]
    and is not upper semicontinuous.

    In fact, the arguments in Mielke et al.\ would still apply if $\lambda(\whbfu)\in\LL^1(0,T)$ for the arc-length reparametrisation $(\whu,\whz)$. Yet, as observed in Step~1, $(\whu,\whz)$ in general only belongs to $\LL^1(0,T;\LL^\infty(\Om;\R^2))$, which is not sufficiently regular.
\end{remark}

\section{Existence, uniqueness, and stability estimate}\label{SEC:well_posedness}
This section is devoted to the proofs of Theorems~\ref{TH:well_posedness}~and~\ref{TH:existence_extended_initials}.

\subsection{Proof of Theorem \ref{TH:well_posedness}}\label{SUBSEC:well_posedness}
The existence proof is based on a time-discretisation. Since we also obtain uniqueness (independent of the existence result), it suffices to carry out the construction on finite time horizons $T<\infty$.
For a fixed time step $\kappa=T/N$, $N\in\N$, we will perform time-incremental minimisation, providing a family $(\bfu_n^\kappa)_{n=0}^N$ that satisfies the (semi-)implicit
Euler approximation of~\eqref{EQ:gfe_general_definition}
\[
	-\bbG(\bfu_n^\kappa)\frac{\bfu_{n+1}^\kappa-\bfu_n^\kappa}{\kappa}\in\subDG\calE(\bfu_{n+1}^\kappa),
\]
which arises as the Euler--Langrange (EL) equation of a suitable minimisation problem (cf.~\eqref{EQ:direct_problem_approximants} below).
For the associated piecewise affine and piecewise constant left- resp.\ right-continuous interpolants $(\whbfu_\kappa)_\kappa$, $(\olbfu_\kappa)_\kappa$, $(\ulbfu_\kappa)_\kappa$, we then derive suitable uniform bounds as well as a Cauchy estimate that builds on Proposition~\ref{PRO:monotonicity_subdifferential} and Lemma~\ref{LEM:W_estimate}.
This yields the strong convergence of the approximants to a unique limit $\bfu$ and a suitable weak convergence of the discrete time derivatives. Applying the closedness properties Corollary~\ref{COR:evolutionary_closedness_subdifferential_energy} and Proposition~\ref{PRO:closedness_subdifferential_diss_potential}, we will pass to the limit in the approximate equation $-\bbG(\ulbfu_\kappa)\dot{\whbfu}_\kappa\in\subDG\calE(\olbfu_\kappa)$ formulated in terms of the interpolants to derive \eqref{EQ:gfe_general_definition}.
Eventually, we show the regularity $\bfu\in\Lloc^2([0,\infty);\HH^2(\Om))\times\Hloc^1([0,\infty);\LL^\infty(\Om))$, the energy-dissipation balance~\eqref{EQ:edb} and the stability estimate~\eqref{EQ:stability_estimate}.\\

Following our roadmap, we will start constructing discrete approximants satisfying (EL) by solving the minimisation problem
\begin{equation}\label{EQ:direct_problem_approximants}
	\min_{\bfu\in\calH_m}\Phi_\kappa^{\calE,\bbG}(\bfv;\bfu):=\calE(\bfu)+\frac{\kappa}{2}\calR\Bigparen{\bfv;\frac{1}{\kappa}(\bfu-\bfv)}.
\end{equation}

\begin{proposition}[Solvability of the direct problem]\label{PRO:solvability_direct_problem}
	For all $\kappa\in(0,\infty)$ and $\bfv\in\dom(\calE)$, minimisation problem~\eqref{EQ:direct_problem_approximants} admits a solution $\bfu_*\in\calH_m$, and any minimiser $\bfu_*$ fulfils the following (EL) equation (with $\delta\calE(\bfu_*)$ as defined in~\eqref{EQ:subdifferential_delta_definition})
	\begin{equation}\label{EQ:euler_lagrange_equation}
		-\frac{1}{\kappa}\bbG(\bfv)(\bfu_*-\bfv)=\delta\calE(\bfu_*).
	\end{equation}
\end{proposition}

\begin{proof}
	Existence of minimisers follows from the direct method of the calculus of variations: $\Phi_\kappa^{\calE,\bbG}(\bfv;\cdot)$ is coercive due to \eqref{EQ:dissipation_norm_equivalence} and $\calE\geq0$; weak lower semicontinuity holds since $\calE$ is weakly lower semicontinuous by Proposition~\ref{PRO:energy_weakly_lower_semicontinuous} and the dissipation potential in the second argument is as well. Via the subdifferential inclusion, any minimiser $\bfu_*$ fulfils
	\[0\in\subD\Phi_\kappa^{\calE,\bbG}(\bfv;\bfu_*)\subseteq\kappa\D\calR\Bigparen{\bfv;\frac{1}{\kappa}(\bfu_*-\bfv)}+\subDG\calE(\bfu_*),\]
	where here $\D\calR$ denotes the Fr\'echet differential with respect to the second variable of $\calR$. Inserting the identities $\D\calR(\bfv;\frac{1}{\kappa}(\bfu_*-\bfv))=\frac{1}{\kappa^2}\bbG(\bfv)(\bfu_*-\bfv)$ and $\subDG\calE(\bfu_*)=\{\delta\calE(\bfu_*)\}$, the claim follows.
\end{proof}

Starting with an initial condition $\bfu^0=(u^0,q^0)\in(\Hav^1(\Om)+m)\times\LL^\infty(\Om)$ and a time step $\kappa=T/N$, $N\in\N$, we construct $\bfu_0^\kappa,\dotsc,\bfu_N^\kappa\in\calH_m$ iteratively as follows: We put $\bfu_0^\kappa:=\bfu^0$ and, assuming $\bfu_0^\kappa,\dotsc,\bfu_n^\kappa$ have been constructed for some $n<N$, we define $\bfu_{n+1}^\kappa$ as the minimiser of $\Phi_\kappa^{\calE,\bbG}(\bfu_n^\kappa;\cdot)$.
\begin{subequations}
	Based on the discrete minimisers, we define our approximate solutions, i.e.\ the left-continuous piecewise constant interpolant
	\begin{equation}
		\olbfu_\kappa:[0,T]\to\calH_m,\qquad\olbfu_\kappa(t):=\bfu_{\lceil\frac{t}{\kappa}\rceil}^\kappa,
	\end{equation}
	the right-continuous piecewise constant interpolant
	\begin{equation}
		\ulbfu_\kappa:[0,T]\to\calH_m,\qquad\ulbfu_\kappa(t):=\bfu_{\lfloor\frac{t}{\kappa}\rfloor}^\kappa,
	\end{equation}
	and the absolutely continuous, piecewise affine interpolant
	\begin{equation}\label{EQ:piecewise_affine_interpolant}
		\whbfu_\kappa:[0,T]\to\calH_m,\qquad\whbfu_\kappa(t):=\bigparen{1-\tfrac{t}{\kappa}+\lfloor\tfrac{t}{\kappa}\rfloor}\bfu_{\lfloor\frac{t}{\kappa}\rfloor}^\kappa+\bigparen{\tfrac{t}{\kappa}-\lfloor\tfrac{t}{\kappa}\rfloor}\bfu_{\lceil\frac{t}{\kappa}\rceil}^\kappa.
	\end{equation}
	By \eqref{EQ:euler_lagrange_equation}, the interpolants $\olbfu_\kappa,\ulbfu_\kappa,\whbfu_\kappa$ satisfy the Euler approximation
	\begin{equation}\label{EQ:Euler_approximation}
		-\bbG(\ulbfu_\kappa)\dot{\whbfu}_\kappa\in\subDG\calE(\olbfu_\kappa)\quad\text{a.e. in }(0,T).\tag{\textsf{EL}}
	\end{equation}
    We also set $\bfmu_\kappa:=-\bbG(\ulbfu_\kappa)\dot{\whbfu}_\kappa$, so that $\bfmu_\kappa(t)=\delta\calE(\olbfu_\kappa(t))$ for all but finitely many $t\in[0,T]$.
\end{subequations}

\begin{lemma}[A priori estimates]\label{LEM:a_priori_estimates}
	There exists a constant $C_3\in(0,\infty)$, independent of $\kappa$, such that
	\begin{enumerate}
		\item\label{ITEM:a_priori_time_derivative} $\norm{\dot{\whbfu}_\kappa}_{\LL^2(0,T;H)}\leq C_3$;
		\item\label{ITEM:a_priori_interpolant_difference} $\norm{\whbfu_\kappa-\olbfu_\kappa}_{\LL^\infty(0,T;H)}+\norm{\olbfu_\kappa-\ulbfu_\kappa}_{\LL^\infty(0,T;H)}\leq C_3\sqrt{\kappa}$;
		\item\label{ITEM:a_priori_Cauchy_estimate} $\norm{\whbfu_{\kappa_1}-\whbfu_{\kappa_2}}_{\CC([0,T];H)}\leq C_3(\sqrt{\kappa_1}+\sqrt{\kappa_2})$.
	\end{enumerate}
\end{lemma}

\begin{proof}
    The proof is split into three steps. In Step~1 we will derive estimates \ref{ITEM:a_priori_time_derivative} and \ref{ITEM:a_priori_interpolant_difference}, which essentially follow from the minimising property of the $\bfu_n^\kappa$'s. In Step~2, which prepares for the subsequent step, we will show that the family $(\olbfu_\kappa)_\kappa\subset\LL^2(0,T;\LL^\infty(\Om;\R^2))$ is bounded. Finally, in Step~3 we will establish the remaining estimate \ref{ITEM:a_priori_Cauchy_estimate} by exploiting the monotonicity property~\eqref{EQ:subdifferential_monotonicity} of $\subDG\calE$.

    Without loss of generality, we henceforth assume that $\tau_*\leq1$ and $\tau^*\geq1$.\\

	\emph{Step 1: Uniform bounds by minimising property.}\\
		By the minimising property, we have for all $1\leq n\leq N$, $N\in\N$,
		\begin{equation}\label{EQ:apriori_energy_increments}
			\frac{1}{\kappa}\calR(\bfu_{n-1}^\kappa;\bfu_n^\kappa-\bfu_{n-1}^\kappa)+\calE(\bfu_n^\kappa)=\Phi_\kappa^{\calE,\bbG}(\bfu_{n-1}^\kappa;\bfu_n^\kappa)\leq\Phi_\kappa^{\calE,\bbG}(\bfu_{n-1}^\kappa;\bfu_{n-1}^\kappa)=\calE(\bfu_{n-1}^\kappa).
		\end{equation}
		This implies $0\leq\calE(\bfu_n^\kappa)\leq\calE(\bfu^0)=:E_0$ for any $0\leq n\leq N$ and, by \eqref{EQ:dissipation_norm_equivalence},
		\begin{align}\label{EQ:apriori_dotwhbfu_L2H}
			\int_0^T\norm{\dot\whbfu_\kappa(t)}_H^2\dd t&=\sum_{n=1}^N\kappa\Bignorm{\frac{1}{\kappa}(\bfu_n^\kappa-\bfu_{n-1}^\kappa)}_H^2\leq\frac{2}{\tau_*}\sum_{n=1}^N\frac{1}{\kappa}\calR(\bfu_{n-1}^\kappa;\bfu_n^\kappa-\bfu_{n-1}^\kappa)\notag\\
			&\leq\frac{2}{\tau_*}\sum_{n=1}^N\bigparen{\calE(\bfu_{n-1}^\kappa)-\calE(\bfu_n^\kappa)}\leq\frac{2E_0}{\tau_*}
		\end{align}
		as a telescope sum. Moreover, \eqref{EQ:dissipation_norm_equivalence} and \eqref{EQ:apriori_energy_increments} yield
		\begin{equation}\label{EQ:apriori_difference_whbfu_olbfu_ulbfu}
			\max\Bigbraces{\norm{\whbfu_\kappa(t)-\olbfu_\kappa(t)}_H,\norm{\olbfu_\kappa(t)-\ulbfu_\kappa(t)}_H}\leq\bignorm{\bfu_{\lfloor\frac{t}{\kappa}\rfloor}^\kappa-\bfu_{\lceil\frac{t}{\kappa}\rceil}^\kappa}_H\leq\sqrt{\kappa}\sqrt{\frac{2E_0}{\tau_*}}
		\end{equation}
		for all $t\in[0,T]$. Thus, we already have established \ref{ITEM:a_priori_time_derivative} and \ref{ITEM:a_priori_interpolant_difference}.

		Later in Step~2, we will need the estimate
		\begin{align}\label{EQ:apriori_bfmu_L2Hdual}
			\int_0^T\norm{\bfmu_\kappa(t)}_{H^*}^2\dd t&=\sum_{n=1}^N\frac{1}{\kappa}\norm{\bbG(\bfu_{n-1}^\kappa)(\bfu_n^\kappa-\bfu_{n-1}^\kappa)}_{H^*}^2\notag\\
			&\leq2\tau^*\sum_{n=1}^N\frac{1}{\kappa}\calR^*\bigparen{\bfu_{n-1}^\kappa;\bbG(\bfu_{n-1}^\kappa)(\bfu_n^\kappa-\bfu_{n-1}^\kappa)}\notag\\
			&=2\tau^*\sum_{n=1}^N\frac{1}{\kappa}\calR(\bfu_{n-1}^\kappa;\bfu_n^\kappa-\bfu_{n-1}^\kappa)\leq2\tau^*E_0,
		\end{align}
		where the second line uses \eqref{EQ:dissipation_dual_norm_equivalence}, and the last inequality follows similarly as in \eqref{EQ:apriori_dotwhbfu_L2H}.\\

	\emph{Step 2: Uniform bound for $\norm{\olbfu_\kappa}_{\LL^2(0,T;\LL^\infty(\Om;\R^2))}$.}\\
		We actually show that $\norm{\olu_\kappa}_{\LL^2(0,T;\HH^2(\Om))}$ and $\norm{\olz_\kappa}_{\LL^\infty(0,T;\LL^\infty(\Om))}$ are uniformly bounded. We first start with $\norm{\olu_\kappa}_{\LL^2(0,T;\HH^2(\Om))}$. Using \eqref{EQ:W_estimate}, we obtain
		\[\int_0^T\norm{\olu_\kappa(t)-m}_{\HH^2(\Om)}^2\dd t\leq C_2^2\int_0^T\bigparen{\norm{\olbfu_\kappa(t)-(m,0)}_H+\norm{\bfmu_\kappa(t)}_{H^*}}^2\dd t\]
		since $\bfmu_\kappa=\delta\calE(\olbfu_\kappa)$ almost everywhere in $(0,T)$. Furthermore, \eqref{EQ:apriori_dotwhbfu_L2H} yields
		\[\norm{\olbfu_\kappa(t)-\bfu^0}_H=\biggnorm{\int_0^{\lceil\frac{t}{\kappa}\rceil}\dot{\whbfu}_\kappa(t)\dd t}_H\leq\sqrt{\frac{2TE_0}{\tau_*}}.\]
		Using this and \eqref{EQ:apriori_bfmu_L2Hdual}, we already find that $\olu_\kappa$ is uniformly bounded in $\LL^2(0,T;\HH^2(\Om))$. In particular, by the Sobolev embedding $\HH^2(\Om)\hookrightarrow\LL^\infty(\Om)$ (possible as $d\leq3$) and $\abs{K(u)}\leq A^*\abs{u}$, we find a constant $C_{3,1}=C_{3,1}(\Om,T,\norm{\bfu^0-(m,0)}_H,E_0,\beta,A^*,\tau^*,\tau_*,m)\in(0,\infty)$ such that for all $\kappa$
		\begin{equation}\label{EQ:apriori_Ku_L2LI}
			\norm{K(\olu_\kappa)}_{\LL^2(0,T;\LL^\infty(\Om))}\leq\CSo A^*\norm{\olu_\kappa}_{\LL^2(0,T;\HH^2(\Om))}\leq C_{3,1}.
		\end{equation}

		For the $\LL^\infty$-bound of $\olz_\kappa$, we use \eqref{EQ:Euler_approximation}, which tells us for all $1\leq n\leq N$, all $\kappa$,
		\[\frac{1}{\kappa}(z_n^\kappa-z_{n-1}^\kappa)=-\frac{1}{\tau(u_{n-1}^\kappa)}z_n^\kappa+\frac{K(u_n^\kappa)}{\tau(u_{n-1}^\kappa)}\]
		in $\LL^2(\Om)$. After some algebraic manipulations, this can be rearranged to
		\[z_n^\kappa=\frac{\tau(u_{n-1}^\kappa)}{\tau(u_{n-1}^\kappa)+\kappa}z_{n-1}^\kappa+\frac{\kappa K(u_n^\kappa)}{\tau(u_{n-1}^\kappa)+\kappa}.\]
		Therefore, we obtain
		\[\norm{z_n^\kappa}_{\LL^\infty(\Om)}\leq\norm{z_{n-1}^\kappa}_{\LL^\infty(\Om)}+\frac{\kappa}{\tau_*}\norm{K(u_n^\kappa)}_{\LL^\infty(\Om)},\]
		and hence, invoking \eqref{EQ:apriori_Ku_L2LI},
		\begin{align*}
			\norm{z_n^\kappa}_{\LL^\infty(\Om)}&=\norm{z^0}_{\LL^\infty(\Om)}+\sum_{l=1}^n\bigparen{\norm{z_l^\kappa}_{\LL^\infty(\Om)}-\norm{z_{l-1}^\kappa}_{\LL^\infty(\Om)}}\\
			&\leq\norm{z^0}_{\LL^\infty(\Om)}+\frac{1}{\tau_*}\sum_{l=1}^n\kappa\norm{K(u_l^\kappa)}_{\LL^\infty(\Om)}\\
			&\leq\norm{z^0}_{\LL^\infty(\Om)}+\frac{1}{\tau_*}\norm{K(\olu_\kappa)}_{\LL^1(0,T;\LL^\infty(\Om))}\leq \norm{z^0}_{\LL^\infty(\Om)}+\frac{\sqrt{T}}{\tau_*}C_{3,1}=:C_{3,2}
		\end{align*}
		where $C_{3,2}=C_{3,2}(\Om,T,\norm{\bfu^0-(m,0)}_H,\norm{z^0}_{\LL^\infty(\Om)},E_0,\beta,A^*,\tau^*,\tau_*,m)\in(0,\infty)$ is independent of $n$ and $\kappa$. This shows that $\olz_\kappa$ is uniformly bounded in $\LL^\infty(0,T;\LL^\infty(\Om))$ by $C_{3,2}$.\\

	\emph{Step 3: Cauchy estimate for $(\whbfu_\kappa)_\kappa$ in $\CC([0,T];\calH_m)$.}\\
		Finally, we show an estimate for $\whbfu_{\kappa_1}-\whbfu_{\kappa_2}$ for two time steps $\kappa_1,\kappa_2$. To this end, we abbreviate $\zeta:=\norm{\dot\whbfu_{\kappa_1}-\dot\whbfu_{\kappa_2}}_H$, which is a function in $\LL^2(0,T)$, and we recall $-\bbK(\ulbfu_{\kappa_i})\bfmu_{\kappa_i}=\dot{\whbfu}_{\kappa_i}$. Then, with Cauchy--Schwarz inequality, \eqref{EQ:apriori_difference_whbfu_olbfu_ulbfu}, the semi-monotonicity of $\subDG\calE$ from Proposition~\ref{PRO:monotonicity_subdifferential} and triangle inequality, we obtain almost everywhere in $(0,T)$
		\begin{align}\label{EQ:apriori_gronwall_preparation}
			\frac{1}{2}\ddd{t}\norm{\whbfu_{\kappa_1}-\whbfu_{\kappa_2}}_H^2&=\bigparen{\dot\whbfu_{\kappa_1}-\dot\whbfu_{\kappa_2},\whbfu_{\kappa_1}-\olbfu_{\kappa_1}}_H+\bigparen{\dot\whbfu_{\kappa_1}-\dot\whbfu_{\kappa_2},\olbfu_{\kappa_2}-\whbfu_{\kappa_2}}_H\notag\\
			&\quad-\bigparen{\bbK(\ulbfu_{\kappa_1})\bfmu_{\kappa_1}-\bbK(\ulbfu_{\kappa_2})\bfmu_{\kappa_2},\olbfu_{\kappa_1}-\olbfu_{\kappa_2}}_H\notag\\
			&\leq\zeta{\textstyle\sqrt{2E_0\tau_*^{-1}}}\bigparen{\sqrt{\kappa_1}+\sqrt{\kappa_2}}+\Lambda(\bfu_{\kappa_1})\norm{\olbfu_{\kappa_1}-\olbfu_{\kappa_2}}_H^2\notag\\
			&\quad+\omega(\bfu_{\kappa_1})\norm{\olu_{\kappa_1}-\ulu_{\kappa_1}}_{\LL^2(\Om)}^2+\omega(\bfu_{\kappa_1})\norm{\olu_{\kappa_2}-\ulu_{\kappa_2}}_{\LL^2(\Om)}^2.
		\end{align}
		For any $i=1,2$, we have with \eqref{EQ:apriori_difference_whbfu_olbfu_ulbfu} and $\kappa_i\leq\sqrt{T}\sqrt{\kappa_i}$
		\begin{align*}
			&3\Lambda(\bfu_{\kappa_1})\norm{\olbfu_{\kappa_i}-\whbfu_{\kappa_i}}_H^2+\omega(\bfu_{\kappa_1})\norm{\olu_{\kappa_i}-\ulu_{\kappa_i}}_{\LL^2(\Om)}^2\\
            &\quad\leq3\Lambda(\bfu_{\kappa_1})\cdot2E_0\tau_*^{-1}\kappa_i+\omega(\bfu_{\kappa_1})\norm{\olu_{\kappa_i}-\ulu_{\kappa_i}}_{\Hav^1(\Om)}\norm{\olu_{\kappa_i}-\ulu_{\kappa_i}}_\Havd\\
			&\quad\leq\Lambda(\bfu_{\kappa_1})\cdot6E_0\tau_*^{-1}\sqrt{T}\sqrt{\kappa_i}+\omega(\bfu_{\kappa_1})\cdot2\sqrt{2E_0}\sqrt{2E_0\tau_*^{-1}}\sqrt{\kappa_i}\\
			&\quad\leq C_{3,3}\sqrt{\kappa_i}(\Lambda(\bfu_{\kappa_1})+\omega(\bfu_{\kappa_1}))
		\end{align*}
        for $C_{3,3}:=1+6E_0\tau_*^{-1}\sqrt{1+T}$. Combining this with \eqref{EQ:apriori_gronwall_preparation}, we arrive at
		\[\ddd{t}\norm{\whbfu_{\kappa_1}-\whbfu_{\kappa_2}}_H^2\leq C_{3,3}\bigparen{\zeta+\Lambda(\bfu_{\kappa_1})+\omega(\bfu_{\kappa_1})}(\sqrt{\kappa_1}+\sqrt{\kappa_2})+3\Lambda(\bfu_{\kappa_1})\norm{\whbfu_{\kappa_1}-\whbfu_{\kappa_2}}_H^2.\]
		By estimate \eqref{EQ:apriori_dotwhbfu_L2H} and Step~2 we find a constant $C_{3,4}\in(0,\infty)$ independent of $\kappa_i$ such that $\norm{\zeta}_{\LL^1(0,T)}+\norm{\Lambda(\bfu_{\kappa_1})}_{\LL^1(0,T)}+\norm{\omega(\bfu_{\kappa_1})}_{\LL^1(0,T)}\leq C_{3,4}$. Eventually, Gr\"onwall's inequality postulates
		\begin{align*}
			\norm{\whbfu_{\kappa_1}(t)-\whbfu_{\kappa_2}(t)}_H^2&\leq\e^{3\norm{\Lambda(\bfu_{\kappa_1})}_{\LL^1(0,t)}}\biggparen{\underbrace{\norm{\whbfu_{\kappa_1}(0)-\whbfu_{\kappa_2}(0)}_H^2}_{=0}\\
            &\qquad\qquad+C_{3,3}(\sqrt{\kappa_1}+\sqrt{\kappa_2})\int_0^t\bigparen{\zeta(s)+\Lambda(\bfu_{\kappa_1}(s))+\omega(\bfu_{\kappa_1}(s))}\dd s}\\
			&\quad\leq C_{3,3}C_{3,4}\e^{3C_{3,4}}\bigparen{\sqrt{\kappa_1}+\sqrt{\kappa_2}}
		\end{align*}
		for all $t\in[0,T]$.
\end{proof}

\begin{theorem}[Limit passage and conclusion]\label{TH:limit_passage}
	\leavevmode
	\begin{enumerate}
		\item\label{ITEM:Ex_limit_passage} There exists $\bfu\in\AClocZ$ such that
		\begin{subequations}\label{EQ:Ex_convergences}
			\begin{empheq}[left=\empheqlbrace]{align}
				\whbfu_\kappa\to\bfu&\text{ in }\CC([0,T];\calH_m),\label{EQ:Ex_whbfu_convergence}\\
				\dot{\whbfu}_\kappa\weakto\dot{\bfu}&\text{ in }\LL^2(0,T;H),\label{EQ:Ex_dotwhbfu_convergence}\\
				\olbfu_\kappa\to\bfu&\text{ in }\LL^\infty(0,T;\calH_m),\label{EQ:Ex_olbfu_convergence}\\
				\ulbfu_\kappa\to\bfu&\text{ in }\LL^\infty(0,T;\calH_m)\label{EQ:Ex_ulbfu_convergence}
			\end{empheq}
		\end{subequations}
		along the whole sequence $\kappa\to0$, and $\bfu$ is a gradient-flow solution of $(\calH_m,\calE,\calR)$.
		\item\label{ITEM:Ex_GFE_properties} The remaining statements of Theorem~\ref{TH:well_posedness} hold true: Any solution $\bfu\in\AClocZ$ with $\bfu(0)\in(\Hav^1(\Om)+m)\times\LL^\infty(\Om)$ satisfies $\bfu\in\LL^2(0,T;\HH^2(\Om))\times\HH^1(0,T;\LL^\infty(\Om))$, the energy-dissipation balance~\eqref{EQ:edb}, and any such two solutions satisfy the stability estimate~\eqref{EQ:stability_estimate}.
	\end{enumerate}
\end{theorem}

\begin{proof}
    \emph{Step 1: Statement \ref{ITEM:Ex_limit_passage}.}\\
		Using Lemma~\ref{LEM:a_priori_estimates}~\ref{ITEM:a_priori_Cauchy_estimate}, we find an element $\bfu=(u,z)\in\CC([0,T];\calH_m)$ such that \eqref{EQ:Ex_whbfu_convergence} holds true along the entire sequence. Applying Lemma~\ref{LEM:a_priori_estimates}~\ref{ITEM:a_priori_time_derivative}, we further deduce that $\bfu\in\AClocZ$ and \eqref{EQ:Ex_dotwhbfu_convergence}. Moreover, Lemma~\ref{LEM:a_priori_estimates}~\ref{ITEM:a_priori_interpolant_difference} together with \eqref{EQ:Ex_whbfu_convergence} immediately yields \eqref{EQ:Ex_olbfu_convergence} and \eqref{EQ:Ex_ulbfu_convergence}.

		Next, we want to show that this $\bfu$ satisfies \eqref{EQ:gfe_general_definition}. To this end, we notice that
		\begin{equation}\label{EQ:limitpassage_energy_bounded}
			\sup_{N\in\N}\esssup_{t\in[0,T]}\calE(\ulbfu_\kappa(t))=\sup_{N\in\N,0\leq n\leq N}\calE(\bfu_n^\kappa)\leq E_0.
		\end{equation}
		Combining \eqref{EQ:Ex_dotwhbfu_convergence}, \eqref{EQ:Ex_ulbfu_convergence} and \eqref{EQ:limitpassage_energy_bounded}, the closedness of $\D_{\bfv}\calR$ from Proposition~\ref{PRO:closedness_subdifferential_diss_potential} implies $\bbG(\ulbfu_\kappa)\dot{\whbfu}_\kappa\weakto\bbG(\bfu)\dot\bfu$ in $\LL^2(0,T;H^*)$. Using this, \eqref{EQ:Ex_olbfu_convergence}, \eqref{EQ:Euler_approximation} and the closedness property of $\subDG\calE$ from Corollary~\ref{COR:evolutionary_closedness_subdifferential_energy}, we conclude that $-\bbG(\bfu(t))\dot\bfu(t)\in\subDG\calE(\bfu(t))$ for almost all $t\in(0,T)$. Eventually, \eqref{EQ:Ex_whbfu_convergence} delivers $\bfu^0=\whbfu_\kappa(0)\to\bfu(0)$.\\

	\emph{Step 2: Any GF solution $\bfu\in\AClocZ$ with $\bfu(0)\in(\Hav^1(\Om)+m)\times\LL^\infty(\Om)$ enjoys the
    regularity $\bfu\in\LL^2(0,T;\HH^2(\Om))\times\HH^1(0,T;\LL^\infty(\Om))$ and satisfies the EDB~\eqref{EQ:edb}.}\\
		We will first show $\bfu\in\LL^2(0,T;\HH^2(\Om))\times\HH^1(0,T;\LL^\infty(\Om))$. Then we are in a position to apply the chain rule in Proposition~\ref{PRO:chain_rule} to $(\bfu,-\bbG(\bfu)\dot\bfu)$. Applying estimate \eqref{EQ:W_estimate}, we get
		\[\int_0^T\norm{u(t)-m}_{\HH^2(\Om)}^2\dd t\leq C_2^2\int_0^T\bigparen{\norm{\bfu(t)-(m,0)}_H+\norm{\bbG(\bfu(t))\dot\bfu(t)}_{H^*}}^2\dd t<\infty,\]
		i.e.\ $u\in\LL^2(0,T;\HH^2(\Om))$. To see $z\in\HH^1(0,T;\LL^\infty(\Om))$, we use \eqref{EQ:gfe_general_definition}, which tells us
		\begin{equation}\label{EQ:z_ODE}
			\dot z=-\frac{1}{\tau(u)}(z-K(u))
		\end{equation}
		in $\LL^2(0,T;\LL^2(\Om))$. Hence, $z$ admits the representation
		\begin{equation}\label{EQ:z_ODE_representation}
			z(t)=\e^{-\int_0^t\frac{1}{\tau(u(s))}\dd s}z(0)+\int_0^t\frac{K(u(s))}{\tau(u(s))}\e^{-\int_s^t\frac{1}{\tau(u(\zeta))}\dd\zeta}\dd s
		\end{equation}
		in $\LL^2(\Om)$ for all $t\in[0,T]$. From there one obtains
		\begin{equation}\label{EQ:z_Linfty_by_ODE_formula}
			\norm{z(t)}_{\LL^\infty(\Om)}\leq\norm{z(0)}_{\LL^\infty(\Om)}+\frac{A^*}{\tau_*}\norm{u}_{\LL^1(0,T;\LL^\infty(\Om))}<\infty
		\end{equation}
		by Sobolev embedding $\HH^2(\Om)\hookrightarrow\LL^\infty(\Om)$. Note that, since $u$ is Bochner-measurable with respect to $\LL^\infty(\Om)$ and $\frac{1}{\tau}$, $K$ are Lipschitz continuous, $z$ is Bochner-measurable with respect to $\LL^\infty(\Om)$, too. Invoking \eqref{EQ:z_Linfty_by_ODE_formula}, we obtain $z\in\LL^\infty(0,T;\LL^\infty(\Om))$. In particular, the right-hand side in \eqref{EQ:z_ODE} is in $\LL^2(0,T;\LL^\infty(\Om))$, hence so is the left-hand side, i.e.\ $\dot z\in\LL^2(0,T;\LL^\infty(\Om))$. Altogether this shows $z\in\HH^1(0,T;\LL^\infty(\Om))$.

		We are now ready to apply the chain rule, which states $\calE\circ\bfu:[0,T]\to[0,\infty)$ is absolutely continuous with $\LL^1$-weak derivative
		\begin{align}\label{EQ:edb_preparation}
			-\ddd{t}\calE(\bfu(t))&=\bigproduct{\bbG(\bfu(t))\dot\bfu(t)}{\dot\bfu(t)}_{H^*,H}\notag\\
			&=\bigproduct{\delta\calE(\bfu(t))}{\bbK(\bfu(t))\delta\calE(\bfu(t))}_{H^*,H}\notag\\
			&=\int_\Om\Bigparen{\bigabs{\nabla\bigparen{-\Delta u+f(u)-A(u)(z-K(u))}}^2+\frac{1}{\tau(u)}(z-K(u))^2}\dd x.
		\end{align}
		From there, \eqref{EQ:edb} follows immediately.\\

	\emph{Step 3: Stability estimate and uniqueness.}\\
		Let $\bfu=(u,z),\bfv=(v,y)\in\AClocZ$ be GF solutions with corresponding initial conditions $\bfu^0,\bfv^0\in(\Hav^1(\Om)+m)\times\LL^\infty(\Om)$. In particular,
		\[-\bbG(\bfu)\dot\bfu=\delta\calE(\bfu)\quad\text{and}\quad\bbG(\bfv)\dot\bfv=\delta\calE(\bfv)\]
		almost everywhere in $(0,T)$. By Step~2 we know $(u,z)\in\LL^2(0,T;\HH^2(\Om))\times\HH^1(0,T;\LL^\infty(\Om))$. Therefore, we can apply the monotonicity from Proposition~\ref{PRO:monotonicity_subdifferential} to obtain
		\[\frac{1}{2}\ddd{t}\norm{\bfu-\bfv}_H^2=\bigparen{\bbK(\bfu)\bbG(\bfu)\dot\bfu-\bbK(\bfv)\bbG(\bfv)\dot\bfv,\bfu-\bfv}_H\leq\Lambda(\bfu)\norm{\bfu-\bfv}_H^2.\]
		An application of Gr\"onwall's inequality provides
		\begin{equation}
			\norm{\bfu(t)-\bfv(t)}_H\leq\exp\bigparen{\norm{\Lambda(\bfu)}_{\LL^1(s,t)}}\norm{\bfu(s)-\bfv(s)}_H
		\end{equation}
		for all $0\leq s\leq t\leq T$, which, by definition of $\Lambda$, is precisely the desired estimate~\eqref{EQ:stability_estimate}.
\end{proof}

\subsection{Proof of Theorem \ref{TH:existence_extended_initials}}\label{SUBSEC:existence_extended_initials}
The proof relies on the previous result. We approximate $\bfu^0\in(\Hav^1(\Om)+m)\times\LL^2(\Om)$ by a sequence of initial values in $(\Hav^1(\Om)+m)\times\LL^\infty(\Om)$, for which Theorem~\ref{TH:well_posedness} guarantees unique gradient-flow solutions satisfying the EDB. Based on this, we show that the approximate solutions are uniformly bounded and thus admit weakly convergent subsequences. Due to the absence of a Cauchy estimate as in Lemma~\ref{LEM:a_priori_estimates}~\ref{ITEM:a_priori_Cauchy_estimate}, we do not obtain the strong convergence in $\CC([0,T];\calH_m)$. As a result, the argument for the closedness of $\subDG\calE$ becomes a bit more delicate. Eventually, we derive the EDI~\eqref{EQ:edi} and the weak--strong stability~\eqref{EQ:ext_stability_estimate}.

\begin{proof}[Proof of Theorem~\ref{TH:existence_extended_initials}]
	Let $\bfu^0\in(\Hav^1(\Om)+m)\times\LL^2(\Om)$. We choose a sequence of initial data $(\bfu_n^0)_{n\in\N}\subset(\Hav^1(\Om)+m)\times\LL^\infty(\Om)$ such that
	\begin{equation}\label{EQ:ext_initial_conditions_convergence}
		\norm{\bfu_n^0-\bfu^0}_{\Hav^1(\Om)\times\LL^2(\Om)}\to0\quad\text{as }n\to\infty.
	\end{equation}
	By Theorem \ref{TH:well_posedness}, we find GF solutions $\bfu_n\in\AClocZ$ that satisfy the EDB, in particular
	\[\calE(\bfu_n(T))+\int_0^T\bigparen{\calR(\bfu_n(t);\dot\bfu_n(t))+\calR^*(\bfu_n(t);-\bfmu_n(t))}\dd t=\calE(\bfu_n^0)<\infty\]
	for all $T\in[0,\infty)$, where $\bfmu_n=-\bbG(\bfu_n)\dot\bfu_n$. Combining \eqref{EQ:ext_initial_conditions_convergence} with the bounded sublevels of $\calE$ in $\HH^1(\Om)\times\LL^2(\Om)$ and using coercivity \eqref{EQ:dissipation_norm_equivalences}, we find a constant $C\in(0,\infty)$ such that for every $n\in\N$
	\begin{equation}\label{EQ:ext_a_priori_by_edb}
		\norm{\bfu_n}_{\LL^\infty([0,\infty);\HH^1(\Om)\times\LL^2(\Om))}+\norm{\dot\bfu_n}_{\LL^2([0,\infty);H)}+\norm{\bfmu_n}_{\LL^2([0,\infty);H)}\leq C.
	\end{equation}
	Now put $\tilde{\mu}_n:=-\Delta u_n+f(u_n)-A(u_n)(z_n-K(u_n))$, so that $\mu_n-\tilde{\mu}_n=\fraka(u_n,z_n)$. Observe that
	\[\abs{\fraka(u_n,z_n)}\leq\int_\Om\abs{f(u_n)}\dd x+A^*\int_\Om\abs{z_n-K(u_n)}\dd x,\]
	i.e.\ $\norm{\fraka(u_n,z_n)}_{\LL^\infty([0,\infty))}\leq C$ by growth condition \eqref{EQ:growth_condition_f} and \eqref{EQ:ext_a_priori_by_edb}. Together with \eqref{EQ:W_estimate} this implies the existence of constants $C(T)\in(0,\infty)$ such that
	\begin{equation}\label{EQ:ext_a_priori_mu_and_H2}
		\norm{\tilde{\mu}_n}_{\LL^2(0,T;\HH^1(\Om))}+\norm{u_n}_{\LL^2(0,T;\HH^2(\Om))}\leq C(T).
	\end{equation}
	\begin{subequations}\label{EQ:ext_convergences}
		Using \eqref{EQ:ext_a_priori_by_edb}, we find a limit $\bfu=(u,z)\in\AClocZ\cap\LL^\infty([0,\infty);\HH^1(\Om)\times\LL^2(\Om))$ such that for a non-relabelled subsequence
		\begin{empheq}[left=\empheqlbrace]{align}
			\dot\bfu_n\weakto\dot\bfu&\text{ in }\LL^2([0,\infty);H),\label{EQ:ext_convergence_dotbfu_H1H}\\
			\bfu_n\starto\bfu&\text{ in }\LL^\infty([0,\infty);\HH^1(\Om)\times\LL^2(\Om)).\label{EQ:ext_convergence_bfu_LIH1L2}
		\end{empheq}
		From \eqref{EQ:ext_a_priori_mu_and_H2} we infer $u\in\Lloc^2([0,\infty);\HH^2(\Om))$ and get some $\tilde{\mu}\in\Lloc^2([0,\infty);\HH^1(\Om))$, such that for further subsequence (again, not-relabelled) it holds
		\begin{empheq}[left=\empheqlbrace]{align}
			u_n\weakto u&\text{ in }\LL^2(0,T;\HH^2(\Om))\quad \forall T<\infty,\label{EQ:ext_convergence_u_L2H2}\\
			u_n\to u&\text{ in }\Lloc^2([0,\infty);\HH^1(\Om)),\label{EQ:ext_convergence_u_L2H1}\\
			u_n\to u&\text{ a.e. in }(0,\infty)\times\Om,\label{EQ:ext_convergence_u_ae}\\
			\tilde{\mu}_n\weakto\tilde{\mu}&\text{ in }\LL^2(0,T;\HH^1(\Om)) \quad \forall T<\infty.\label{EQ:ext_convergence_mu_L2H1}
		\end{empheq}
		Note that \eqref{EQ:ext_convergence_u_L2H1} is implied by Aubin--Lions compactness, while \eqref{EQ:ext_convergence_u_ae} and \eqref{EQ:ext_convergence_mu_L2H1} are obtained via a diagonal argument. We observe that \eqref{EQ:ext_convergence_dotbfu_H1H}, \eqref{EQ:ext_convergence_u_ae} and $\abs{\tau}\leq\tau^*$ imply literally as in Lemma~\ref{PRO:closedness_subdifferential_diss_potential}
		\begin{equation}
			\bbG(\bfu_n)\dot\bfu_n\weakto\bbG(\bfu)\dot\bfu\quad\text{in }\LL^2([0,\infty);H).\label{EQ:ext_convergence_Gudotu}
		\end{equation}
	\end{subequations}
	To conclude $-\bbG(\bfu)\dot\bfu=\delta\calE(\bfu)$, we can not use Proposition~\ref{PRO:closedness_subdifferential_energy} because strong convergence $\bfu_n\to\bfu$ in $\Lloc^2([0,\infty);\calH_m)$ is not available. However, since $-\bbG(\bfu_n)\dot\bfu_n=\bfmu_n=\delta\calE(\bfu_n)$ and weak limits are unique, it is sufficient to verify that
	\begin{align}
		&\tilde{\mu}=-\Delta u+f(u)-A(u)(z-K(u))\text{ and}\label{EQ:ext_mu_limit_passage}\\
		&z_n-K(u_n)\weakto z-K(u)\text{ in }\LL^2(0,T;\LL^2(\Om))\quad\forall T<\infty.\label{EQ:ext_xi_limit_passage}
	\end{align}
	The latter one already follows from \eqref{EQ:ext_convergence_bfu_LIH1L2}, \eqref{EQ:ext_convergence_u_L2H1} and \eqref{EQ:ext_convergence_u_ae}. Note that \eqref{EQ:ext_xi_limit_passage} and \eqref{EQ:ext_convergence_u_ae} also provides $A(u_n)(z_n-K(u_n))\weakto A(u)(z-K(u))$ in $\LL^2(0,T;\LL^2(\Om))$ for any $T<\infty$ since $\abs{A}\leq A^*$. Moreover, \eqref{EQ:ext_convergence_u_L2H2} yields $\Delta u_n\weakto\Delta u$ in $\LL^2(0,T;\LL^2(\Om))$. Hence, once we have $f(u_n)\to f(u)$ in $\Lloc^1([0,\infty);\LL^1(\Om))$ established, the identity \eqref{EQ:ext_mu_limit_passage} follows. To this end, fix $T<\infty$ and observe that \eqref{EQ:ext_convergence_u_L2H1} and \eqref{EQ:ext_convergence_u_ae}, together with \eqref{EQ:growth_condition_f} and Sobolev embedding $\HH^1(\Om)\hookrightarrow\LL^p(\Om)$, imply for a subsequence $\norm{f(u_{n_k})-f(u)}_{\LL^1(\Om)}\to0$ almost everywhere in $(0,T)$. Hence, invoking
	\[\norm{f(u_{n_k}(t))-f(u(t))}_{\LL^1(\Om)}\leq c_1\bigparen{2(\Leb^d(\Om))^1+\CSo\bigparen{\norm{u_{n_k}(t)}_{\HH^1(\Om)}^p+\norm{u(t)}_{\HH^1(\Om)}^p}}\]
	and \eqref{EQ:ext_convergence_bfu_LIH1L2}, the dominated convergence theorem yields $f(u_{n_k})\to f(u)$ in $\LL^1(0,T;\LL^1(\Om))$. Since this applies to any subsequence we might have chosen, it follows $f(u_n)\to f(u)$ in $\LL^1(0,T;\LL^1(\Om))$ as desired. Taking into account that $\bfu(0)\leftharpoonup\bfu_n(0)=\bfu_n^0\to\bfu^0$ in $\calH_m$, $\bfu$ eventually qualifies as a gradient-flow solution.\\

	To derive the EDI~\eqref{EQ:edi}, we pass to the limit in the EDB~\eqref{EQ:edb} satisfied by $\bfu_n$. Since $\bfu_n(t)\weakto\bfu(t)$ in $\calH_m$ for all $t$, it holds $\calE(\bfu(t))\leq\liminf_{n\to\infty}\calE(\bfu_n(t))$ by Proposition~\ref{PRO:energy_weakly_lower_semicontinuous}. In the dissipative terms, we apply lower semicontinuity to the weak convergences \eqref{EQ:ext_convergence_mu_L2H1}, \eqref{EQ:ext_xi_limit_passage} and the identification \eqref{EQ:ext_mu_limit_passage}. Eventually, for the passage in the initial energy, we just observe that \eqref{EQ:ext_initial_conditions_convergence} implies $\calE(\bfu_n^0)\to\calE(\bfu^0)$.\\

	The stability estimate~\eqref{EQ:ext_stability_estimate} can be shown literally as in Step~3 of Theorem \ref{TH:limit_passage}, by noting that the argument needs the $\LL^\infty(\Om)$-regularity only for one of the solutions. The very last statement concerning the bound of $\norm{z-K(u)}_{\LL^2(0,T;\LL^\infty(\Om))}$ can be deduced easily from \eqref{EQ:W_estimate}, the ODE formula~\eqref{EQ:z_ODE_representation}, EDI~\eqref{EQ:edi} and estimates of type $\norm{z}_{\LL^2(\Om)}\leq\sqrt{2\calE(\bfu)}+A^*\norm{u}_{\LL^2(\Om)}$, $\norm{u-m}_\Havd\leq\CPW\norm{u-m}_{\LL^2(\Om)}\leq\CPW^2\norm{\nabla u}_{\LL^2(\Om;\R^d)}\leq\CPW^2\sqrt{2\calE(\bfu)}$ (Poincar\'e-Wirtinger).
\end{proof}

\section{Relaxation limit}\label{SEC:relaxation_limit}
Throughout this section, we suppose Assumption~\ref{ASS:relaxation_limit} in addition to our general hypotheses listed in Assumption~\ref{ASS:general}. Before turning to the proof of Theorem~\ref{TH:relaxation_limit}, we establish the $\Gamma$-convergence of the energies and collect further auxiliary properties.

\subsection{\texorpdfstring{$\boldsymbol{\Gamma}$-}{}Convergence of energies, compactness, and closures}

We will frequently use the following $\LL^2$-convergence result, which is an easy consequence of the uniform convergence $A_\ve\to A$ from Assumption~\ref{ASS:relaxation_limit}:

\begin{lemma}\label{LEM:L2_convergence_with_K}
	Let $T<\infty$. If $u_\ve,u\in\CC([0,T];\LL^2(\Om))$ are such that $u_\ve\to u$ in $\CC([0,T];\LL^2(\Om))$, then $K_\ve(u_\ve)\to K(u)$ and $K_\ve^{-1}(u_\ve)\to K^{-1}(u)$ in $\CC([0,T];\LL^2(\Om))$ as well.

    The same conclusion holds if we replace `` $\CC([0,T];\LL^2(\Om))$'' by `` $\LL^2(\Om)$''.
\end{lemma}

\begin{proposition}[$\Gamma$-limit of $\calE_\ve$]\label{PRO:energies_gamma_convergence}
The family $(\calE_\ve)_{\ve\in(0,1)}$ $\Gamma$-converges with respect to $\calH_m$ to the functional
	\[\calE_0:\calH_m\to[0,\infty],\qquad\calE_0(\bfu):=\begin{cases}
		\calECH(u)&\text{if }z=K(u)\text{ and }u\in\Hav^1(\Om)+m,\\
		+\infty&\text{otherwise},
	\end{cases}\]
    i.e.\ the following holds true:
	\begin{enumerate}
		\item If $\bfu_\ve\to\bfu$ in $\calH_m$, then $\calE_0(\bfu)\leq\liminf_{\ve\to0}\calE_\ve(\bfu_\ve)$.
		\item For any $\bfu\in\calH_m$, we find $(\bfu_\ve)_{\ve\in(0,1)}\subset\calH_m$ with $\bfu_\ve\to\bfu$ and $\limsup_{\ve\to0}\calE_\ve(\bfu_\ve)\leq\calE_0(\bfu)$.
	\end{enumerate}
\end{proposition}

\begin{proof}
	\leavevmode
	\begin{enumerate}
		\item For a suitable subsequence we have $\liminf_{\ve\to0}\calE_\ve(\bfu_\ve)=\lim_{k\to\infty}\calE_\vek(\bfu_\vek)=:E$. We may also assume that $E<\infty$, otherwise there is nothing to do, and $\bfu_\vek\in\HH^1(\Om)\times\LL^2(\Om)$ for all $k$. In particular, $(u_\vek)_{k\in\N}$ is bounded in $\HH^1(\Om)$. Therefore, $u\in\Hav^1(\Om)+m$ and, for a further non-relabelled subsequence, $u_\vek\weakto u$ in $\HH^1(\Om)$, $u_\vek\to u$ in $\LL^2(\Om)$ and $u_\vek(x)\to u(x)$ for almost every $x\in\Om$. Via weak lower semicontinuity of the norm and Fatou's lemma, as in Proposition~\ref{PRO:energy_weakly_lower_semicontinuous}, it follows
		\[\calECH(u)\leq\liminf_{k\to\infty}\calE_\vek(\bfu_\vek)=E.\]
		Since $\norm{z_\ve-K_\ve(u_\ve)}_{\LL^2(\Om)}\leq\ve C$ for some $C\in(0,\infty)$, we have $z_\ve-K_\ve(u_\ve)\to0$ in $\LL^2(\Om)$. Furthermore, $K_\vek(u_\vek)\to K(u)$ in $\LL^2(\Om)$ by Lemma~\ref{LEM:L2_convergence_with_K}, so that $z_\vek\to K(u)$ in $\LL^2(\Om)$, thus $z=K(u)$, i.e.\ $\calECH(u)=\calE_0(\bfu)$.
		\item We may assume $z=K(u)$. Hence, for $\bfu_\ve:=(u,K_\ve(u))$ it holds $\calE_\ve(\bfu_\ve)=\calECH(u)$ and $K_\ve(u)\to K(u)$ in $\LL^2(\Om)$ by Lemma~\ref{LEM:L2_convergence_with_K}, especially $\bfu_\ve\to\bfu$ in $\calH_m$.\qedhere
	\end{enumerate}
\end{proof}

The following compactness property will play a key role in obtaining strong convergence required for the strong--weak closedness of the graph of $\subDG\calE_\ve^\alpha$.

\begin{lemma}\label{LEM:energies_relative_compact}
	For every sequence $(\bfu_\vek)_{k\in\N}\subset\calH_m$, $\vek\searrow0$ fulfilling $\sup_{k\in\N}\calE_\vek(\bfu_\vek)<\infty$, we find a subsequence $(\bfu_{\vek_n})_{n\in\N}\subseteq(\bfu_\vek)_{k\in\N}$ with $\bfu_{\vek_n}\to\bfu_0$ in $\calH_m$ for some $\bfu_0\in\calH_m$.
\end{lemma}

\begin{proof}
	The uniform bound of $\calE_\vek(\bfu_\vek)$ implies, on the one hand, $z_\vek-K_\vek(u_\vek)\to0$ in $\LL^2(\Om)$ as $k\to\infty$, and on the other hand it provides a subsequence $(\bfu_{\vek_n})_{n\in\N}\subseteq(\bfu_\vek)_{k\in\N}$ and some $u_0\in\Hav^1(\Om)+m$ such that $u_{\vek_n}\weakto u_0$ in $\HH^1(\Om)$ as $n\to\infty$. By Rellich-Kondrachov we infer $u_{\vek_n}\to u_0$ in $\LL^2(\Om)$, and whence $K_{\vek_n}(u_{\vek_n})\to K(u_0)$ in $\LL^2(\Om)$ by Lemma \ref{LEM:L2_convergence_with_K}, hence $z_{\vek_n}\to K(u_0)$ in $\LL^2(\Om)$ as well. Altogether we have, as $n\to\infty$,
	\[\norm{\bfu_{\vek_n}-(u_0,K(u_0))}_H^2=\norm{u_{\vek_n}-u_0}_\Havd^2+\norm{z_{\vek_n}-K(u_0)}_{\LL^2(\Om)}^2\to0.\qedhere\]
\end{proof}

\begin{lemma}[Strong--weak closure of $\subDG\calE_\ve$]\label{LEM:limit_subdifferential}
	Let $\bfu_\vek=(u_\vek,z_\vek),\bfu=(u,z)\in\calH_m$, $\bfmu_\vek,\bfmu=(\mu,\xi)\in H^*$. Assume that $\bfmu_\vek\in\subDG\calE_\vek(\bfu_\vek)$ for all $k\in\N$, $\bfu_\vek\to\bfu$ in $\calH_m$ and $\bfmu_\vek\weakto\bfmu$ in $H^*$. Then $\bfmu\in\subDL\calE_0(\bfu)$, where the limit subdifferential is defined as
	\[\subDL\calE_0(\bfu):=\Set{\begin{pmatrix}
			-\Delta u+f(u)-A(u)v-\tilde{\fraka}(u,v)\\
			v
	\end{pmatrix}}{\begin{array}{c}
			v\in\LL^2(\Om),\\
			-\Delta u+f(u)-A(u)v\in\HH^1(\Om)
		\end{array}
	}\]
	for $\bfu=(u,K(u))$ with $u\in\Hav^2(\Om)+m$ such that $\nabla u\cdot\n=0$ on $\boundary{\Om}$. Here, we used $\tilde{\fraka}(u,v):=\avint{_\Om}(f(u)-A(u)v)\dd x$.

	In particular, if $\mu=0$, then $A(u)\xi=-\Delta u+f(u)-\avint{_\Om}(f(u)-A(u)\xi)\dd x$. And if $\xi=0$, then $\mu=-\Delta u+f(u)-\avint{_\Om}f(u)\dd x$.
\end{lemma}

\begin{proof}
	To simplify notation, we write $\ve$ instead of $\vek$. By assumption, we have $\bfmu_\ve\weakto(\mu,\xi)$ in $\Hav^1(\Om)\times\LL^2(\Om)$, which means that
	\[-\Delta u_\ve+f(u_\ve)-\frac{1}{\ve^2}A_\ve(u_\ve)(z_\ve-K_\ve(u_\ve))-\fraka(u_\ve,z_\ve)\weakto\mu\]
	in $\HH^1(\Om)$ and $\frac{1}{\ve^2}(z_\ve-K_\ve(u_\ve))\weakto\xi$ in $\LL^2(\Om)$. The latter particularly implies $z_\ve-K_\ve(u_\ve)\to0$ in $\LL^2(\Om)$. Furthermore, Lemma~\ref{LEM:W_epsilon_estimate} tells us that $(u_\ve)_\ve$ is uniformly bounded in $\HH^2(\Om)$. We conclude $u\in\HH^2(\Om)$ and $u_\ve\weakto u$ in $\HH^2(\Om)$. In particular, since $u_\ve\in\Hav^2(\Om)+m$ and $\nabla u_\ve\cdot\n=0$ on $\boundary{\Om}$, it follows $u\in\Hav^2(\Om)+m$ and $\nabla u\cdot\n=0$ as well. Moreover, for non-relabelled subsequence, Rellich-Kondrachov postulates $u_\ve\to u$ in $\LL^p(\Om)$ and almost everywhere in $\Om$. Hence, since $z_\ve-K_\ve(u_\ve)\to0$ in $\LL^2(\Om)$, Lemma~\ref{LEM:L2_convergence_with_K} yields $z=K(u)$, i.e.\ $\bfu=(u,K(u))$.

	Eventually, it holds $\Delta u_\ve\weakto\Delta u$ in $\LL^2(\Om)$, $f(u_\ve)\to f(u)$ in $\LL^1(\Om)$ by the $p$-growth \eqref{EQ:growth_condition_f} of $f$, $\frac{1}{\ve^2}A_\ve(u_\ve)(z_\ve-K_\ve(u_\ve))\weakto A(u)\xi$ in $\LL^2(\Om)$, because $A_\ve$ converges uniformly to $A$, and
	\[\fraka(u_\ve,z_\ve)=\int_\Om\Bigparen{f(u_\ve)-\frac{1}{\ve^2}A_\ve(u_\ve)(z_\ve-K_\ve(u_\ve))}\dd x\to\int_\Om(f(u)-A(u)\xi)=\tilde{\fraka}(u,\xi),\]
	so that $\mu=-\Delta u+f(u)-A(u)\xi-\tilde{\fraka}(u,\xi)$ is identified by uniqueness of weak limits. This shows
	\[\bfmu=\begin{pmatrix}
		-\Delta u+f(u)-A(u)\xi-\tilde{\fraka}(u,\xi)\\
		\xi
	\end{pmatrix}\in\subDL\calE_0(\bfu)\]
	as desired. The special cases $\mu=0$ resp.\ $\xi=0$ are now readily deduced.
\end{proof}

Combined with Lemma~\ref{LEM:abstract_subdifferential_closedness}, we infer:

\begin{corollary}[Evolutionary closure of $\subDG\calE_\vek$]\label{COR:evolutionary_epsilon_limit_subdifferential}
	Let $T<\infty$. If $\bfu,\bfu_\vek\in\LL^2(0,T;\calH_m),\bfmu,\bfmu_\vek\in\LL^2(0,T;H^*)$ are such that
	\begin{enumerate}[(1)]
		\item $\bfu_\vek\to\bfu$ in $\LL^2(0,T;\calH_m)$ as $k\to\infty$,
		\item $\bfmu_\vek\weakto\bfmu$ in $\LL^2(0,T;H^*)$ as $k\to\infty$ and
		\item $\bfmu_\vek(t)\in\subDG\calE_\vek(\bfu_\vek(t))$ for almost all $t\in(0,T)$, all $k\in\N$,
	\end{enumerate}
	then $\bfmu(t)\in\subDL\calE_0(\bfu(t))$ for almost all $t\in(0,T)$.
\end{corollary}

\begin{proof}
	Applying Lemma~\ref{LEM:abstract_subdifferential_closedness} to $S_k:=\subDG\calE_\vek$ and noticing that $\subDL\calE_0$ is closed and convex in $H^*$, the claim follows together with Lemma~\ref{LEM:limit_subdifferential}.
\end{proof}

\begin{lemma}\label{LEM:W_epsilon_estimate}
	There exists a constant $C=C(\Om,d,A^*,\tau^*,\beta)\in(0,\infty)$ such that
	\begin{equation}\label{EQ:W_epsilon_estimate}
		\norm{u-m}_{\HH^2(\Om)}\leq C\bigparen{\norm{(u-m,z)}_H+\norm{\bfmu}_{H^*}}
	\end{equation}
	for all $\bfu=(u,z)\in\dom(\subDG\calE_\ve)$, $\bfmu\in\subDG\calE_\ve(\bfu)$, all $\ve\in(0,1)$.
\end{lemma}

\begin{proof}
	As the proof closely parallels that of Lemma~\ref{LEM:W_estimate}, we omit the detailed computations. Let $\bfu=(u,z)\in\dom(\subDG\calE_\ve)$, $\bfmu\in\subDG\calE_\ve(\bfu)$, and multiplying $-\Delta u+f(u)-\frac{1}{\ve^2}A_\ve(u)(z-K_\ve(u))$ with $-\Delta u$, by-part integration and Young's inequality yields as in Lemma~\ref{LEM:W_estimate}
	\begin{align*}
		\norm{\Delta u}_{\LL^2(\Om)}^2&\leq\Bignorm{\nabla\bigparen{-\Delta u+f(u)-\frac{1}{\ve^2}A_\ve(u)(z-K_\ve(u))}}_{\LL^2(\Om;\R^d)}^2\\
		&\quad+(1+2\beta)\norm{u-m}_{\Hav^1(\Om)}^2+(A^*)^2\frac{1}{\ve^4}\norm{z-K_\ve(u)}_{\LL^2(\Om)}^2.
	\end{align*}
	Since $\bfmu=(-\Delta u+f(u)-\frac{1}{\ve^2}A_\ve(u)(z-K_\ve(u)),\frac{1}{\ve^2}(z-K_\ve(u)))$, we easily deduce
	\[\norm{\Delta u}_{\LL^2(\Om)}^2\leq\max\{1,(A^*)^2\}\norm{\bfmu}_{H^*}^2+(1+2\beta)\norm{u-m}_{\Hav^1(\Om)}^2\]
	and conclude as in Lemma~\ref{LEM:W_estimate} to get a constant $C$ independent of $\ve$ as claimed.
\end{proof}

\subsection{Proof of Theorem~\ref{TH:relaxation_limit}}
Let us first outline the proof. Exploiting \eqref{EQ:edi_general_definition}, we first show that $\bfmu_\ve:=-\D_{\bfv}\calR_\ve^{(\gamma,\kappa)}(\bfu_\ve;\dot{\bfu}_\ve)$ and $\dot\bfu_\ve$ are uniformly bounded, and that $(\bfu_\ve)_\ve$ is equicontinuous. In Step~2, we capitalise on these bounds and apply Lemma~\ref{LEM:energies_relative_compact} to deduce \eqref{EQ:relaxation_u_convergence}, \eqref{EQ:relaxation_z_convergence}, and the weak convergence of $(\bfmu_\vek)_k$ to some limit $\bfmu$. In Step~3, we perform the limit passage in \eqref{EQ:gfe_general_definition} by identifying $\bfmu$ using Corollary~\ref{COR:evolutionary_epsilon_limit_subdifferential}. Finally, in the last step, we show that $u_0$ satisfies \eqref{EQ:edb_general_definition}, from which we further deduce the convergence of energies~\eqref{EQ:relaxation_energy_convergence} and strong convergence of the time derivatives~\eqref{EQ:relaxation_time_derivative_convergence}.\\

\emph{Step 1: Uniform bounds and equicontinuity.}\\
	By the energy-dissipation inequality \eqref{EQ:edi_general_definition}, we know that for all $t\in[0,\infty)$
	\begin{equation}\label{EQ:relaxation_edi}
		\calE_\ve(\bfu_\ve(t))+\int_0^t\Bigparen{\calR_\ve^{(\gamma,\kappa)}(\bfu_\ve(s);\dot{\bfu}_\ve(s))+\calR_\ve^{(\gamma,\kappa),*}(\bfu_\ve(s);-\bfmu_\ve(s))}\dd s\leq\calE_\ve(\bfu_\ve^0).
	\end{equation}
	Since $\sup_{\ve\in(0,1)}\calE_\ve(\bfu_\ve^0)<\infty$ by well-preparedness~\eqref{EQ:init_well_prepared}, there exists a constant $C\in(0,\infty)$, independent of $\ve$, satisfying
	\begin{align}\label{EQ:singular_limit_uniform_bounds}
		\sup_{t\in[0,\infty)}\calE_\ve(\bfu_\ve(t))&+\ve^\gamma\norm{\dot u_\ve}_{\LL^2([0,\infty);\Havd)}^2+\ve^\kappa\norm{\dot z_\ve}_{\LL^2([0,\infty);\LL^2(\Om))}^2\notag\\
		&+\ve^{-\gamma}\norm{\mu_\ve}_{\LL^2([0,\infty);\Hav^1(\Om))}^2+\ve^{-\kappa}\norm{\xi_\ve}_{\LL^2([0,\infty);\LL^2(\Om))}^2\leq C.
	\end{align}
	Consequently, by the bounds of the time derivative in \eqref{EQ:singular_limit_uniform_bounds}, the sequences
	\begin{equation}\label{EQ:singular_limit_hoelder_continuity}
		\left\{
        \begin{alignedat}{3}
			&(u_\ve)_{\ve\in(0,1)}&&\text{ in }\Hdav+m&\quad&\text{if }(\gamma,\kappa)=(0,+),\\
			&(\bfu_\ve)_{\ve\in(0,1)}&&\text{ in }\calH_m&&\text{if }(\gamma,\kappa)=(0,0),\\
			&(z_\ve)_{\ve\in(0,1)}&&\text{ in }\LL^2(\Om)&&\text{if }(\gamma,\kappa)=(+,0)
		\end{alignedat}
        \right.
	\end{equation}
	enjoy uniform $\frac{1}{2}$-H\"older continuity in the respective cases, in particular they are equicontinuous.\\

\emph{Step 2: Extraction of subsequences with convergence in the sense \eqref{EQ:relaxation_u_convergence} and \eqref{EQ:relaxation_z_convergence}.}\\
	In case $(\gamma,\kappa)=(0,0)$, Ascoli's theorem (e.g.\ \cite[Lemma 1]{Simon_1986}) together with Lemma \ref{LEM:energies_relative_compact} allow us -- via a diagonal argument -- to extract a subsequence $(\ve_k)_{k\in\N}\subset(0,1)$, $\ve_k\searrow0$ such that $\bfu_\vek\to\bfu_0$ in $\CC_\mathrm{loc}([0,\infty);\calH_m)$ (i.e.\ uniform convergence on any compact interval) for some limit $\bfu_0=(u_0,z_0)\in\CC([0,\infty);\calH_m)$. Moreover, by interpolation
    \begin{equation}\label{EQ:relaxation_interpolation_L2_convergence}
        \norm{u_\vek(t)-u_0(t)}_{\LL^2(\Om)}\leq\underbrace{\norm{u_\vek(t)-u_0(t)}_\Havd^{1/2}}_{\to0}\bigparen{\underbrace{\norm{u_\vek(t)}_{\Hav^1(\Om)}^{1/2}+\norm{u_0(t)}_{\Hav^1(\Om)}^{1/2}}_{\text{bounded by \eqref{EQ:singular_limit_uniform_bounds} and Proposition \ref{PRO:energies_gamma_convergence}}}}
    \end{equation}
	we obtain $u_\vek\to u_0$ in $\CC_\mathrm{loc}([0,\infty);\LL^2(\Om))$, so that \eqref{EQ:relaxation_u_convergence} is established. In order to achieve \eqref{EQ:relaxation_z_convergence}, we need to show $z_0=K(u_0)$. First note that Lemma~\ref{LEM:L2_convergence_with_K} yields $K_\vek(u_\vek)\to K(u_0)$ in $\CC_\mathrm{loc}([0,\infty);\LL^2(\Om))$. The energy bound in \eqref{EQ:singular_limit_uniform_bounds} provides
	\begin{equation}\label{EQ:singular_limit_zKu_L2_convergence}
		\sup_{t\in[0,\infty)}\norm{z_\ve(t)-K_\ve(u_\ve(t))}_{\LL^2(\Om)}\leq\ve\sup_{t\in[0,\infty)}\sqrt{\calE_\ve(\bfu_\ve)}\leq\sqrt{C}\ve
	\end{equation}
	for all $\ve\in(0,1)$, which particularly implies $z_\ve-K_\ve(u_\ve)\to0$ in $\CC([0,\infty);\LL^2(\Om))$. As a consequence, $z_0=K(u_0)$.

    The other cases work similarly: If ``$(0,+)$'', then \eqref{EQ:singular_limit_hoelder_continuity} provides a subsequence with convergence $u_\vek\to u_0$ in $\CC_\mathrm{loc}([0,\infty);\Hdav+m)$, which, as in \eqref{EQ:relaxation_interpolation_L2_convergence}, improves to $u_\vek\to u_0$ in $\CC_\mathrm{loc}([0,\infty);\LL^2(\Om))$. Thus, Lemma~\ref{LEM:L2_convergence_with_K} tells us $K_\vek(u_\vek)\to K(u_0)$ in $\CC_\mathrm{loc}([0,\infty);\LL^2(\Om))$. Using \eqref{EQ:singular_limit_zKu_L2_convergence} again, we infer $z_\vek\to K(u_0)$ in $\CC_\mathrm{loc}([0,\infty);\LL^2(\Om))$.
    In the remaining case ``$(+,0)$'', \eqref{EQ:singular_limit_hoelder_continuity} implies $z_\vek\to z_0$ in $\CC_\mathrm{loc}([0,\infty);\LL^2(\Om))$, and subsequently, \eqref{EQ:singular_limit_zKu_L2_convergence} yields $K_\vek(u_\vek)\to z_0$ in $\CC_\mathrm{loc}([0,\infty);\LL^2(\Om))$. Applying Lemma~\ref{LEM:L2_convergence_with_K} to the inverse functions $K_\ve^{-1}$, it follows $u_\vek\to K^{-1}(z_0)$ in $\CC_\mathrm{loc}([0,\infty);\LL^2(\Om))$.\\

	Hence, we have established $\bfu_\vek\to(u_0,K(u_0))$ in $\CC_\mathrm{loc}([0,\infty);\LL^2(\Om;\R^2))$ in all three cases for some $u_0\in\CC([0,\infty);\LL^2(\Om))$, and the energy bound in \eqref{EQ:singular_limit_uniform_bounds} gives the additional regularity $u_0\in\LL^\infty([0,\infty);\HH^1(\Om))$. The uniform bounds on $\bfmu_\ve$ in \eqref{EQ:singular_limit_uniform_bounds} also ensure $u_0\in\Lloc^2([0,\infty);\HH^2(\Om))$ by Lemma~\ref{LEM:W_epsilon_estimate} and provide some limit $\bfmu=(\mu,\xi)\in\LL^2([0,\infty);H^*)$ such that
	\begin{equation}\label{EQ:relaxation_mu_convergence}
		\begin{cases}
			\bfmu_\vek\weakto(\mu,0)&\text{if }(0,+),\\
			\bfmu_\vek\weakto(\mu,\xi)&\text{if }(0,0),\\
			\bfmu_\vek\weakto(0,\xi)&\text{if }(+,0)
		\end{cases}
	\end{equation}
	in $\LL^2([0,\infty);H^*)$ (the convergences to zero even hold in the strong sense).\\

\emph{Step 3: Derivation of \eqref{EQ:gfe_general_definition} in (CH), (vCH) and (mAC).}\\
	Applying Corollary~\ref{COR:evolutionary_epsilon_limit_subdifferential}, we find that
	\begin{equation}\label{EQ:singular_limit_bfmu_general_limit}
		\bfmu=\begin{pmatrix}
			\mu\\
			\xi
		\end{pmatrix}=\begin{pmatrix}
			-\Delta u_0+f(u_0)-A(u_0)\xi-\tilde{\fraka}(u_0,z_0)\\
			\xi
		\end{pmatrix}.
	\end{equation}
	Thus, it remains to address the convergence of the time derivatives $\dot u_0$ and $\dot z_0$ in the respective cases, and to reformulate the results as gradient-flow equations corresponding to the gradient systems for (CH), (vCH) and (mAC).

	We start with the case ``$(0,+)$''. By \eqref{EQ:singular_limit_uniform_bounds}, $(\dot u_\ve)_\ve$ is uniformly bounded in $\LL^2([0,\infty);\Havd)$. Consequently, $u_0$ possesses a time derivative $\dot u_0\in\LL^2([0,\infty);\Havd)$ with $\dot u_\vek\weakto\dot u_0$ in $\LL^2([0,\infty);\Havd)$, which in turn leads to $(-\Delta)^{-1}\dot u_\vek\weakto(-\Delta)^{-1}\dot u_0$ in $\LL^2([0,\infty);\Hav^1(\Om))$. Combining this with \eqref{EQ:singular_limit_bfmu_general_limit}, $\xi=0$ (see \eqref{EQ:relaxation_mu_convergence}) and the equation $-(-\Delta)^{-1}\dot u_\vek=\mu_\vek$ for $k\in\N$, we arrive at the identity $-(-\Delta)^{-1}\dot u_0=-\Delta u_0+f(u_0)-\avint{_\Om}f(u_0)\dd x$. Equivalently, $-\D\calRCH(\dot u_0)\in\subDF\calECH(u_0)$, i.e.\ $u_0$ is a GF solution of $(\Hdav+m,\calECH,\calRCH)$.

	Next, we treat ``$(0,0)$''. Again, \eqref{EQ:singular_limit_uniform_bounds} provides a uniform bound, now for both components $(\dot u_\ve,\dot z_\ve)_\ve$ in $\LL^2([0,\infty);H)$. Hence, $\bfu_0$ has a time derivative with $\dot\bfu_\vek\weakto\dot\bfu_0$ in $\LL^2([0,\infty);H)$. Since $\tau_\ve\to\tau$ uniformly and $u_\vek\to u_0$ in $\CC_\mathrm{loc}([0,\infty);\LL^2(\Om))$, it follows $\tau_\vek(u_\vek)\dot z_\vek\weakto\tau(u_0)\dot z_0$ in $\LL^2([0,\infty);\LL^2(\Om))$. Thus, $((-\Delta)^{-1}\dot u_\vek,\tau_\vek(u_\vek)\dot z_\vek)\weakto((-\Delta)^{-1}\dot u_0,\tau(u_0)\dot z_0)$ in $\LL^2([0,\infty);H^*)$. Together with \eqref{EQ:relaxation_mu_convergence}, \eqref{EQ:singular_limit_bfmu_general_limit} and $\bfmu_\vek=((-\Delta)^{-1}\dot u_\vek,\tau_\vek(u_\vek)\dot z_\vek)$, we obtain
	\begin{equation}\label{EQ:relaxation_gfe_passage_vCH}
		\begin{pmatrix}
			-(-\Delta)^{-1}\dot u_0\\
			-\tau(u_0)\dot z_0
		\end{pmatrix}=\begin{pmatrix}
			-\Delta u_0+f(u_0)-A(u_0)\xi-\avint{_\Om}(f(u_0)-A(u_0)\xi)\dd x\\
			\xi
		\end{pmatrix}.
	\end{equation}
	In order to rearrange \eqref{EQ:relaxation_gfe_passage_vCH} to $-\D_{\bfv}\calRvCH(u_0;\dot u_0)\in\subDF\calECH(u_0)$, we observe that $\pa_t(K\circ u_0)$ can be expressed in two ways due to Lemma~\ref{LEM:sobolev_bochner_chain_rule}, namely
	\begin{equation}\label{EQ:time_derivative_comparison}
		(\dot z_0(t),\phi)_{\LL^2(\Om)}=\product{\pa_t(K\circ u_0)(t)}{\phi}_{(\HH^1(\Om))^*,\HH^1(\Om)}=\product{\dot u_0}{A(u_0(t))\phi}_{(\HH^1(\Om))^*,\HH^1(\Om)}
	\end{equation}
	for $\phi\in\HH^1(\Om)$ and almost every $t\in[0,\infty)$, where the right-hand side is to be understood through the identification $\Havd\simeq\Hdav$ in \eqref{EQ:dual_average_isoiso}. Since $\HH^1(\Om)\ni\phi\mapsto A(u_0(t))\phi\in\HH^1(\Om)$ is a bijective assignment whenever $u_0(t)\in\HH^2(\Om)$, which holds for almost every $t\in(0,\infty)$, the comparison \eqref{EQ:time_derivative_comparison} yields $\dot u_0=\frac{1}{A(u_0)}\dot z_0\in\LL^2([0,\infty);\Lav^2(\Om))$. Inserting $\xi=-\tau(u_0)\dot z_0=-\tau(u_0)A(u_0)\dot u_0$ into the upper equation of \eqref{EQ:relaxation_gfe_passage_vCH}, we arrive at
	\[-\biggparen{(-\Delta)^{-1}+A^2(u_0)\tau(u_0)-\avint{_\Om}A^2(u_0)\tau(u_0)(\cdot)\dd x}\dot u_0=-\Delta u_0+f(u_0)-\avint{_\Om}f(u_0)\dd x.\]
	Hence, $u_0$ is a GF solution to $(\Lav^2(\Om)+m,\calECH,\calRvCH)$.

	Eventually, consider ``$(+,0)$''. In this case, \eqref{EQ:singular_limit_uniform_bounds} gives a uniform bound for $(\dot z_\ve)_\ve$, so that $\dot z_\vek\weakto\dot z_0$ in $\LL^2([0,\infty);\LL^2(\Om))$. As before, we derive $\tau_\vek(u_\vek)\dot z_\vek\weakto\tau(u_0)\dot z_0$ in $\LL^2([0,\infty);\LL^2(\Om))$. Combining with \eqref{EQ:relaxation_mu_convergence}, \eqref{EQ:singular_limit_bfmu_general_limit} and $\tau_\vek(u_\vek)\dot z_\vek=\xi_\vek$, we deduce
    \begin{equation}\label{EQ:relaxation_gfe_passage_mAC}
        -A(u_0)\tau(u_0)\dot z_0=-\Delta u_0+f(u_0)-\avint{_\Om}(f(u_0)-A(u_0)\tau(u_0)\dot z_0)\dd x.
    \end{equation}
    By an approximation argument similar to that in the proof of the chain rule Lemma~\ref{LEM:sobolev_bochner_chain_rule}, we deduce that $\dot u_0$ exists in $\LL^2([0,\infty);\LL^2(\Om))$, namely
	\[(\dot u_0(t),\phi)_{\LL^2(\Om)}=\product{\pa_t(K^{-1}\circ z_0)(t)}{\phi}_{(\HH^1(\Om))^*,\HH^1(\Om)}=\Bigparen{\dot z_0,\frac{1}{A(u_0(t))}\phi}_{\LL^2(\Om)}.\]
	In particular, as $\avint{_\Om}u_0(t)\dd x=m$, we even have $\dot u_0=\frac{1}{A(u_0)}\dot z_0\in\LL^2([0,\infty);\Lav^2(\Om))$. Inserting this into \eqref{EQ:relaxation_gfe_passage_mAC}, we end up with
	\[-\biggparen{A^2(u_0)\tau(u_0)\dot u_0-\avint{_\Om}A^2(u_0)\tau(u_0)(\cdot)\dd x}\dot u_0=-\Delta u_0+f(u_0)-\avint{_\Om}f(u_0)\dd x.\]
	Hence, $u_0$ is a GF solution of $(\Lav^2(\Om)+m,\calECH,\calRmAC)$.\\

    To conclude this step, note that $u_\vek^0\to u_0^0$ in $\LL^2(\Om)$ (by well-preparedness~\eqref{EQ:init_well_prepared} and interpolation) and $u_\vek(0)\to u_0(0)$ by \eqref{EQ:relaxation_u_convergence}. Thus, the initial condition $u_0(0)=u^0$ is satisfied.\\

\emph{Step 4: \eqref{EQ:edb_general_definition}, \eqref{EQ:relaxation_energy_convergence} and \eqref{EQ:relaxation_time_derivative_convergence}.}\\
	Since $\calECH$ is semiconvex with respect to $\Havd$ as well as $\LL^2(\Om)$, the chain rule from \cite[$(2.\mathrm{E}_4)$ and Remark 2.5]{Mielke_Rossi_Savare_2013} can be applied in all three cases, so that the solution $u_0$ satisfies \eqref{EQ:edb_general_definition}. To derive \eqref{EQ:relaxation_energy_convergence} and \eqref{EQ:relaxation_time_derivative_convergence}, we will use the particular form of \eqref{EQ:edb_general_definition} (cf.\ \eqref{EQ:fenchel_equivalences})
	\begin{equation}\label{EQ:relaxation_edb_particular_form}
		\calECH(u_0(t))+\int_s^t2\calR(u_0(r);\dot u_0(r))\dd r=\calECH(u_0(s))
	\end{equation}
	for $0\leq s\leq t<\infty$ and $\calR\in\{\calRCH,\calRvCH,\calRmAC\}$, depending on $(\gamma,\kappa)$.

	We start with the case $(\gamma,\kappa)=(0,0)$. By fixing $t\in(0,\infty)$ and passing to the limit in \eqref{EQ:relaxation_edi}, and in light of the well-preparedness \eqref{EQ:init_well_prepared} and \eqref{EQ:relaxation_edb_particular_form}, we obtain
	\begin{align*}
		&\limsup_{k\to\infty}\Bigparen{\calE_\vek(\bfu_\vek(t))+\norm{\dot u_\vek}_{\LL^2(0,t;\Havd)}^2+\bignorm{{\textstyle\sqrt{\tau_\vek(u_\vek)}}\dot z_\vek}_{\LL^2(0,t;\LL^2(\Om))}^2}\\
		&\quad=\limsup_{k\to\infty}\Bigparen{\calE_\vek(\bfu_\vek(t))+\int_0^t\bigparen{2\calR_\vek^{(\gamma,\kappa)}(\bfu_\vek(s);\dot\bfu_\vek(s))}\dd s}\\
		&\quad\leq\limsup_{k\to\infty}\calE_\vek(\bfu_\vek^0)=\calECH(u_0^0)\\
		&\quad=\calECH(u_0(t))+\norm{\dot u_0}_{\LL^2(0,t;\Havd)}^2+\norm{A(u_0)\sqrt{\tau(u_0)}\dot u_0}_{\LL^2(0,t;\LL^2(\Om))}^2.
	\end{align*}
	From there we obtain, by weak lower semicontinuity of norms (applied to $\dot u_\vek\weakto\dot u_0$ and $\sqrt{\tau_\vek(u_\vek)}\dot z_\vek\weakto A(u_0)\sqrt{\tau(u_0)}\dot u_0$) and the liminf-estimate from Proposition \ref{PRO:energies_gamma_convergence},
	\[\begin{cases}
		\limsup_{k\to\infty}\norm{\dot u_\vek}_{\LL^2(0,t;\Havd)}\leq\norm{\dot u_0}_{\LL^2(0,t;\Havd)},\\
		\limsup_{k\to\infty}\norm{\sqrt{\tau_\vek(u_\vek)}\dot z_\vek}_{\LL^2(0,t;\LL^2(\Om))}\leq\norm{A(u_0)\sqrt{\tau(u_0)}\dot u_0}_{\LL^2(0,t;\LL^2(\Om))},\\
		\limsup_{k\to\infty}\calE_\vek(\bfu_\vek(t))\leq\calECH(u_0(t)).
	\end{cases}\]
	Hence, $\dot u_\vek\to\dot u_0$ in $\LL^2(0,t;\Havd)$, $\sqrt{\tau_\vek(u_\vek)}\dot z_\vek\to A(u_0)\sqrt{\tau(u_0)}\dot u_0$ in $\LL^2(0,t;\LL^2(\Om))$ and $\calE_\vek(\bfu_\vek(t))\to\calECH(u_0(t))$. Using \eqref{EQ:relaxation_u_convergence} and that
    $1/\abs{\tau_\vek}\leq1/\tau_*$, we get $\dot z_\vek\to A(u_0)\dot u_0$ in $\LL^2(0,t;\LL^2(\Om))$.

	To deal with the other cases $(\gamma,\kappa)=(0,+)$ and $(+,0)$, we note that
	\[\max\Bigbraces{\norm{\dot u_\vek(t)}_\Havd^2,\bignorm{{\textstyle\sqrt{\tau_\vek(u_\vek(t))}}\dot z_\vek(t)}_{\LL^2(\Om)}^2}\leq2\calR_\vek^{(\gamma,\kappa)}(\bfu_\vek(t);\dot\bfu_\vek(t)),\]
	and similarly derive
	\begin{multline*}
		\begin{rcases}
			\limsup_{k\to\infty}\bigparen{\calE_\vek(\bfu_\vek(t))+\norm{\dot u_\vek}_{\LL^2(0,t;\Havd)}^2}&\text{if }(0,+),\\
			\limsup_{k\to\infty}\bigparen{\calE_\vek(\bfu_\vek(t))+\norm{\sqrt{\tau_\vek(u_\vek)}\dot z_\vek}_{\LL^2(0,t;\LL^2(\Om))}^2}&\text{if }(+,0)
		\end{rcases}\\
		\leq\begin{cases}
			\calECH(u_0(t))+\norm{\dot u_0}_{\LL^2(0,t;\Havd)}^2&\text{if }(0,+),\\
			\calECH(u_0(t))+\norm{A_0(u)\sqrt{\tau(u_0)}\dot u_0}_{\LL^2(0,t;\LL^2(\Om)}^2&\text{if }(+,0).
		\end{cases}
	\end{multline*}
	The arguments then proceed as before.\\

The proof of Theorem~\ref{TH:relaxation_limit} is now complete.

\appendix

\section{Appendix}
\renewcommand\theequation{A.\arabic{equation}}
\subsection{Approximation of Lebesgue integral by Riemann sums}
In the proof of the chain rule, Proposition \ref{PRO:chain_rule}, we approximated the Lebesgue integral by Riemann sums. For this argument to be valid, we need to ensure that the Riemann sums indeed converge to the Lebesgue integral. The specific situation needed in Proposition~\ref{PRO:chain_rule} was recently treated in~\cite{Caillet_Santambrogio_2024}, where it was shown, in particular, that the partition points can be chosen outside a set of Lebesgue measure zero.

\begin{lemma}[\texorpdfstring{\cite[Lemma A.1]{Caillet_Santambrogio_2024}}{}]\label{LEM:lebesgue_riemann_approximation}
	Let $a,b\in\R$, $a<b$, and $\Sigma\subseteq[a,b]$ such that $a,b\in\Sigma$ and $\Leb^1([a,b]\setminus\Sigma)=0$. Moreover, let $g:[a,b]\to[0,\infty]$ be Lebesgue measurable with $\int_a^bg(t)\dd t<\infty$. Fix $\sigma\in\{0,1\}$ and $\ve\in(0,1)$. Then, there exists a partition $a=t_0<t_1<\dotsc<t_n=b$ such that $t_i\in\Sigma$ for all $0\leq i\leq n$, $\max_{1\leq i\leq n}\abs{t_i-t_{i-1}}<\ve$ and
	\[\biggabs{\sum_{i=1}^ng(t_{i-\sigma})(t_i-t_{i-1})-\int_a^bg(t)\dd t}<\ve.\]
\end{lemma}

\subsection{Evolutionary closedness}
In the existence and relaxation limit proofs, we employed the evolutionary closedness properties formulated in Corollaries~\ref{COR:evolutionary_closedness_subdifferential_energy}~and~\ref{COR:evolutionary_epsilon_limit_subdifferential}. In both cases, we proved a static strong-weak closedness in advance, namely Proposition~\ref{PRO:closedness_subdifferential_energy} and Lemma~\ref{LEM:limit_subdifferential}. The following lemma, which is based on the methods in~\cite{Rossi_Savare_2006}, builds the crucial link between the static and evolutionary properties. There, the $S_n$'s are the placeholders for the subdifferentials. More precisely, in the existence part, we apply it to $S_n=\subDG\calE$, so that $S_n$ is in fact independent of $n$. In the relaxation limit, the situation differs: here, $S_n=\subDG\calE_{\ve_n}^\alpha$, i.e.\ it genuinely depends on $n$, and the limit operator is $S_\infty=\subDL\calE_0$.

\begin{lemma}\label{LEM:abstract_subdifferential_closedness}
	Let $H$ be a real, separable Hilbert space, and let $S_n:H\to\calP(H^*)$ be mappings for $n\in\N\cup\{\infty\}$ from $H$ into the power set of $H^*$. We assume that $S_\infty(x)$ is closed in $H^*$ and convex for all $x\in H$. Moreover, assume \eqref{EQ:abstract_static_closedness} holds for all subsequences $(n_k)_{k\in\N}\subseteq(n)_{n\in\N}$:
	\begin{equation}\label{EQ:abstract_static_closedness}
		\left.\begin{array}{r}
			(x_{n_k})_{k\in\N}\subset H,y_{n_k}\in S_{n_k}(x_{n_k}),x_\infty\in H,y_\infty\in H^*\\
			x_{n_k}\to x_\infty\text{ in }H,y_{n_k}\weakto y_\infty\text{ in }H^*\text{ as }k\to\infty
		\end{array}\right\}\quad\Longrightarrow\quad y_\infty\in S_\infty(x_\infty).
	\end{equation}
	Then, we have the following evolutionary closedness for any $T\in(0,\infty)$:
	\begin{multline}\label{EQ:abstract_evolutionary_closedness}
		\left.\begin{array}{r}
			x_n,x_\infty\in\LL^2(0,T;H),y_n,y_\infty\in\LL^2(0,T;H^*)\\
			x_n\to x_\infty\text{ in }\LL^2(0,T;H)\text{ for }n\to\infty,\\
			y_n\weakto y_\infty\text{ in }\LL^2(0,T;H^*)\text{ for }n\to\infty,\\
			y_n(t)\in S_n(x_n(t))\text{ for a.e. }t\in(0,T),n\in\N
		\end{array}\right\}\\
		\Longrightarrow\quad y_\infty(t)\in S_\infty(x_\infty(t))\text{ a.e. }t\in(0,T).
	\end{multline}
\end{lemma}

\begin{proof}
	The proof mostly follows \cite[Proposition 2.7]{Liero_Reichelt_2018}. Assume the left-hand side of \eqref{EQ:abstract_evolutionary_closedness}. From $x_n\to x_\infty$ strongly and the condition $y_n\in S_n(x_n)$, we particularly find a subsequence $(n_k)_{k\in\N}\subseteq(n)_{n\in\N}$ and a null set $N\subset[0,T]$ such that, for all $t\in[0,T]\setminus N$, it holds $x_{n_k}(t)\to x_\infty(t)$ in $H$ as $k\to\infty$ and $y_{n_k}(t)\in S_{n_k}(x_{n_k}(t))$ for all $k\in\N$.

	As a weakly convergent sequence, $(y_{n_k})_{k\in\N}$ is bounded in $\LL^2(0,T;H^*)$. Applying the fundamental theorem for weak topologies \cite[Theorem 3.2]{Rossi_Savare_2006}, we find a time-parametrised Young measure $(\nu_t)_{t\in[0,T]}$ such that $(\nu(t))(L(t))=1$, where
	\[L(t):=\Bigset{y\in H^*}{\exists(y_{n_{k_\ell}})_{\ell\in\N}\subseteq(y_{n_k})_{k\in\N}\text{ with }y_{n_{k_\ell}}\weakto y\text{ in }H^*\text{ as }\ell\to\infty},\]
	and
	\begin{equation}\label{EQ:limit_young_measure_representation}
		y_\infty(t)=\int_{H^*}z\dd\nu_t(z)
	\end{equation}
	for all $t\in[0,T]\setminus N$ (by making $N$ possibly larger, but still $\Leb^1(N)=0$).

	Taking $y\in L(t)$, we find a subsequence such that $y_{n_{k_\ell}}\weakto y$ in $H^*$. As shown above, we also have $x_{n_{k_\ell}}(t)\to x(t)$ in $H$ and $y_{n_{k_\ell}}(t)\in S_{n_{k_\ell}}(x_{n_{k_\ell}}(t))$ for all $\ell\in\N$. Therefore, \eqref{EQ:abstract_static_closedness} implies $y\in S_\infty(x_\infty(t))$. Hence, $L(t)\subseteq S_\infty(x_\infty(t))$. This shows that for all $t\in[0,T]\setminus N$, the measure $\nu(t)$ is concentrated on the convex and closed set $S_\infty(x_\infty(t))$, so we infer $y_\infty(t)\in S_\infty(x_\infty(t))$ by \eqref{EQ:limit_young_measure_representation}.
\end{proof}

\subsection{Auxiliary Sobolev--Bochner chain rule}

\begin{lemma}\label{LEM:sobolev_bochner_chain_rule}
	Let $T<\infty$, and $\Om\subset\R^d$, $1\leq d\leq 3$, be open and bounded. Assume $u\in\HH^1(0,T;(\HH^1(\Om))^*)\cap\LL^2(0,T;\HH^2(\Om))$ and $A\in\WW^{1,\infty}(\R)$, and let $K:\R\to\R$ be the primitive of $A$ with $K(0)=0$. Then $K\circ u\in\WW^{1,\frac{4}{3}}(0,T;(\HH^1(\Om))^*)\cap\LL^2(0,T;\HH^2(\Om))$, with distributional time derivative
	\begin{equation}\label{EQ:appendix_chain_rule}
		\product{\pa_t(K\circ u)(t)}{\phi}_{(\HH^1(\Om))^*,\HH^1(\Om)}=\product{\pa_tu(t)}{A(u(t))\phi}_{(\HH^1(\Om))^*,\HH^1(\Om)}
	\end{equation}
	for all $\phi\in\HH^1(\Om)$, almost all $t\in[0,T]$.
\end{lemma}

\begin{proof}
	Since $A$ and $K$ are Lipschitz continuous, we have $K(u)\in\LL^2(0,T;\HH^2(\Om))$ with $\nabla(K(u))=A(u)\nabla u$ and $\Hess(K(u))=A'(u)\nabla u\otimes\nabla u+A(u)\Hess(u)$ by product rule in Sobolev spaces and the chain rule, e.g.\ \cite[Lemma 2.1]{Marcus_Mizel_1972}.

	Next, we show that the right-hand side in \eqref{EQ:appendix_chain_rule} has the regularity $\WW^{1,\frac{4}{3}}(0,T;(\HH^1(\Om))^*)$. To this end, we let $\phi\in\HH^1(\Om)$ and estimate for almost all $t\in(0,T)$
	\[\ell(t)(\phi):=\product{\pa_tu(t)}{A(u(t))\phi}_{(\HH^1(\Om))^*,\HH^1(\Om)}\leq\norm{\pa_tu(t)}_{(\HH^1(\Om))^*}\norm{A(u)\phi}_{\HH^1(\Om)}.\]
	Since $\nabla(A(u(t))\phi)=A(u(t))\nabla\phi+A'(u(t))\phi\nabla u(t)$ in $\LL^2(\Om;\R^d)$, the Sobolev embedding $\HH^1(\Om)\hookrightarrow\LL^6(\Om)$ entails
	\[\norm{A(u(t))\phi}_{\HH^1(\Om)}\leq\bigparen{\sqrt{2}\norm{A}_\infty+\sqrt{2}\CSo\norm{A'}_\infty\norm{\nabla u(t)}_{\LL^3(\Om;\R^d)}}\norm{\phi}_{\HH^1(\Om)}.\]
	Consequently, H\"older's inequality applied to the exponents $\frac{3}{2}$, $3$ leads to
	\begin{align*}
		\int_0^T\norm{\ell(t)}_{(\HH^1(\Om))^*}^{4/3}\dd t&\leq C\int_0^T\norm{\pa_tu}_{(\HH^1(\Om))^*}^{4/3}\bigparen{1+\norm{\nabla u}_{\LL^3(\Om;\R^d)}}^{4/3}\dd t\\
		&\leq C\biggparen{\int_0^T\norm{\pa_tu}_{(\HH^1(\Om))^*}^2\dd t}^{2/3}\biggparen{\int_0^T\abs{1+\norm{\nabla u}_{\LL^3(\Om;\R^d)}}^4\dd t}^{1/3},
	\end{align*}
	which is finite by interpolation $\norm{\cdot}_{\LL^4(0,T;\LL^3(\Om;\R^d))}\leq\norm{\cdot}_{\LL^\infty(0,T;\LL^2(\Om;\R^d))}^{1/2}\norm{\cdot}_{\LL^2(0,T;\HH^1(\Om;\R^d))}^{1/2}$.

	It remains to show that for any $\phi\in\HH^1(\Om)$ and $\eta\in\CC_c^\infty((0,T))$
	\[\int_0^T\eta'(t)(K(u(t)),\phi)_{\LL^2(\Om)}\dd t=-\int_0^T\eta(t)\product{\pa_tu(t)}{A(u(t))\phi}_{(\HH^1(\Om))^*,\HH^1(\Om)}\dd t.\]
	By approximation \cite[Lemma 7.2]{Roubicek_NonlinearPDEBook_2013}, we find a sequence $(u_n)_{n\in\N}\subset\CC^1([0,T];\HH^2(\Om))$ such that $u_n\to u$ in $\LL^2(0,T;\HH^2(\Om))$ and $\pa_tu_n\to\pa_tu$ in $\LL^2(0,T;(\HH^1(\Om))^*)$. Exploiting the Lipschitz continuity of $A$, one finds $A(u_n)\eta\phi\weakto A(u)\eta\phi$ in $\LL^2(0,T;\HH^1(\Om))$, so that
	\begin{align*}
		\int_0^T\eta'(t)(K(u(t)),\phi)_{\LL^2(\Om)}\dd t&=\lim_{n\to\infty}\int_0^T\eta'(t)(K(u_n(t)),\phi)_{\LL^2(\Om)}\dd t\\
		&=-\lim_{n\to\infty}\int_0^T\product{\pa_t u_n(t)}{A(u_n(t))\eta(t)\phi}_{(\HH^1(\Om))^*,\HH^1(\Om)}\dd t\\
		&=-\int_0^T\eta(t)\product{\pa_t u(t)}{A(u(t))\phi}_{(\HH^1(\Om))^*,\HH^1(\Om)}\dd t.\qedhere
	\end{align*}
\end{proof}

\section*{Acknowledgments}
The authors thank Matthias Liero and Alexander Mielke for contributing their expertise through several valuable discussions on the analysis of gradient systems.

% \emergencystretch=1em
\printbibliography[heading=bibintoc]
\end{document}